\documentclass{article}
\usepackage{graphicx} 

\usepackage[utf8]{inputenc}

\usepackage[a4paper,top=2cm,bottom=2cm,left=2cm,right=2cm,marginparwidth=1cm]{geometry}

\usepackage{upgreek}
\usepackage{amsmath,amssymb,amsthm,amsfonts}
\usepackage{mathtools}
\usepackage{dsfont}
\usepackage{graphicx}
\usepackage{subcaption}
\usepackage{float}
\usepackage{bm}
\usepackage{enumitem}
\usepackage{pifont}
\usepackage{crossreftools}
\usepackage{comment}

\usepackage[colorinlistoftodos]{todonotes}
\usepackage[allcolors=black]{hyperref}

\usepackage[square,numbers]{natbib}
\newcommand{\PP}{{\mathbb{P}}}
\newcommand{\QQ}{{\mathbb{Q}}}

\newcommand{\PPpmzero}[1]{{\mathbb{P}_{#1}^{\pm}}}

\newcommand{\EE}{{\mathbb{E}}}
\newcommand{\EEpm}[1]{{\mathbb{E}_{#1}^{\pm}}}
\newcommand{\EEpmtime}[2]{{\mathbb{E}_{#1,#2}^{\pm}}}
\newcommand{\EEplus}[1]{{\mathbb{E}_{#1}^{+}}}

\newcommand{\eps}{{\epsilon}}
\newcommand{\HH}{{\mathbb{H}}}

\newcommand{\CR}{{\operatorname{CR}}}

\newcommand{\diam}{{\operatorname{diam}}}
\newcommand{\vol}{{\operatorname{vol}}}
\newcommand{\der}{{\operatorname{d}}}
\newcommand{\Int}{{\operatorname{Int}}}

\newtheorem{theorem}{Theorem}[section]

\newtheorem{lemma}[theorem]{Lemma}

\newtheorem{proposition}[theorem]{Proposition}
\newtheorem{claim}[theorem]{Claim}
\theoremstyle{definition}

\theoremstyle{remark}
\newtheorem{remark}[theorem]{Remark}
\theoremstyle{remark}

\newcommand{\optionaldesc}[2]{%
  \phantomsection
  #1\protected@edef\@currentlabel{#1}\label{#2}%
}

\title{Interface scaling limit for the critical planar Ising model perturbed by a magnetic field}
\author{L\'eonie Papon \thanks{Durham University}}
\date{\today}

\begin{document}

\maketitle

\begin{abstract}
We prove that the interface separating $+1$ and $-1$ spins in the critical planar Ising model with Dobrushin boundary conditions perturbed by an external magnetic field has a scaling limit. This result holds when the Ising model is defined on a bounded and simply connected subgraph of $\delta \mathbb{Z}^2$, with $\delta >0$. We show that if the scaling of the external field is of order $\delta^{15/8}$, then, as $\delta \to 0$, the interface converges in law to a random curve whose law is conformally covariant and absolutely continuous with respect to SLE$_3$. This limiting law is a massive version of SLE$_3$ in the sense of Makarov and Smirnov and we give an explicit expression for its Radon-Nikodym derivative with respect to SLE$_3$. We also prove that if the scaling of the external field is of order $\delta^{15/8}g_1(\delta)$ with $g_1(\delta)\to 0$, then the interface converges in law to SLE$_3$. In contrast, we show that if the scaling of the external field is of order $\delta^{15/8}g_2(\delta)$ with $g_2(\delta) \to \infty$, then the interface degenerates to a boundary arc.
\end{abstract}

\section{Introduction}

\subsection{Main results}

Chordal SLE$_\kappa$ curves are a one-parameter family of planar curves, characterized by conformal invariance and a Markovian property, that were introduced by Schramm \cite{Schramm_SLE}. They have been shown to arise as the scaling limits of interfaces in many planar statistical mechanics models at criticality when the boundary conditions are chosen appropriately \cite{percolation, LERW, HE, DGFF, cvgIsingSLE}.

However, many interesting questions arise when looking at near-critical perturbations of these models, which are obtained by sending some of their parameters to their critical values at an appropriate rate. In the scaling limit, these perturbations introduce a finite correlation length and break conformal invariance: observables of the model that are conformally invariant in the critical regime are often only conformally covariant in the near-critical regime. In \cite{massiveSLE}, Makarov and Smirnov asked whether SLE-type curves could nevertheless describe the scaling limits of interfaces in the near-critical regime. The idea is that, on a lattice, the discrete interface of the near-critical model can be seen as a perturbation of the interface of the critical one. In the scaling limit, one can therefore expect the near-critical limiting interface to be in some sense a perturbation of an SLE$_{\kappa}$ curve. This near-critical limiting interface would depend on a mass parameter, arising from the near-critical perturbation, which led Makarov and Smirnov to call the (conjectural) limiting laws of near-critical interfaces massive SLEs \cite{massiveSLE}. In particular, these massive SLE should still enjoy a Markovian property similar to that of SLE$_{\kappa}$ and the perturbation would only affect their transformation law under conformal maps: the mass parameter would have to be changed when mapping one domain to another via a conformal map and as a result, massive SLE should be conformally covariant in law, instead of conformally invariant as SLE.

When $\kappa$ is fixed, SLE$_{\kappa}$ can often be characterized by a martingale observable and this characterization is a useful tool to prove scaling-limit-type results for models at criticality. In contrast, in the near-critical regime, less is known about the properties that should characterize the law of the limiting interface. Nevertheless, convergence of near-critical interfaces has been established in a few cases. For example, massive loop-erased random walk has been shown to converge to massive SLE$_2$ \cite{mLERW,off_dimers} and the massive harmonic explorer converges to massive SLE$_4$ \cite{mHE}, which is also the level line of the massive GFF with appropriate boundary conditions \cite{mSLE4-mGFF}. Some progress has also been made towards establishing convergence of the interface of the near-critical FK-Ising model \cite{massive_FK} and of near-critical percolation \cite{massiveSLE-6, near_percolation}, although in these two cases proving uniqueness of the limiting law remains a challenge.

In this paper, we look at another example of interface in a near-critical regime: the interface separating $+1$ and $-1$ spins in the Ising model at critical temperature with Dobrushin boundary conditions perturbed by an external magnetic field. When there is no external field, this interface is known to converge to an SLE$_3$ curve \cite{cvgIsingSLE}. According to \cite{massiveSLE}, in the presence of an external field, one should expect the limiting interface to be a massive version of SLE$_3$, in the sense that its law is obtained by "perturbing" the law of SLE$_3$. Our main theorem makes this precise. We will give an exact definition of the Ising model perturbed by an external magnetic field and of its interface below but for now, let us simply state the result.

\begin{theorem} \label{theorem_magnetic_SLE3}
    Let $h: \Omega \to \mathbb{R}_{+}$ be a bounded and Lipschitz function. Let $\Omega \subset \mathbb{C}$ be a bounded and simply connected domain with two marked boundary points $a,b \in \partial \Omega$. Let $(\Omega_{\delta};a_{\delta},b_{\delta})_{\delta}$ be discrete approximations of $(\Omega;a,b)$ where for each $\delta > 0$, $\Omega_{\delta}$ is a subgraph of $\delta\mathbb{Z}^2$. Consider the interface $\gamma_{\delta}$ of the critical Ising model in $\Omega_{\delta}$ with Dobrushin boundary conditions and external magnetic field $H(x,\delta)=C_{\sigma}^{-1}h(x)\delta^{15/8}$, where $C_{\sigma}>0$ is a constant that will be defined below and corresponds to the fact that we are working on the square lattice. Then, as $\delta \to 0$, $\gamma_{\delta}$ converges in law to a massive version of SLE$_3$ in $\Omega$ from $a$ to $b$, with law denoted by $\PP_{h}^{(\Omega,a,b)}$. 
    
    The Radon-Nikodym derivative of $\PP_{h}^{(\Omega,a,b)}$ with respect to the law $\PP_{\operatorname{SLE}_3}^{(\Omega,a,b)}$ of SLE$_3$ in $\Omega$ from $a$ to $b$ is given by, for any $t \geq 0$,
    \begin{equation} \label{RN_time_theorem_intro}
        \left.\frac{\der \PP_{h}^{(\Omega,a,b)}}{\der \PP_{\operatorname{SLE}_3}^{(\Omega,a,b)}}(\gamma)\right|_{ \sigma(\gamma(s): 0\leq s \leq t)} = \frac{1}{\mathcal{Z}_h(\Omega)} \sum_{k \geq 0} \frac{1}{k!}\int_{\Omega_t^k} h(z_1)\dots h(z_k)f_{t}^{(k)}(z_1,\dots , z_k) \prod_{j=1}^{k} dz_j
    \end{equation}
    where $\mathcal{Z}_{h}(\Omega)$ is a normalization constant. Above, we have set $\Omega_t = \Omega \setminus \gamma([0,t])$ and the functions $f_t^{(k)}: \Omega_t^k \to \mathbb{R}$ are explicit.
\end{theorem}

The precise assumptions on $(\Omega_\delta;a_\delta,b_\delta)_\delta$ and $(\Omega;a,b)$ together with the topologies in which weak convergence is obtained will be detailed in Section \ref{sec_setting}. Above, the constant $C_{\sigma}$ and the functions $(f_{t}^{(k)})_{k \geq 1}$ are such that the $k$-point spin correlations of the critical Ising model in the discrete slit domain $\Omega_{\delta,t}$ with Dobrushin boundary conditions converge to $C_{\sigma}^kf_t^{(k)}$ as $\delta \to 0$, see \cite{spinCorrelations} and Section \ref{section_spin_correlations}. Rescaling the external field by $C_{\sigma}^{-1}$ guarantees that the limiting law $\PP_{h}^{(\Omega,a,b)}$ does not depend on the fact that the sequence $(\gamma_{\delta})_{\delta}$ is defined on the square lattice. In this way, one can expect $\PP_{h}^{(\Omega,a,b)}$ to be universal. Note that rescaling the strength of the perturbation by a lattice-dependent constant is common in the near-critical regime, see for example \cite{mHE}.

Since $\PP_{h}^{(\Omega,a,b)}$ is obtained as the scaling limit of an interface in a near-critical model, it is not conformally invariant, but conformally covariant, as shown below. Not surprisingly given the nature of the near-critical perturbation, the behavior of $\PP_{h}^{(\Omega,a,b)}$ under conformal maps is the same as that of the continuum magnetization field of the critical Ising model constructed in \cite{GarbanMagnetization}.

\begin{proposition} \label{prop_conformal_covariance}
    The law $\PP_{h}^{(\Omega,a,b)}$ is conformally covariant in the following sense. If $\gamma$ has law $\PP_{h}^{(\Omega,a,b)}$ and $\psi: \Omega \to \tilde \Omega$ is a conformal map, then $\psi(\gamma)$ has law $\PP_{\tilde h}^{(\tilde \Omega,\psi(a),\psi(b))}$ where $\tilde h$ is given by, for $w \in \tilde \Omega$,
    \begin{equation*}
        \tilde h(w) = \vert (\psi^{-1})'(w) \vert^{\frac{15}{8}} h(\psi^{-1}(w)).
    \end{equation*}
\end{proposition}

Theorem \ref{theorem_magnetic_SLE3} describes the scaling limit of the interface when the scaling of the external magnetic field is the near-critical one. A natural question that then arises is to give a description of the scaling limit of the interface when the scaling of the magnetic field is not the near-critical one. The first case to consider is when the scaling of the field goes to $0$ faster than $\delta^{15/8}$ as $\delta \to 0$. In this case, we show that the scaling limit of the interface is the same as in the critical regime. The topologies in which this convergence holds are the same as in Theorem \ref{theorem_magnetic_SLE3} and will be described in Section \ref{sec_setting}.

\begin{proposition} \label{prop_small_H}
    Let $g_1:\mathbb{R}_{+} \to \mathbb{R}_{+}$ be such that $g_1(\delta) \to 0$ as $\delta \to 0$. Assume that the function $h$, the domain $(\Omega;a,b)$ and its approximations $(\Omega_{\delta};a_{\delta},b_{\delta})_{\delta}$ are as in Theorem \ref{theorem_magnetic_SLE3}. Consider the interface $\gamma_{\delta}$ of the critical Ising model in $\Omega_{\delta}$ with Dobrushin boundary conditions and external magnetic field $H(x,\delta) = h(x)\delta^{15/8}g_{1}(\delta)$. Then, as $\delta \to 0$, $\gamma_{\delta}$ converges in law to an $\operatorname{SLE}_3$ in $\Omega$ from $a$ to $b$. 
\end{proposition}

The second case to consider is that of an external magnetic field whose strength is larger than $\delta^{15/8}$. In this case, as in near-critical percolation away from the near-critical window \cite{massiveSLE-6}, we prove that the interface degenerates to a boundary arc. 

\begin{proposition} \label{prop_large_H}
    Let $g_2:\mathbb{R}_{+} \to \mathbb{R}_{+}$ be such that $g_2(\delta) \to \infty$ as $\delta \to 0$. Assume that $(\Omega;a,b)$ and its discrete approximations $(\Omega_{\delta};a_{\delta},b_{\delta})_{\delta}$ are as in Theorem \ref{theorem_magnetic_SLE3}. Consider the interface $\gamma_{\delta}$ of the critical Ising model in $\Omega_{\delta}$ with Dobrushin boundary conditions and external magnetic field $H(\delta) = h\delta^{15/8}g_{2}(\delta)$, for some $h>0$. Then, as $\delta \to 0$, $\gamma_{\delta}$ converges in law to $\partial \Omega^{-}$. Here, $\partial \Omega^{-}$ is the boundary arc of $\partial \Omega$ where $-1$ boundary conditions are imposed.
\end{proposition}

Observe that Proposition \ref{prop_large_H} in fact covers the case $h < 0$. Indeed,  by spin flip (broken) symmetry, using Proposition \ref{prop_large_H} but with the opposite boundary conditions for the Ising model, one can easily see that when $h<0$, $\gamma_{\delta}$ converges in law to $\partial \Omega^{+}$, where $\partial \Omega^{+}$ is the boundary arc of $\partial \Omega$ where $+1$ boundary conditions are imposed.

In view of Proposition \ref{prop_large_H}, one may wonder if in the continuum, the measure $\PP_{h}^{(\Omega,a,b)}$ of Theorem \ref{theorem_magnetic_SLE3} also degenerates to a measure supported on $\partial \Omega^{-}$ as $h \to \infty$. In this direction, we prove that the $\PP_{h}^{(\Omega,a,b)}$-probability that the curve exits an arbitrarily small neighborhood of $\partial \Omega^{-}$ vanishes as $h \to \infty$.

\begin{proposition} \label{prop_h_limit}
    Let $\Omega \subset \mathbb{C}$ be a bounded and simply connected domain with two marked boundary points $a,b \in \partial \Omega$. Let $\eta > 0$ and define $\Omega^{-}(\eta):=\{ z \in \Omega: \operatorname{dist}(z,\partial \Omega^{-}) \leq \eta \}$. Then
    \begin{equation*}
        \lim_{h \to \infty} \PP_{h}^{(\Omega,a,b)}[\gamma \cap (\Omega \setminus \Omega^{-}(\eta)) \neq \emptyset] = 0.
    \end{equation*}
\end{proposition}

For the same reasons as those explained below Proposition \ref{prop_large_H} and using the convergence result of Theorem \ref{theorem_magnetic_SLE3}, it can be deduced from Proposition \ref{prop_h_limit} that the $\PP_{h}^{(\Omega,a,b)}$-probability that the curve exits an arbitrarily small neighborhood of $\partial \Omega^{+}$ vanishes as $h \to -\infty$.

Let us conclude this section of the introduction by a few remarks. Contrary to most results establishing convergence of a discrete interface to an SLE$_{\kappa}$ curve, the proof of Theorem \ref{theorem_magnetic_SLE3} does not rely on a martingale characterization of $\PP_{h}^{(\Omega,a,b)}$. As such, it leaves open the question of how to characterize the law of $\PP_{h}^{(\Omega,a,b)}$: the characterization of the limiting interface is often an issue in near-critical regimes. We note that \cite{massiveSLE} gives a conjecture on which observable could provide a characterizing martingale for $\PP_{h}^{(\Omega,a,b)}$ but even if this was established, it would probably not be very useful to prove scaling-limit type results. Indeed, the discrete version of this observable would satisfy a discrete boundary value problem for which convergence of the discrete solution to its continuum analogue is not known.

Let us also mention that Theorem \ref{theorem_magnetic_SLE3} and its proof suggest that it should be possible to construct a coupling between the continuum magnetization field of the critical Ising model with Dobrushin boundary conditions and an SLE$_3$ curve, in a spirit similar to the GFF-SLE$_4$ coupling \cite{contour_line, Dub_SLE}. As in the case of the massive GFF, this coupling could be extended to the near-critical regime \cite{mSLE4-mGFF}. These questions will be investigated in future work. Such a coupling may turn out to be useful to compare the arm exponents in the near-critical regime (if they exist) with those in the critical regime, which is a question raised in \cite{massiveSLE}.

Besides, we note that from Theorem \ref{theorem_magnetic_SLE3}, given the topologies in which convergence is obtained, we know that the driving function of the curve $\gamma$ under $\PP_{h}^{(\Omega,a,b)}$ exists, is continuous, and its law is absolutely continuous with respect to that of the driving function of $\gamma$ under $\PP_{\operatorname{SLE}_3}^{(\Omega,a,b)}$. If continuity in $t$ of the random variable on the right-hand side of \eqref{RN_time_theorem_intro} in Theorem \ref{theorem_magnetic_SLE3} was established, one could deduce an explicit expression for the driving function of $\gamma$ under $\PP_{h}^{(\Omega,a,b)}$. This driving function would have the form of $\sqrt{3}$ times a Brownian motion plus a drift term, where we recall that $\sqrt{3}$ times a Brownian motion is the driving function of $\gamma$ under $\PP_{\operatorname{SLE}_3}^{(\Omega,a,b)}$.

Finally, perturbing the critical Ising model by an external field is not the only way to move away from the critical regime. Another way is to perturb the temperature and look at the Ising model with inverse temperature $\beta = \beta_c + m\delta$, where $\beta_c$ is the critical inverse temperature and $m \in \mathbb{R}$ \cite{deformation-Ising, universality-massive-Ising}. Proving convergence of the interface in this setting is more challenging: this stems from the fact that the energy field of the critical Ising model cannot be constructed as a random field in the continuum, whereas the magnetization field can be \cite{GarbanMagnetization, MourratIsing}.

Let us now briefly discuss the proofs of Theorem \ref{theorem_magnetic_SLE3}, Proposition \ref{prop_small_H}, Proposition \ref{prop_large_H} and Proposition \ref{prop_h_limit}.

\subsection{Outline of the proofs}

The proofs of Theorem \ref{theorem_magnetic_SLE3} and Proposition \ref{prop_small_H}, given respectively in Section \ref{sec_proof_near_critical} and Section \ref{sec_proof_small_H}, follow the same strategy. Let us denote by $\PP_{\delta,h}^{\pm}$, respectively $\PP_{\delta}^{\pm}$, the law of $\gamma_{\delta}$ when $h \neq 0$, respectively $h\equiv 0$. The first step of the proofs of Theorem \ref{theorem_magnetic_SLE3} and Proposition \ref{prop_small_H} is to show tightness of $(\gamma_{\delta})_{\delta}$ under $(\PP_{\delta,h}^{\pm})_{\delta}$. For this, we observe that for each $\delta > 0$, the law of $\gamma_{\delta}$ under $\PP_{\delta,h}^{\pm}$ has a Radon-Nikodym derivative $F_{\delta}(\gamma_{\delta})$ with respect to the law of $\gamma_{\delta}$ under $\PP_{\delta}^{\pm}$. Moreover, $(\gamma_{\delta})_{\delta}$ under $(\PP_{\delta}^{\pm})_{\delta}$ is tight in an appropriate topology by \cite{cvgIsingSLE}. In view of this, to establish tightness of $(\gamma_{\delta})_{\delta}$ under $(\PP_{\delta,h}^{\pm})_{\delta}$, it therefore suffices to prove that for any $p \in [1,\infty)$, $\sup_{\delta > 0}\EE_{\delta}^{\pm}[\vert F_{\delta}(\gamma_{\delta})\vert^p] < \infty$, where $\EE_{\delta}^{\pm}$ denotes the expectation with respect to $\PP_{\delta}^{\pm}$. This is what we achieve in the first part of the proofs of Theorem \ref{theorem_magnetic_SLE3} and Proposition \ref{prop_small_H}.

The second step of the proofs of Theorem \ref{theorem_magnetic_SLE3} and Proposition \ref{prop_small_H} is to prove a characterization of the limiting laws obtained along weakly convergent subsequences of $(\gamma_{\delta})_{\delta}$ under $(\PP_{\delta,h}^{\pm})_{\delta}$. Let $(\gamma_n)_n$ be such a weakly convergent subsequence under $(\PP_{n,h}^{\pm})_n$. To characterize its limiting law, we fist use the fact that $(\gamma_n)_n$ under $(\PP_{n}^{\pm})_n$ converges in law to $\PP_{\operatorname{SLE}_3}^{(\Omega,a,b)}$ \cite{cvgIsingSLE}. By Skorokhod representation theorem, we can therefore define on a common probability space $(S,\mathcal{F},\PP)$ a sequence $(\gamma_n)_n$ and a random curve $\gamma$ such that for each $n$, $\gamma_n$ has law $\PP_{n}^{\pm}$, $\gamma$ has law $\PP_{\operatorname{SLE}_3}^{(\Omega,a,b)}$ and $\PP$-almost surely, $\gamma_n \to \gamma$. The Radon-Nikodym derivatives $(F_n(\gamma_n))_n$ of $(\PP_{n,h}^{\pm})_n$ with respect to $(\PP_{n}^{\pm})_n$ are then all defined on $(S,\mathcal{F},\PP)$. In the setting of Proposition \ref{prop_small_H}, when the scaling of the external field is $\delta_n^{15/8}g_1(\delta_n)$, we characterize the limiting law of weakly convergent subsequences under $(\PP_{n,h}^{\pm})_n$ by showing that the discrete Radon-Nikodym derivatives $(F_n(\gamma_n))_n$ converge to $1$ in $L^2(\PP)$. This yields that the limiting law is $\PP_{\operatorname{SLE}_3}^{(\Omega,a,b)}$, as claimed. In the setting of Theorem \ref{theorem_magnetic_SLE3}, when the scaling of the external field is $\delta_n^{15/8}$, the characterization of the limiting laws of weakly convergent subsequences $(\PP_{n,h}^{\pm})_n$ is established by proving that $(F_n(\gamma_n))_n$ has a subsequence that converges in $L^1(\PP)$ to an explicit random variable. To show this convergence, a crucial result is the convergence of the $k$-point spin correlations of the critical Ising model obtained in \cite{spinCorrelations}. This ultimately allows us to express the limit of $(F_n(\gamma_n))_n$ along a subsequence in terms of explicit functions of the domain slit by the limiting curve. This is similar in spirit to the construction of the continuum planar Ising magnetization field \cite{GarbanMagnetization}. However, in our setting, more work is needed as we must also take into account the contribution of points that arbitrarily close to the boundary. Moreover, since the underlying domain is typically $\Omega$ slit by a portion of a rough curve, we cannot assume smoothness of the boundary. Another important result used in the proof of the convergence of $F_n(\gamma_n)$ is the computation of the two-arm exponent of the critical Ising model \cite{armExponentsIsing}. This is instrumental to show that a quantity that can be seen as the magnetization of the discrete interface vanishes in the limit $n \to \infty$. This also provides a useful bound on the growth of the size of the interface, which is necessary to control what happens near the boundary of the domain.

The proof of Proposition \ref{prop_large_H}, given in Section \ref{sec_proof_large_H}, is more probabilistic in nature. The heuristic behind this result is that spins tend to align with the magnetic field and that the scaling $\delta^{\frac{15}{8}} g_2(\delta)$ makes this phenomena dominate in the limit $\delta \to 0$. To implement this idea, we rely on the Edwards-Sokal coupling for the Ising model with an external magnetic field, which was investigated in \cite{EScoupling_magnetic, CamiaFKIsing, Camia_exponential_decay}. For this, the first step is to change the Dobrushin boundary conditions into $+1$ boundary conditions using spatial-mixing-type arguments which build on what happens at criticality \cite{crossing-rectangles}. This change of boundary conditions then enables us to use the Edwards-Sokal coupling with an external field to show that, for arbitrary $\eta >0$, the probability that there is a path of $+1$ spins staying within distance $\eta$ from $\partial \Omega_{\delta}^{-}$ and joining the $\eta$-neighborhoods of $a_{\delta}$ and $b_{\delta}$ on $\partial \Omega_{\delta}^{+}$ converges to $1$ as $\delta \to 0$. To show this, an important result is the joint convergence of the lattice FK-Ising cluster boundaries to CLE$_{16/3}$ and of the discrete FK-Ising cluster area measures to CME$_{16/3}$ \cite{cvg-CLE-CME}. Having such a path of $+1$ spins rules out that a path with only $-1$ spins on one of its sides reaches distance $\eta$ from $\partial \Omega_{\delta}^{-}$ as $\delta \to 0$, which proves Proposition \ref{prop_large_H}.

Proposition \ref{prop_h_limit} is shown in Section \ref{sec_h_limit}. The proof relies on both Theorem \ref{theorem_magnetic_SLE3} and Proposition \ref{prop_large_H}. In particular, for the argument to work, we use results obtained in the discrete setting and, interestingly, we do not know how to obtain Proposition \ref{prop_h_limit} directly in the continuum.

\paragraph{Acknowledgments.} The author thanks S\'ebastien Martineau for suggesting a short argument to prove Proposition \ref{prop_h_limit}. The author is also grateful to Ellen Powell for a careful reading of an earlier version of this manuscript. Tyler Helmuth and Titus Lupu are also thanked for interesting discussions at various stages of this project. Part of this work was conducted while the author was at Sorbonne University; their hospitality is gratefully acknowledged. This work was funded by an EPSRC studentship.

\section{Setting and background} \label{sec_setting}

We detail here the assumptions under which Theorem \ref{theorem_magnetic_SLE3}, Proposition \ref{prop_small_H} and Proposition \ref{prop_large_H} are proved: the assumptions on $(\Omega_{\delta};a_{\delta},b_{\delta})_{\delta}$ and $(\Omega;a,b)$ are made precise in Section \ref{sec_assumptions_domain} and the topologies in which weak convergence is established are described in Section \ref{sec_topology_curves}. Section \ref{section_background_Ising} and Section \ref{section_spin_correlations} provide some background on the Ising model.

\subsection{Assumptions on the domain and Carath\'eodory convergence} \label{sec_assumptions_domain}

We consider a bounded, open and simply connected subset $\Omega$ of the complex plane $\mathbb{C}$ such that $0 \in \Omega$. We fix two marked boundary points $a, b \in \partial \Omega$ and we assume that both $a$ and $b$ are degenerate prime ends of $\Omega$, see \cite[Section~2.5]{Pommeranke} for a definition. These assumptions on $\Omega$ and the boundary points $a$ and $b$ will in particular allow us to use \cite[Theorem~4.2]{Karrila}. 

Regarding the assumptions on the boundary $\partial \Omega$ of $\Omega$, Theorem \ref{theorem_magnetic_SLE3} and Proposition \ref{prop_small_H} will be shown assuming that the Hausdorff dimension of $\partial \Omega$ is strictly smaller that $\frac{7}{4}$. However, the proof of Proposition \ref{prop_large_H} requires $\partial \Omega$ to be more regular and we will actually take it to be smooth, for reasons that we explain in the proof of Claim \ref{claim_proba_path_clusters}.

We assume that $(\Omega_{\delta})_{\delta}$ is a sequence of graphs approximating $\Omega$ in a sense that we will now explain. For each $\delta > 0$, $\Omega_{\delta}$ is a simply connected subgraph of the square lattice $\delta \mathbb{Z}^2$, so that every edge of $\Omega_{\delta}$ has length $\delta$. We denote by $V(\Omega_{\delta})$ the set of vertices of $\Omega_{\delta}$ and define the boundary $\partial \Omega_{\delta}$ of $\Omega_{\delta}$ as
\begin{equation*}
    \partial \Omega_{\delta} := \{ w \in \delta \mathbb{Z}^2 \setminus V(\Omega_{\delta}): \text{there exists $v \in V(\Omega_{\delta})$ such that $v \sim w$}\}
\end{equation*}
where $v \sim w$ means that there is an edge of $\delta \mathbb{Z}^2$ connecting $v$ and $w$. With this definition of $\partial \Omega_{\delta}$, it is legitimate to set $\Int(\Omega_{\delta}):=V(\Omega_{\delta})$. We also let $E(\Omega_{\delta})$ denote the set of edges of $\delta\mathbb{Z}^2$ with at least one endpoint in $\Int(\Omega_{\delta})$.

We associate to each $\Omega_{\delta}$ an open and simply connected polygonal domain
$\hat \Omega_{\delta} \subset \mathbb{C}$ by taking the union of all open squares with side length $\delta$ centered at vertices in $V(\Omega_{\delta})$. We assume that for any $\delta > 0$, $0$ belongs to $\hat \Omega_{\delta}$ and that there exists $\bar{R} > 0$ such that for any $\delta > 0$, $\hat \Omega_{\delta} \subset B(0,\bar{R})$ and $\Omega \subset B(0,\bar{R})$. These assumptions are necessary to apply \cite[Theorem~4.2]{Karrila}. The marked boundary points $a$ and $b$ of $\partial \Omega$ are then approximated by sequences $(a_{\delta})_{\delta}$ and $(b_{\delta})_{\delta}$ where, for each $\delta >0$, $a_{\delta}$ and $b_{\delta}$ belong to $\partial \hat \Omega_{\delta}$. 

The sequence $(\hat \Omega_{\delta}; a_{\delta}, b_{\delta})_{\delta}$ is assumed to converge to $(\Omega; a, b)$ in the Carath\'eodory topology. That is,
\begin{itemize}
    \item each inner point $z \in \Omega$ belongs to $\hat \Omega_{\delta}$ for $\delta$ small enough;
    \item for every boundary point $\zeta \in \partial \Omega$, there exists a sequence $(\upzeta_{\delta})_{\delta}$ such that $\upzeta_{\delta} \to \upzeta$ as $\delta \to 0$, where, for each $\delta >0$, $\upzeta_{\delta} \in \partial \hat \Omega_{\delta}$.
\end{itemize}
This can be rephrased in terms of conformal maps. Let $\psi: \Omega \to \mathbb{D}$ be a conformal map such that $\psi(a)=1$ and $\psi(0)=0$. Similarly, for each $\delta >0$, let $\psi_{\delta}: \hat \Omega_{\delta} \to \mathbb{D}$ be a conformal map such that $\psi_{\delta}(a_{\delta})=1$ and $\psi_{\delta}(0)=0$. Then, by \cite[Theorem~1.8]{Pommeranke}, the Carath\'eodory convergence of $(\hat \Omega_{\delta}; a_{\delta}, b_{\delta})_{\delta}$ to $(\Omega; a, b)$ is equivalent to
\begin{align*}
    & \psi_{\delta} \to \psi \quad \text{uniformly on compact subsets of $\Omega$ and}\\
    & \psi_{\delta}^{-1} \to \psi^{-1} \quad \text{uniformly on compact subsets of $\mathbb{D}$}.
\end{align*}
Furthermore, we suppose that $a_{\delta}$, respectively $b_{\delta}$, is a close approximation of $a$, respectively $b$, as defined by Karrila in \cite[Section~4.3]{Karrila}. Let us recall what this means. To lighten the notations, we identify the prime ends $a_{\delta}$ and $a$ with their corresponding radial limit points. For $r>0$, let $S_r$ be the arc of $\partial B(a,r) \cap \Omega$ disconnecting in $\Omega$ the prime end $a$ from $0$ and that is closest to $a$. In other words, $S_r$ is the last arc from the (possibly countable) collection $\partial B(a,r) \cap \Omega$ of arcs that a path running from $0$ to $a$ inside $\Omega$ must cross. Such an arc exists by \cite[Lemma~A.1]{Karrila} and approximation by radial limits. $a_{\delta}$ is then said to be a close approximation of $a$ if
\begin{itemize}
    \item $a_{\delta} \to a$ as $\delta \to 0$; and
    \item for each $r$ small enough and for all sufficiently (depending on $r$) small $\delta$, the boundary point $a_{\delta}$ of $\hat \Omega_{\delta}$ is connected to the midpoint of $S_r$ inside $\hat \Omega_{\delta} \cap B(a,r)$. 
\end{itemize}

In what follows, we will need to split the boundary $\partial \Omega_{\delta}$ of $\Omega_{\delta}$ into two set of vertices, depending on their position with respect to $a_{\delta}$ and $b_{\delta}$. We therefore denote by $\partial \Omega_{\delta}^{-}$, respectively $\partial \Omega_{\delta}^{+}$, the vertices of $\partial \Omega_{\delta}$ that are traversed when following the vertices of $\partial \Omega_{\delta}$ from $a_{\delta}$ to $b_{\delta}$ counterclockwise, respectively clockwise.

\subsection{The Ising model} \label{section_background_Ising}

For $\delta >0$, let $(\Omega_{\delta};a_{\delta}, b_{\delta})$ be as above. The Ising model on $\Omega_{\delta}$ is a probability measure on spin configurations $\sigma: V(\Omega_{\delta}) \cup \partial \Omega_{\delta} \to \{-1,1\}^{\vert V(\Omega_{\delta}) \cup \partial \Omega_{\delta}\vert}$. In our setting, since the boundary $\partial \Omega_{\delta}$ of $\Omega_{\delta}$ is non-empty, boundary conditions can be imposed to specify the measure. Two different choices of boundary conditions will be of interest here. The first one is that of Dobrushin boundary conditions: the spins of the vertices of $\partial \Omega_{\delta}^{-}$ are all $-1$ while the spins of the vertices of $\partial \Omega_{\delta}^{+}$ are all $+1$. The probability of an Ising spin configuration $\sigma$ in $\Omega_{\delta}$ with Dobrushin boundary conditions is given by
\begin{equation} \label{def_Ising_pm}
    \PP_{\Omega_{\delta},\beta}^{\pm}(\sigma)=\frac{1}{\mathcal{Z}_{\beta}}\exp\bigg(\beta \sum_{x \sim y} \sigma_x\sigma_y\bigg)\mathbb{I}_{\sigma_{\vert \partial \Omega_{\delta}^{+}}=+1}\mathbb{I}_{\sigma_{\vert \partial \Omega_{\delta}^{-}}=-1}
\end{equation}
where $\beta > 0$ is a parameter called the inverse temperature. Above, the sum is over nearest neighbor vertices $x \sim y$ of $\delta \mathbb{Z}^2$ such that at least one of them belong to $V(\Omega_{\delta})$.

The other set of boundary conditions that we will consider is that of $+1$ boundary conditions and $-1$ boundary conditions. The probability $\PP_{\Omega_{\delta}, \beta}^{+}(\sigma)$, respectively $\PP_{\Omega_{\delta}, \beta}^{-}(\sigma)$, of an Ising spin configuration $\sigma$  in $\Omega_{\delta}$ with $+1$, respectively $-1$, boundary conditions is given by replacing $\mathbb{I}_{\sigma_{\vert \partial \Omega_{\delta}^{+}}=+1}\mathbb{I}_{\sigma_{\vert \partial \Omega_{\delta}^{-}}=-1}$ by $\mathbb{I}_{\sigma_{\vert \partial \Omega_{\delta}}=+1}$, respectively $\mathbb{I}_{\sigma_{\vert \partial \Omega_{\delta}}=-1}$ in \eqref{def_Ising_pm}. In the course of the proof of Proposition \ref{prop_large_H}, we will also use the Ising model with free boundary conditions: in this case, the spins of the vertices on $\partial \Omega_{\delta}$ is not prescribed.

The parameter $\beta$ plays an important role in the macroscopic behavior of the Ising model on $\Omega_{\delta}$. Indeed, there exists a critical value $0 < \beta_{c} < \infty$ at which the model undergoes a phase transition. On the square lattice, it has been shown that $\beta_{c} = \frac{1}{2}\log(\sqrt{2}+1)$ \cite{book-Ising}. Moreover, at $\beta = \beta_{c}$, as $\delta \to 0$, the Ising model on the square lattice exhibits conformal invariance properties and this has been the topic of extensive research in the past twenty years, see \cite{Chelkak-survey, DC-survey} for a survey. Here, we will always take $\beta = \beta_{c}$ and we set $\PP_{\Omega_{\delta}}^{\pm}:=\PP_{\Omega_{\delta}, \beta_c}^{\pm}$ and $\PP_{\Omega_{\delta}}^{+}:=\PP_{\Omega_{\delta}, \beta_c}^{+}$ and call the corresponding Ising model the critical Ising model.

In what follows, we will consider the critical Ising model perturbed by an external magnetic field. For $H: V(\Omega_{\delta}) \to \mathbb{R}$ a bounded function, when we impose Dobrushin boundary conditions, this is a probability measure denoted $\PP_{\Omega_{\delta}, H}^{\pm}$ on spin configurations $\sigma: V(\Omega_{\delta}) \cup \partial \Omega_{\delta} \to \{-1,+1\}^{\vert V(\Omega_{\delta}) \cup \partial \Omega_{\delta} \vert}$ that can be defined via its Radon-Nikodym derivative with respect to $\PP_{\Omega_{\delta}}^{\pm}$ as follows:
\begin{equation*}
    \frac{\der \PP_{\Omega_{\delta}, H}^{\pm}}{\der \PP_{\Omega_{\delta}}^{\pm}} (\sigma) = \frac{1}{\mathcal{Z}_{H}}\exp\bigg( \sum_{x \in V(\Omega_{\delta})} H_x \sigma_x \bigg)
\end{equation*}
where $\mathcal{Z}_{H}$ is a normalization constant. We can also impose $+1$ boundary conditions to the model, in which case the corresponding probability measure $\PP_{\Omega_{\delta},H}^{+}$ has Radon-Nikodym derivative with respect to $\PP_{\Omega_{\delta}}^{+}$ given by
\begin{equation*}
    \frac{\der \PP_{\Omega_{\delta}, H}^{+}}{\der \PP_{\Omega_{\delta}}^{+}} (\sigma) = \frac{1}{\mathcal{Z}_{H}}\exp\bigg( \sum_{x \in V(\Omega_{\delta})} H_x \sigma_x\bigg).
\end{equation*}
Below, we will typically choose functions $H:V(\Omega_{\delta}) \to \mathbb{R}$ which depend on the meshsize $\delta >0$ of $\Omega_{\delta}$. Moreover, if $h: \Omega \to \mathbb{R}$, we define the function $h_{\delta}: V(\Omega_{\delta}) \to \mathbb{R}$ as the restriction of $h$ to $V(\Omega_{\delta})$, that is, for $x \in V(\Omega_{\delta})$, $h_{\delta}(x)=h(x)$.

An important property of the Ising model at arbitrary inverse temperature $\beta > 0$ perturbed by an external magnetic field $H$ is the FKG property \cite[Theorem~3.21]{book-Ising}. We introduce the partial order $\leq$ on spin configurations defined by $\sigma \leq \sigma'$ if and only if $\sigma_x \leq \sigma_x'$ for any $x \in V(\Omega_{\delta}) \cup \partial \Omega_{\delta}$. A function $f: \{-1,1\}^{V(\Omega_{\delta}) \cup \partial \Omega_{\delta}} \to \mathbb{R}$ is said to be increasing if $f(\sigma) \leq f(\sigma')$ whenever $\sigma \leq \sigma'$. Similarly, we say that an event $E$ is increasing if the indicator function $\mathbb{I}_{E}$ is increasing.

\begin{proposition}[FKG inequality]
    Let $f,g$ be two increasing functions and $H: V(\Omega_{\delta}) \to \mathbb{R}$. Then, for any boundary conditions $\xi$, $\EE_{\Omega_{\delta},\beta,H}^{\xi}[fg] \geq \EE_{\Omega_{\delta},\beta,H}^{\xi}[f]\EE_{\Omega_{\delta},\beta,H}^{\xi}[g]$.
\end{proposition}

As a consequence of the FKG inequality, we have the following comparison between the Ising model perturbed by a non-negative external field and the Ising model with no external field.

\begin{lemma} 
    Let $f$ be an increasing function and let $H: V(\Omega_{\delta}) \to \mathbb{R}_{+}$. Then, $\EE_{\Omega_{\delta},\beta,H}^{+}[f] \geq \EE_{\Omega_{\delta},\beta}^{+}[f]$.
\end{lemma}

Ising measures with different boundary conditions can also be compared with one another. The following inequality will be useful, see \cite[Theorem~7.6]{book-parafermions} for a proof.

\begin{proposition} \label{prop_comparison_BC_Ising}
    Let $f$ be an increasing function. Then, $\EE_{\Omega_{\delta}, \beta}^{+}[f] \geq \EE_{\Omega_{\delta},\beta}^{\operatorname{free}}[f]$.
\end{proposition}

\subsubsection{The discrete interface of the Ising model}

Under $\PP_{\Omega_{\delta}}^{\pm}$ and $\PP_{\Omega_{\delta},H}^{\pm}$, in each spin configuration, there is an interface $\gamma_{\delta}$ separating $+1$ and $-1$ spins: this is a discrete path running from $a_{\delta}$ to $b_{\delta}$ such that a vertex adjacent to it has spin $+1$ if it is on its left and $-1$ if it is on its right. Note that this interface is a priori not unique. In what follows, we always choose the leftmost such interface but this choice is irrelevant: any other reasonable way of prescribing how to turn when encountering a face surrounded by four alternating spins would yield the same scaling limit as $\delta \to 0$. As in \cite{cvgIsingSLE}, for technical reasons, we draw $\gamma_{\delta}$ on the auxiliary square-octagon lattice, with octagons corresponding to vertices of $\Omega_{\delta}$. This for example prevents $\gamma_{\delta}$ from having self-intersection points.

\subsubsection{The Edwards-Sokal coupling} \label{sec_ES_coupling_no_external_field}

The Edwards-Sokal coupling provides a coupling between a edge configuration on the graph $G_{\delta}=(V(\Omega_{\delta}),E_{\delta})$, sampled from the FK-Ising measure, and an Ising configuration on $\Omega_{\delta}$ (with no external field). As this coupling will be a useful tool in what follows, we recall it below.

We call a function $\omega:E(\Omega_{\delta}) \to \{0,1\}^{\vert E(\Omega_{\delta}) \vert}$ an edge configuration. We say that an edge $e$ is open in $\omega$ if $\omega(e)=1$. Otherwise, $e$ is said to be closed. We identify $\omega$ with the subgraph whose vertex set is $V(\Omega_{\delta})$ and whose edge set is $\{e \in E(\Omega_{\delta}): \omega(e)=1 \}$. A cluster of $\omega$ is a connected component of this subgraph. One can also impose boundary conditions on $\partial \Omega_{\delta}$: this is a partition $\xi$ of $\partial \Omega_{\delta}$ and two vertices are wired together if they are in the same element of the partition. When $\xi=\operatorname{free}$, no vertices of $\partial \Omega_{\delta}$ are wired together and when $\xi=\operatorname{wired}$, all boundary vertices are wired together.

With these notations, the FK-Ising percolation measure with parameter $p \in [0,1]$ is defined as follows: for $\xi$ a set of boundary conditions and $\omega$ an edge configuration,
\begin{equation*}
    \PP_{\delta,p}^{\operatorname{FK},\xi}(\omega) = \frac{1}{\mathcal{Z}_{p}} 2^{k(\omega^{\xi})}\prod_{e \in E(\Omega_{\delta})} p^{\omega(e)}(1-p)^{1-\omega(e)}
\end{equation*}
where $\omega^{\xi}$ is the edge configuration obtained by identifying wired boundary vertices and $k(\omega^{\xi})$ the number of clusters of $\omega^{\xi}$. Above, $\mathcal{Z}_{p}$ is a normalization constant chosen such that $\PP_{\delta,p}^{\operatorname{FK},\xi}$ is a probability measure.

The measure $\PP_{\delta,p}^{\operatorname{FK},\xi}$ satisfies an FKG property similar to that of the Ising model and we will use it on several occasions. To compare edges configurations, we introduce a partial order $\leq$ defined as follows: for $\omega, \omega'$ two edge configurations, $\omega \leq \omega'$ if and only if $\omega(e) \leq \omega'(e)$ for any $e \in E(\Omega_{\delta})$. A function $f:\{0,1\}^{E(\Omega_{\delta})} \to \mathbb{R}$ is then said to be increasing if $f(\omega) \leq f(\omega')$ whenever $\omega \leq \omega'$. Similarly, an event $E$ is said to be increasing if the indicator function $\mathbb{I}_{E}$ is increasing. The FKG property of the FK-Ising model is the following inequality \cite[Theorem~4.14]{book-parafermions}. 

\begin{proposition} \label{prop_FKG_FK}
    Fix $p \in [0,1]$. Let $f$ and $g$ be two increasing functions and let $\xi$ be a set of boundary conditions. Then $\EE_{\delta,p}^{\operatorname{FK},\xi}[fg] \geq \EE_{\delta,p}^{\operatorname{FK},\xi}[f] \EE_{\delta,p}^{\operatorname{FK},\xi}[g]$.
\end{proposition}

This proposition yields the following comparison between boundary conditions, see \cite[Corollary~4.19]{book-parafermions} for a proof.

\begin{proposition} \label{prop_comparison_BC_FK}
    Fix $p \in [0,1]$. Let $\xi \leq \psi$ be two sets of boundary conditions and let $f$ be an increasing function. Then, $\EE_{\delta,p}^{\operatorname{FK},\psi}[f] \geq \EE_{\delta,p}^{\operatorname{FK},\xi}[f]$.
\end{proposition}

We can now describe the Edwards-Sokal coupling for the Ising model in $\Omega_{\delta}$. Fix $p \in [0,1]$ and let $\xi \in \{ \operatorname{wired}, \operatorname{free}\}$. Sample first an edge configuration $\omega$ on $G_{\delta}$ according to $\PP_{\delta,p}^{\operatorname{FK},\xi}$. Then, assign independently to each cluster of $\omega$ not connected to $\partial \Omega_{\delta}$ a spin $+1$ or $-1$ with probability $1/2$. If $\xi=\operatorname{free}$, proceed in a similar way for clusters connected to $\partial \Omega_{\delta}$ and if $\xi=\operatorname{wired}$, assign the spin $+1$ to the cluster connected to $\partial \Omega_{\delta}$. Finally, for each vertex $x \in V(\Omega_{\delta})$, define $\sigma_x$ to be the spin of the cluster of $x$ in $\omega$.

\begin{proposition}[Edwards-Sokal coupling \cite{ES-coupling}] \label{prop_ES_coupling_no_field}
    The spin configuration $\sigma$ constructed above is distributed according to the Ising measure with inverse temperature $\beta = -\frac{1}{2}\log(1-p)$ and boundary conditions $\xi$, where $\xi \in \{ \operatorname{wired}, \operatorname{free} \}$.
\end{proposition}

Below, when $\beta = \beta_{c}$, we simply write $\PP_{\delta}^{\operatorname{FK},\xi}$ instead of $\PP_{\delta,p_c}^{\operatorname{FK},\xi}$, where $p_c$ is given by Proposition \ref{prop_ES_coupling_no_field}.

\subsection{Scaling limit of the interface at criticality}

\subsubsection{Loewner chains}

Set $\HH:= \{ z \in \mathbb{C}: \Im(z) > 0 \}$ and let $\gamma: [0,\infty) \to \overline{\mathbb{H}}$ be a non-self-crossing curve targeting $\infty$ and such that $\gamma(0)=0$. For $t \geq 0$, let $K_t$ be the hull generated by $\gamma([0,t])$, that is $\HH \setminus K_t$ is the unbounded connected component of $\HH \setminus \gamma([0,t])$. In the case where $\gamma([0,t])$ is non-self-touching, $K_t$ is simply given by $\gamma([0,t])$. For each $t \geq 0$, it is easy to see that there exists a unique conformal $g_t: \HH \setminus K_t \to \HH$ satisfying the normalization $g_t(\infty) = \infty$ and such that $\lim_{z \to \infty} (g_t(z) - z) = 0$. It can then be proved that $g_t$ satisfies the asymptotic
\begin{equation*}
    g_t(z) = z + \frac{a_1(t)}{z} + O(\vert z \vert^{-2}) \quad \text{as } \vert z \vert \to \infty.
\end{equation*}
The coefficient $a_1(t)$ is equal to $\text{hcap}(K_t)$, the half-plane capacity of $K_t$, which, roughly speaking, is a measure of the size of $K_t$ seen from $\infty$. Moreover, one can show that $a_1(0)=0$ and that $t \mapsto a_1(t)$ is continuous and strictly increasing. Therefore, the curve $\gamma$ can be reparametrized in such a way that at each time $t$, $a_1(t) = 2t$. $\gamma$ is then said to be parameterized by half-plane capacity.

In this time-reparametrization and with the normalization of $g_t$ just described, it is known that there exists a unique real-valued function $t \mapsto W_t$, called the driving function, such that the following equation, called the Loewner equation, is satisfied:
\begin{equation} \label{eq_Loewner}
    \partial_t g_t(z) = \frac{2}{g_t(z) - W_t},  \quad g_0(z)=z, \quad \text{for all $z \in \HH \setminus K_t$}.
\end{equation}
Indeed, it can be shown that $g_t$ extends continuously to $\gamma(t)$ and setting $W_t = g_t(\gamma(t))$ yields the above equation, see e.g. \cite[Chapter~4]{book_Lawler} and \cite[Chapter~4]{book_SLE}.

Conversely, given a continuous and real-valued function $t \mapsto W_t$, one can construct a locally growing family of hulls $(K_t)_t$ by solving the equation \eqref{eq_Loewner}. Under additional assumptions on the function $t \mapsto W_t$, the family of hulls obtained using \eqref{eq_Loewner} is generated by a curve, in the sense explained above. 

Schramm-Loewner evolutions, or SLE for short, are random Loewner chains introduced by Schramm \cite{Schramm_SLE}. For $\kappa \geq 0$, SLE$_\kappa$ is the Loewner chain obtained by considering the Loewner equation \eqref{eq_Loewner} with driving function $W_t = \sqrt{\kappa}B_t$, where $(B_t, t \geq 0)$ is a standard one-dimensional Brownian motion. As such, SLE$_{\kappa}$ is defined in $\HH$ but, thanks to the conformal invariance of the Loewner equation, SLE$_{\kappa}$ can be defined in any simply connected domain $\Omega \subset \mathbb{C}$ with two marked boundary points $a, b \in \partial \Omega$ by considering a conformal map $\phi: \Omega \to \HH$ with $\phi(a)=0$ and $\phi(b) = \infty$ and taking the image of SLE$_{\kappa}$ in $\HH$ by $\phi^{-1}$. In particular, SLE$_{\kappa}$ is conformally invariant and it turns out that this conformal invariance property together with a certain domain Markov property characterize the family (SLE$_{\kappa}, \kappa \geq 0)$. 

\subsubsection{Topology and convergence on the space of curve} \label{sec_topology_curves}

Following \cite{tightness-curves}, the space of curves that we will consider is a subspace of the space of continuous mappings from $[0,1]$ to $\mathbb{C}$ modulo reparametrization. More precisely, let
\begin{equation*}
    \mathcal{C}':= \left \{
    f \in \mathcal{C}([0,1], \mathbb{C}): 
    \begin{aligned}
    &\text{ either $f$ is not constant on any subinterval of $[0,1]$} \\ &\text{or $f$ is constant on $[0,1]$} 
    \end{aligned}
    \right \}
\end{equation*}
and let $f_1, f_2 \in \mathcal{C}'$ be equivalent if there exists an increasing homeomorphism $\psi: [0,1] \to [0,1]$ with $f_2 = f_1 \circ \psi$. We denote by $[f]$ the equivalence class of $f$ under this equivalence relation and set
\begin{equation*}
    X(\mathbb{C}):= \{ [f]: f \in \mathcal{C}'\}.
\end{equation*}
$X(\mathbb{C})$ is called the space of curves. We turn $X(\mathbb{C})$ into a metric space by equipping it with the metric
\begin{equation*}
    d_X(f,g) := \inf \{ \| f_0 - g_0 \|_{\infty}: \, f_0 \in [f], g_0 \in [g] \}.
\end{equation*}
$(X(\mathbb{C}), d_X)$ is a separable and complete metric space, but it is not compact. Given $\Omega \subsetneq \mathbb{C}$ with $\partial \Omega \neq \emptyset$ and two marked boundary points $a$ and $b$, we define the space of simple curves from $a$ to $b$ in $\Omega$ as
\begin{equation*}
    X_{\text{simple}}(\Omega, a, b) := \{ [f]: f \in \mathcal{C}', f((0,1)) \subset \Omega, f(0)=a, f(1)=b, \, \text{$f$ injective} \}.
\end{equation*}
We then let $X_0(\Omega, a, b)$ be the closure of $X_{\text{simple}}(\Omega, a, b)$ in $X(\mathbb{C})$ with respect to the metric $d_X$. Curves in $X_0(\Omega, a, b)$ run from $a$ to $b$, may touch $\partial \Omega$ elsewhere than at their endpoints, may touch themselves and have multiple points but they can have no transversal self-crossings. Notice that if $(\PP_n)_n$ is a sequence of probability measures supported on $X_{\text{simple}}(\Omega, a, b)$ that converges weakly to a probability measure $\PP^{*}$, then a priori $\PP^{*}$ is supported on $X_0(\Omega, a, b)$.

Assume that we have a family of random curves $(\gamma_{\delta})_{\delta}$ distributed according to some family $(\PP_{\delta})_{\delta}$ such that for each $\delta > 0$, $\gamma_{\delta}$ is an element of $X_{0}(\hat \Omega_{\delta};a_{\delta},b_{\delta})$, with $((\hat \Omega_{\delta};a_{\delta},b_{\delta}))_{\delta}$ satisfying the assumptions of Section \ref{sec_assumptions_domain}. The curves $(\gamma_{\delta})_{\delta}$ can be made to be supported on the same space $X_{0}(\Omega;a,b)$ by uniformization. More precisely, let $\phi: \Omega \to \mathbb{H}$ be a conformal map such that $\phi(a) = 0$ and $\phi(b)=\infty$ and similarly, for $\delta > 0$, let $\phi_{\delta}: \hat \Omega_{\delta} \to \mathbb{H}$ be a conformal map such that $\phi_{\delta}(a_{\delta}) = 0$ and $\phi_{\delta}(b_{\delta})=\infty$. Setting $\gamma_{\delta}^{\HH}:= \phi_{\delta}(\gamma_{\delta})$, the family $(\gamma_{\delta}^{\HH})_{\delta}$ is supported on $X_{0}(\HH;0,\infty)$. Moreover, for each $\delta > 0$, when parametrized by half-plane capacity, $\gamma_{\delta}^{\HH}$ has an associated driving function $W_{\delta}$. In this setting, there are two different topological spaces in which we may wish to show tightness of $(\gamma_{\delta})_{\delta}$:
\begin{enumerate}[leftmargin=1.5cm]
    \item [{\crtcrossreflabel{(T.1)}[topo_1]}] the space of curves $X_0(\HH,0,\infty)$ equipped with the metric $d_X$;
    \item [{\crtcrossreflabel{(T.2)}[topo_2]}] the metrizable space of continuous functions on $[0, \infty)$ with the topology of uniform convergence on compact subsets of $[0, \infty)$.
\end{enumerate}
Tightness could also be established for the driving functions $(W_{\delta})_{\delta}$, in which case the topological space to consider is
\begin{enumerate}[itemsep=-1ex, leftmargin=1.5cm]
    \item [{\crtcrossreflabel{(T.3)}[topo_3]}] the metrizable space of continuous functions on $[0, \infty)$ with the topology of uniform convergence on compact subsets of $[0, \infty)$.
\end{enumerate}
It turns out that, by \cite[Corollary~1.7]{tightness-curves}, weak convergence in one of the topologies \ref{topo_1}--\ref{topo_3} implies weak convergence in the other two and that the limits agree, in the sense that the limiting random curve is driven by the limiting driving function.

Another natural topological space in which to establish tightness of $(\gamma_{\delta})_{\delta}$ is
\begin{enumerate}[leftmargin=1.5cm]
    \item [{\crtcrossreflabel{(T.4)}[topo_4]}] the space of curves $X(\mathbb{C})$ equipped with the metric $d_X$.
\end{enumerate}
Under the assumption that $(\hat \Omega_{\delta}; a_{\delta}, b_{\delta})_{\delta}$ converges in the Carath\'eodory sense to $(\Omega;a,b)$ and that $(a_{\delta})_{\delta}$ and $(b_{\delta})_{\delta}$ are close approximations of the degenerate prime ends $a$ and $b$, \cite[Theorem~4.1]{Karrila} and \cite[Theorem~4.2]{Karrila} show that weak convergence of $(\gamma_{\delta})_{\delta}$ in the topology \ref{topo_4} implies weak convergence in the other three topologies \ref{topo_1}--\ref{topo_3}, see the discussion below \cite[Theorem~4.2]{Karrila}. Moreover, if $\gamma$ denotes the limit in the topology \ref{topo_4} and $\gamma^{\HH}$ denotes the limit in the topology \ref{topo_1}, then $\gamma$ has the same law as $\phi^{-1}(\gamma^{\HH})$. We note that \cite[Theorem~4.2]{Karrila} also guarantees that $\gamma$ is supported on $\overline{\Omega}$.

\subsubsection{The critical Ising interface}

A conformal invariance property of the critical Ising model emerges when looking at the scaling limit of the interface $\gamma_{\delta}$ under $\PP_{\Omega_{\delta}}^{\pm}$ as $\delta \to 0$. Indeed, as a consequence of the next theorem, its limiting law is conformally invariant. 

\begin{theorem}[\cite{cvgIsingSLE}] \label{theorem_SLE3}
Under the above assumptions on $(\Omega;a,b)$ and $(\Omega_{\delta},a_{\delta},b_{\delta})_{\delta}$, as $\delta \to 0$, $(\gamma_{\delta})_{\delta}$ under $(\PP_{\delta}^{\pm})_{\delta}$ converges weakly in the topologies \ref{topo_1}--\ref{topo_4} to SLE$_3$ in $\Omega$ from $a$ to $b$.
\end{theorem}

Theorem \ref{theorem_SLE3} will be the starting point of the proof of Theorem \ref{theorem_magnetic_SLE3}.

\subsection{Spin correlations at $\beta = \beta_c$} \label{section_spin_correlations}

The crucial input for the proof of Theorem \ref{theorem_magnetic_SLE3} is the existence of a scaling limit for the $k$-point spin correlations of the critical Ising model, when the boundary conditions are either $+1$, $-1$ or Dobrushin. This result, established in \cite{spinCorrelations}, holds under no assumption on the smoothness of the boundary of the underlying domain $\Omega$. This fact is essential as we will typically apply this result when the underlying domain is a domain slit by a portion of an SLE$_3$ curve. 

\begin{theorem}[\cite{spinCorrelations}] \label{theorem_spin_correlations}
Let $\Omega \subset \mathbb{C}$ be a bounded and simply connected domain and $a, b \in \partial \Omega$ be two prime ends of $\partial \Omega$. Let $(\Omega_{\delta};a_{\delta}, b_{\delta})_{\delta}$ be discrete approximations of $(\Omega;a,b)$ satisfying the assumptions of Section \ref{sec_assumptions_domain}. Then, for any $\eta >0$ and any $k \geq 1$, we have that
\begin{align*}
    &\delta^{-\frac{k}{8}} \EEpm{\delta}[\sigma_{x_1}\dots \sigma_{x_{k}}] = C_{\sigma}^k f_{\Omega}^{(\pm, k)}(x_1,\dots ,x_k) + o(1) \\
    &\delta^{-\frac{k}{8}} \EEplus{\delta}[\sigma_{x_1}\dots \sigma_{x_{k}}] = C_{\sigma}^k f_{\Omega}^{(+,k)}(x_1,\dots ,x_k) + o(1) \\
    &\delta^{-\frac{k}{8}} \EE_{\delta}^{-}[\sigma_{x_1}\dots \sigma_{x_{k}}] = C_{\sigma}^k f_{\Omega}^{(-,k)}(x_1,\dots ,x_k) + o(1)
\end{align*}
as $\delta \to 0$, uniformly over all $x_1, \dots, x_k \in \Omega$ at distance at least $\eta$ from $\partial \Omega$ and from each other. Here, $C_{\sigma}>0$ is a constant that corresponds to the fact that we are working on the square lattice. The functions $f_{\Omega}^{(\pm,k)}$, $f_{\Omega}^{(+,k)}$ and $f_{\Omega}^{(-,k)}$ are explicit and satisfy the following conformal covariance property. If $\psi: \Omega \to \tilde \Omega$ is a conformal map, then for any $k \geq 1$ and any $x_1, \dots , x_k \in \Omega$, with $\xi \in \{\pm,-,+\}$,
\begin{align*}
    &f_{\Omega}^{(\xi, k)}(x_1,\dots ,x_k) = \prod_{j=1}^{k} \vert \psi'(x_j) \vert^{1/8} f_{\tilde \Omega}^{(\xi,k)}(\psi(x_1),\dots ,\psi(x_k)).
\end{align*}
\end{theorem}

\begin{remark}
In \cite{spinCorrelations}, Theorem \ref{theorem_spin_correlations} is established under the assumption that $\Omega_{\delta}$ is a subset of the rotated square lattice $\sqrt{2}\delta e^{i\frac{\pi}{4}}\mathbb{Z}^2$ with meshsize $\sqrt{2}\delta$. The difference of meshsize does not change anything while rotating $\delta \mathbb{Z}^2$ only affects the value of the constant $C_{\sigma}$. The functions $(f_{\Omega}^{(\xi,k)})_{k\geq 1}$, $\xi \in \{\pm, -, +\}$, remain the same.   
\end{remark}

Spin correlations of the critical Ising model with $+1$ boundary conditions are often easier to work with than those of the critical Ising model with Dobrushin boundary conditions. The following lemma bounds the latter correlations by the former and this bound will enable us to get sufficient control on the moments of $\sum_x \sigma_x$.

\begin{lemma} \label{lemma_EEpm_EEplus}
Let $k \geq 1$. For any $x_1, \dots, x_k \in \Omega_{\delta}$, $\vert \EEpm{\delta}[\prod_{j=1}^{k}\sigma_{x_j}] \vert \leq \EEplus{\delta}[\prod_{j=1}^k \sigma_{x_j}]$. 
\end{lemma}

\begin{proof}
The proof relies on the Edwards-Sokal coupling of Proposition \ref{prop_ES_coupling_no_field}. Let us first examine the case when the Ising model in $\Omega_{\delta}$ has $+$ boundary conditions. 
Let $k \geq 1$. We identify the vertices on the boundary of $\Omega_{\delta}$ with a single vertex $b_{+}$. For $x_1, \dots, x_k \in \Omega_{\delta}$, an FK-Ising percolation configuration on $\Omega_{\delta}$ induces a graph on $\{x_1, \dots , x_k, b_{+}\}$, where two vertices $x_i$ and $x_j$ are connected via an edge if there is a direct open path between them and a vertex $x_i$ is connected to $b_{+}$ if there is a direct open path between $x_i$ and the boundary of $\Omega_{\delta}$. Let us denote by $A_{x_1, \dots, x_k}$ the event that the graph on $\{x_1, \dots , x_k, b_{+}\}$ induced by the FK-Ising percolation configuration is such that each cluster contains an even number of vertices, except possibly the cluster containing $b_{+}$. Then, using the Edwards-Sokal coupling, it is easy to see that
\begin{equation*}
    \EEplus{\delta}[\prod_{j=1}^{k} \sigma_{x_{j}}] = \PP_{\delta}^{\operatorname{FK},\operatorname{wired}}(A_{x_{1}, \dots, x_{k}})
\end{equation*}
Note also that $A_{x_{1}, \dots, x_{k}}$ is an increasing event.

The Edwards-Sokal coupling can also be used in the case of Dobrushin boundary conditions. The coupling is as with $+$ boundary conditions except that the vertices of $\partial \Omega_{\delta}^{+}$ are wired together, the vertices of $\partial \Omega_{\delta}^{-}$ are wired together but $\partial \Omega_{\delta}^{+}$ is not wired to $\partial \Omega_{\delta}^{-}$. Moreover, the FK-Ising percolation configuration is in that case conditioned on the event $\{ \partial \Omega_{\delta}^{+} 	\nleftrightarrow \partial \Omega_{\delta}^{-}\}$ that there are no clusters containing both a vertex of $\partial \Omega_{\delta}^{+}$ and a vertex of $\partial \Omega_{\delta}^{-}$. We denote by $\PP_{\delta}^{2-w}(\cdot \vert \partial \Omega_{\delta}^{+} 	\nleftrightarrow \partial \Omega_{\delta}^{-} )$ the corresponding probability measure. We identify the vertices on $\partial \Omega_{\delta}^{+}$, respectively $\partial \Omega_{\delta}^{-}$, with a single vertex $b_{+}$, respectively $b_{-}$. As before, for $x_1, \dots, x_k \in \Omega_{\delta}$, such an FK-Ising percolation configuration on $\Omega_{\delta}$ induces a graph on $\{x_1, \dots , x_k, b_{+}, b_{-}\}$, where two vertices $x_i$ and $x_j$ are connected via an edge if there is a direct open path between them and a vertex $x_i$ is connected to $b_{+}$, respectively $b_{-}$, if there is a direct open path between $x_i$ and $\partial \Omega_{\delta}^{+}$, respectively $\partial \Omega_{\delta}^{-}$. Let us denote by $A_{x_{1}, \dots, x_{k}}^{\operatorname{(even)}}$, respectively $A_{x_{1}, \dots, x_{k}}^{\operatorname{(odd)}}$, the event that the graph induced on $\{x_1, \dots , x_k, b_{+}, b_{-}\}$ by the FK-Ising percolation configuration is such that each cluster not connected to $b_{+}$ or $b_{-}$ contains an even number of vertices and that the cluster of $b_{-}$ contains an even, respectively odd, number of vertices. No restrictions are imposed on the cluster connected to $b_{+}$. Then, it is easy to see that
\begin{equation*}
    \EEpm{\delta}[\prod_{j=1}^{k} \sigma_{x_{j}}] = \PP_{\delta}^{2-w}(A_{x_{1}, \dots, x_{k}}^{\operatorname{(even)}} \vert \partial \Omega_{\delta}^{+} 	\nleftrightarrow \partial \Omega_{\delta}^{-}) - \PP_{\delta}^{2-w}(A_{x_{1}, \dots, x_{k}}^{\operatorname{(odd)}} \vert \partial \Omega_{\delta}^{+} 	\nleftrightarrow \partial \Omega_{\delta}^{-}).
\end{equation*}
This yields that $\vert \EEpm{\delta}[\prod_{j=1}^{k} \sigma_{x_{j}}] \vert \leq \PP_{\delta}^{2-w}(A_{x_{1}, \dots, x_{k}}^{\operatorname{(even)}} \cup A_{x_{1}, \dots, x_{k}}^{\operatorname{(odd)}}  \vert \partial \Omega_{\delta}^{+} 	\nleftrightarrow \partial \Omega_{\delta}^{-})$. Now, since $A_{x_{1}, \dots, x_{k}}^{\operatorname{(even)}} \cup A_{x_{1}, \dots, x_{k}}^{\operatorname{(odd)}}$ is an increasing event and $\{\partial \Omega_{\delta}^{+} 	\nleftrightarrow \partial \Omega_{\delta}^{-}\}$ is a decreasing event, we have that, by the FKG inequality (Proposition \ref{prop_FKG_FK}),
\begin{align*}
    \PP_{\delta}^{2-w}(A_{x_{1}, \dots, x_{k}}^{\operatorname{(even)}} \cup A_{x_{1}, \dots, x_{k}}^{\operatorname{(odd)}}  \vert \partial \Omega_{\delta}^{+} 	\nleftrightarrow \partial \Omega_{\delta}^{-}) &= \frac{\PP_{\delta}^{2-w}((A_{x_{1}, \dots, x_{k}}^{\operatorname{(even)}} \cup A_{x_{1}, \dots, x_{k}}^{\operatorname{(odd)}}) \cap \{\partial \Omega_{\delta}^{+} 	\nleftrightarrow \partial \Omega_{\delta}^{-}\})}{\PP_{\delta}^{2-w}(\partial \Omega_{\delta}^{+} 	\nleftrightarrow \partial \Omega_{\delta}^{-})} \\
    &\leq \frac{\PP_{\delta}^{2-w}(A_{x_{1}, \dots, x_{k}}^{\operatorname{(even)}} \cup A_{x_{1}, \dots, x_{k}}^{\operatorname{(odd)}})\PP_{\delta}^{2-w}(\partial \Omega_{\delta}^{+} 	\nleftrightarrow \partial \Omega_{\delta}^{-})}{\PP_{\delta}^{2-w}(\partial \Omega_{\delta}^{+} 	\nleftrightarrow \partial \Omega_{\delta}^{-})} \\
    &= \PP_{\delta}^{2-w}(A_{x_{1}, \dots, x_{k}}^{\operatorname{(even)}} \cup A_{x_{1}, \dots, x_{k}}^{\operatorname{(odd)}}).
\end{align*}
Moreover, by Proposition \ref{prop_comparison_BC_FK}, $\PP_{\delta}^{2-w}(A_{x_{1}, \dots, x_{k}}^{\operatorname{(even)}} \cup A_{x_{1}, \dots, x_{k}}^{\operatorname{(odd)}}) \leq \PP_{\delta}^{\operatorname{FK},\operatorname{wired}}(A_{x_{1}, \dots, x_{k}}^{\operatorname{(even)}} \cup A_{x_{1}, \dots, x_{k}}^{\operatorname{(odd)}})$. Since when $\partial \Omega_{\delta}^{+}$ and $\partial \Omega_{\delta}^{-}$ are wired together, $A_{x_{1}, \dots, x_{k}}^{\operatorname{(even)}} \cup A_{x_{1}, \dots, x_{k}}^{\operatorname{(odd)}}$ and $A_{x_{1}, \dots, x_{k}}$ are the same events, this implies that $\vert \EEpm{\delta}[\prod_{j=1}^{k} \sigma_{x_{j}}] \vert \leq \PP_{\delta}^{\operatorname{FK},\operatorname{wired}}(A_{x_{1}, \cdots, A_{x_{k}}})$, from which we deduce the lemma.
\end{proof}

The next lemma provides a bound on $k$-point spin correlations of the critical Ising model with $+1$ boundary conditions. It is certainly well-known but for completeness, we prove it below.

\begin{lemma} \label{lemma_bound_EEplus}
There exists a constant $C>0$ such that for any $k \geq 1$ and any $x_1, \dots , x_k \in \Omega_{\delta}$, 
\begin{equation*}
    \EEplus{\delta}[\prod_{j=1}^{k}\sigma_{x_j}] \leq C^k\delta^{\frac{k}{8}}\exp\bigg( \frac{1}{8}\sum_{j=1}^{k} \log \frac{1}{\frac{1}{2}L(x_j) \wedge \text{d}(x_j, \partial \Omega_{\delta})} \bigg)
\end{equation*}
where for each $j$, $L(x_j) = \min_{p:p \neq j} \vert x_p - x_j \vert$. Above, for $x \in \Omega_{\delta}$, $\text{d}(x, \partial \Omega_{\delta})$ is the Euclidean distance between $x$ and $\partial \Omega_{\delta}$.
\end{lemma}

\begin{proof}
The proof of this lemma relies on the Edwards-Sokal coupling and the Markov property of FK-Ising percolation. We note that similar ideas appear in the proof of \cite[Proposition~3.10]{MourratIsing}. Let $k \geq 1$ and let $x_1, \dots , x_k \in \Omega_{\delta}$. For each $j \in \{1, \dots, k\}$, set $L(x_j) := \frac{1}{2}\min_{p: p \neq j}\vert x_j - x_p \vert$ and $d(x_j):=d(x_j, \partial \Omega_{\delta})$. We then define $\Lambda(x_j)$ to be the Euclidean square of side-length $\frac{1}{\sqrt{2}}\times (d(x_j) \wedge L(x_j))$ centered at $x_j$. Observe that $\Lambda(x_j) \subset \Omega$. Using that $\EEplus{\delta}[\sigma_{x_1}\dots\sigma_{x_{k}}]$ can expressed in terms of connection events in a FK-Ising percolation configuration (see the proof of Lemma \ref{lemma_EEpm_EEplus}) and the Markov property of FK-Ising percolation, we obtain that
\begin{equation*}
    \EEplus{\delta}[\sigma_{x_1}\dots\sigma_{x_{k}}] \leq \prod_{j=1}^{k} \PP_{\Lambda_{\delta}(x_j)}^{\operatorname{FK},\operatorname{wired}}(x_j \leftrightarrow \partial \Lambda(x_j)),
\end{equation*}
where $ \Lambda_{\delta}(x_j) := \Lambda(x_j) \cap \delta\mathbb{Z}^2$. By \cite[Lemma~5.4]{RSWbounds}, there exists a constant $C>0$ such that for any $j$, $\PP_{\Lambda(x_j)}^{w}(x_j \leftrightarrow \partial \Lambda(x_j)) \leq C \delta^{1/8}(d(x_j) \wedge L(x_j))^{-1/8}$. This concludes the proof of the lemma.
\end{proof}

\section{Scaling limit of the near-critical interface: proof of Theorem \ref{theorem_magnetic_SLE3}} \label{sec_proof_near_critical}

In this section, we prove Theorem \ref{theorem_magnetic_SLE3}. We give the proof for $H_x=C_{\sigma}^{-1}h\delta^{15/8}$ where $h \in \mathbb{R}$ and $C_{\sigma}$ is the constant of Theorem \ref{theorem_spin_correlations}. The extension to space-varying external fields can easily be obtained, as long as the function $h: \Omega \to \mathbb{R}$ is bounded and Lipschitz. Indeed, this last property guarantees that discrete sums approximate continuum integrals in a sufficiently nice way, see \cite{Discrete_analysis}, which makes the generalization of the proof below straightforward.

To lighten notations, we set $\PP_{\delta}^{\pm}:= \PP_{\Omega_{\delta}}^{\pm}$ and $\PP_{\delta,h}^{\pm}:= \PP_{\Omega_{\delta,H}}^{\pm}$. Similarly, when $(\delta_n)_n$ is a sequence converging to $0$, we set $\Omega_n:=\Omega_{\delta_n}$, $\PP_{n}^{\pm}:=\PP_{\Omega_{\delta_{n}}}^{\pm}$ and $\PP_{n,h}^{\pm}:= \PP_{\Omega_{\delta_{n},H}}^{\pm}$. We will also often write $x \in \Omega_{\delta}$ instead of $x \in \Int(\Omega_{\delta})$.

\begin{proof}[Proof of Theorem \ref{theorem_magnetic_SLE3}]
The proof consists of two main steps. First, in Proposition \ref{prop_tightness}, we establish tightness of $(\gamma_{\delta})_{\delta}$ under $(\PP_{\delta,h}^{\pm})_{\delta}$ in the topology \ref{topo_4}. Then, in Proposition \ref{prop_characterization}, we prove a characterization of the limiting law of any weakly convergent subsequence obtained as a consequence of the tightness of $(\gamma_{\delta})_{\delta}$ under $(\PP_{\delta,h}^{\pm})_{\delta}$. This yields convergence in law of $(\gamma_{\delta})_{\delta}$ under $(\PP_{\delta,h}^{\pm})_{\delta}$ in the topology \ref{topo_4}. As explained in Section \ref{sec_topology_curves}, this also implies tightness and convergence of $(\gamma_{\delta})_{\delta}$ under $(\PP_{\delta,h}^{\pm})_{\delta}$ in the topologies \ref{topo_1}--\ref{topo_3} and the limiting laws in each topology agree.

Both to prove tightness and characterize the limiting law, we are going to use the fact that for each $\delta>0$, the law of $\gamma_{\delta}$ under $\PP_{\delta,h}^{\pm}$ is absolutely continuous with respect to that of $\gamma_{\delta}$ under $\PP_{\delta}^{\pm}$ with Radon-Nikodym given by
\begin{equation}\label{RN_delta}
    F_{\delta}(\gamma_{\delta}):= \frac{\der \PP_{\delta,h}^{\pm}}{\der \PP_{\delta}^{\pm}}(\gamma_{\delta}) =\frac{1}{\mathcal{Z}_{\delta}(h)}\EE_{\delta}^{\pm}\bigg[ \exp\bigg(C_{\sigma}^{-1}h\delta^{15/8}\sum_{x \in \Omega_{\delta}} \sigma_x \bigg) \vert \gamma_{\delta} \bigg].
\end{equation}
Tightness of $(\gamma_{\delta})_{\delta}$ under $(\PP_{\delta,h}^{\pm})_{\delta}$ in the topology \ref{topo_4} in fact follows from the uniform boundedness of $(F_{\delta}(\gamma_{\delta}))_{\delta}$ in $L^p$ ($p\ge 1$), which is established in the proposition below.

\begin{proposition}[tightness] \label{prop_tightness}
    For any $p \geq 1$, $\sup_{\delta > 0}\EE_{\delta}^{\pm}[F_{\delta}(\gamma_{\delta})^p] < \infty$. In particular, since $(\gamma_{\delta})_{\delta}$ under $(\PP_{\delta}^{\pm})_{\delta}$ is tight in the topology \ref{topo_4}, $(\gamma_{\delta})_{\delta}$ under $(\PP_{\delta,h}^{\pm})_{\delta}$ is also tight in the topology \ref{topo_4}.
\end{proposition}

We postpone the proof of Proposition \ref{prop_tightness} to Section \ref{sec_tightness} and now state our characterization of the limiting law of weakly convergent subsequences.

\begin{proposition}[characterization of the limiting law] \label{prop_characterization}
    There exists a unique law $\PP_{h}^{(\Omega,a,b)}$ such that the following holds.
    The Radon-Nikodym derivative of $\PP_{h}^{(\Omega,a,b)}$ with respect to $\PP_{\operatorname{SLE}_3}^{(\Omega,a,b)}$ is given by
    \begin{equation} \label{RN_infty_continuum}
        \frac{\der \PP_{h}^{(\Omega,a,b)}}{\der \PP_{\operatorname{SLE}_3}^{(\Omega,a,b)}}(\gamma)= \frac{1}{\mathcal{Z}_h(\Omega)}\bigg(\sum_{k \geq 0}\frac{h^k}{k!}\int_{\Omega_L^k} f_{L}^{(+,k)}(z_1, \dots, z_k) \prod_{j=1}^k dz_j\bigg)\bigg(\sum_{k \geq 0}\frac{h^k}{k!}\int_{\Omega_{R}^k} f_{R}^{(-,k)}(z_1,\dots,z_k)\prod_{j=1}^{k}dz_j\bigg)
    \end{equation}
    and, for any sequence $(\delta_n)_n$ such that $(\PP_{n,h}^{\pm})_n$ converges weakly in the topology \ref{topo_4}, there exists a further subsequence $(\delta_{n_k})_k$ such that $\PP_{n_k,h}^{\pm} \to \PP_h^{(\Omega,a,b)}$ weakly as $k \to \infty$ in the topology \ref{topo_4}. Above, $\Omega_R$, respectively $\Omega_L$, denotes the connected component of $\Omega \setminus \gamma$ on the right, respectively left, of $\gamma$ and the functions $f_{L}^{(+,k)}:= f_{\Omega_L}^{(+,k)}$ and $f_{R}^{(-,k)}:= f_{\Omega_R}^{(-,k)}$ are as in Theorem \ref{theorem_spin_correlations}.
\end{proposition}

We will give the proof of Proposition \ref{prop_characterization} in Section \ref{sec_proof_characterization}. In view of Proposition \ref{prop_tightness} and Proposition \ref{prop_characterization}, the convergence part of Theorem \ref{theorem_magnetic_SLE3} follows from Prokhorov's theorem. It remains to prove the equality \eqref{RN_time_theorem_intro} for the Radon-Nikodym derivative of $\PP_{h}^{(\Omega,a,b)}$ with respect to $\PP_{\operatorname{SLE}_3}^{(\Omega,a,b)}$ when restricted to $\sigma(\gamma(s): 0 \leq s \leq t)$. This is the content of the next proposition, whose proof is postponed to Section \ref{sec_proof_characterization}. This proposition concludes the proof of Theorem \ref{theorem_magnetic_SLE3}.

\begin{proposition}\label{proposition_RN_time}
    Let $t \geq 0$. The Radon-Nikodym derivative of $\PP_{h}^{(\Omega,a,b)}$ with respect to $\PP_{\operatorname{SLE}_3}^{(\Omega,a,b)}$ restricted to the $\sigma$-field $\sigma(\gamma(s): 0 \leq s \leq t)$ is given by
    \begin{equation} \label{RN_time_statement}
       \left. \frac{\der \PP_{h}^{(\Omega,a,b)}}{\der \PP_{\operatorname{SLE}_3}^{(\Omega,a,b)}}(\gamma)\right|_{ \sigma(\gamma(s): 0 \leq s \leq t)} =\frac{1}{\mathcal{Z}_h(\Omega)}\sum_{k \geq 0}\frac{h^k}{k!}\int_{\Omega_t^k}f_{t}^{(\pm,k)}(z_1,\dots,z_k) \prod_{j=1}^k dz_j
    \end{equation}
    where we have set $\Omega_t:= \Omega \setminus \gamma([0,t])$ and the functions $f_{t}^{(\pm,k)}:= f_{\Omega_t}^{(\pm,k)}$ are as in Theorem \ref{theorem_spin_correlations}.
\end{proposition}
\end{proof}

\begin{remark}
    Let us note that a priori, Proposition \ref{proposition_RN_time} could be derived from Proposition \ref{prop_characterization} by computing the conditional expectation of the random variable on the right-hand side of \eqref{RN_infty_continuum} given $\sigma(\gamma(s): 0 \leq s \leq t)$. However, it is not clear how to do this. Instead, to establish Proposition \ref{proposition_RN_time}, we obtain the right-hand side of \eqref{RN_time_statement} by computing the scaling limit of the corresponding discrete quantity. This thus makes the statement of Theorem \ref{theorem_magnetic_SLE3} stronger than the one we would obtain by only proving Proposition \ref{prop_characterization}.
\end{remark}

\subsection{Tightness of $(\PP_{\delta}^{\pm})_{\delta}$: proof of Proposition \ref{prop_tightness}} \label{sec_tightness}

In this section, we prove Proposition \ref{prop_tightness}. Recall the definition \eqref{RN_delta} of $F_{\delta}(\gamma_{\delta})$.

\begin{proof}[Proof of Proposition \ref{prop_tightness}]
We give the proof in the case where $h: \Omega \to \mathbb{R}$ is a bounded and Lipchitz function: compared to the case when $h$ is a constant, this requires a few modifications that we prefer to detail. We first show that for any $p \in [1,\infty)$, $\sup_{\delta} \EE_{\delta}^{\pm}[F_{\delta}(\gamma_{\delta})^p] < \infty$. Fix $p \in [1,\infty)$ and let $\delta > 0$. Set 
\begin{equation*}
    S_{\delta}(\gamma_{\delta}):=\EE_{\delta}^{\pm}\bigg[\exp(\delta^{15/8}\sum_{x \in \Omega_{\delta}} \tilde h_{\delta}(x)\sigma_x) \vert \gamma_{\delta} \bigg]
\end{equation*}
with $\tilde h_{\delta}(x):=C_{\sigma}^{-1}h_{\delta}(x)$. We are going to obtain uniform bounds on $\mathcal{Z}_{\delta}(h)^{-p}$ and $\EE_{\delta}^{\pm}[S_{\delta}(\gamma_{\delta})^p]$ separately. For $\EE_{\delta}^{\pm}[S_{\delta}(\gamma_{\delta})^p]$, using first Jensen's inequality for the conditional expectation and then H\"older's inequality, we have that
\begin{align} \label{ineq_Sn_p}
    \EE_{\delta}^{\pm}[S_{\delta}(\gamma_{\delta})^p] &\leq \EE_{\delta}^{\pm}\bigg[ \exp\bigg( p\delta^{\frac{15}{8}} \sum_{x \in \Omega_{\delta}} \tilde h_{\delta}(x)\sigma_x \bigg) \bigg] \nonumber \\
    &=\EE_{\delta}^{\pm}\bigg[ \exp\bigg( p\delta^{\frac{15}{8}} \sum_{x \in \Omega_{\delta}} (\tilde h_{\delta}(x)-\vert \tilde h_{\delta}(x) \vert) \sigma_x \bigg) \exp\bigg( p\delta^{\frac{15}{8}} \sum_{x \in \Omega_{\delta}} \vert \tilde h_{\delta}(x) \vert \sigma_x \bigg)\bigg] \nonumber \\
    &\leq \EE_{\delta}^{\pm}\bigg[ \exp\bigg( 2p\delta^{\frac{15}{8}} \sum_{x \in \Omega_{\delta}} (\tilde h_{\delta}(x)-\vert \tilde h_{\delta}(x) \vert) \sigma_x \bigg)\bigg]^{1/2}\EE_{\delta}^{\pm}\bigg[ \exp\bigg( 2p\delta^{\frac{15}{8}} \sum_{x \in \Omega_{\delta}} \vert \tilde h_{\delta}(x) \vert \sigma_x \bigg) \bigg]^{1/2}.
\end{align}
Setting $M = C_{\sigma}^{-1}\max_{x \in \Omega} \vert h(x) \vert$ and applying Proposition \ref{prop_comparison_BC_Ising} with the increasing function $\sigma \mapsto \sum_x \vert \tilde h_{\delta}(x) \vert \sigma_x$, we obtain that
\begin{align} \label{ineq_pm_plus}
    \EE_{\delta}^{\pm}\bigg[ \exp\bigg( 2p\delta^{\frac{15}{8}} \sum_{x \in \Omega_{\delta}} \vert \tilde h_{\delta}(x) \vert \sigma_x \bigg) \bigg] \leq \EE_{\delta}^{+}\bigg[ \exp\bigg( 2p\delta^{\frac{15}{8}} \sum_{x \in \Omega_{\delta}} \vert \tilde h_{\delta}(x) \vert \sigma_x \bigg) \bigg] 
    \leq \EE_{\delta}^{+}\bigg[ \exp\bigg( 2pM\delta^{\frac{15}{8}} \sum_{x \in \Omega_{\delta}} \sigma_x \bigg) \bigg]
\end{align}
where the last inequality follows by expanding the exponential and noticing that for any $k\geq 1$, $\EE_{\delta}^{+}[\prod_{j \leq k} \sigma_j] \geq 0$. The proof of \cite[Proposition~3.5]{GarbanMagnetization} shows that
\begin{equation} \label{ineq_log_EE}
    \log \EE_{\delta}^{+}\bigg[ \exp\bigg( 2pM\delta^{\frac{15}{8}} \sum_{x \in \Omega_{\delta}} \sigma_x \bigg) \bigg] \leq 2pM\delta^{\frac{15}{8}}\EE_{\delta}^{+}\big[\sum_{x \in \Omega_{\delta}} \sigma_x \big] + 2p^2M^{2}\delta^{\frac{15}{4}}\EE_{\delta}^{+}\big[ \big( \sum_{x \in \Omega_{\delta}} \sigma_x - \EE_{\delta}^{+}\big[\sum_{x \in \Omega_{\delta}} \sigma_x \big]\big)^2 \big].
\end{equation}
Since both $\delta^{\frac{15}{8}}\EE_{\delta}^{+}[\sum_x \sigma_x]$ and $\delta^{\frac{15}{4}}\EE_{\delta}^{+}[(\sum_x \sigma_x)^2]$ converge as $\delta \to 0$, see e.g. \cite{GarbanMagnetization} or Lemma \ref{lemma_moments_infty} below, it follows from the above inequality that there exist $\delta_0 >0$ and $c_1, c_2 > 0$ (depending only on $\Omega$) such that for any $0 <\delta < \delta_0$,
\begin{equation*}
    \EE_{\delta}^{+}\bigg[ \exp\bigg( 2pM\delta^{\frac{15}{8}} \sum_{x \in \Omega_{\delta}} \sigma_x \bigg) \bigg] \leq \exp(2pMc_1 + 2p^2M^2c_2).
\end{equation*}
On the other hand, using the spin flip symmetry of the Ising model, we have that
\begin{equation*}
    \EE_{\delta}^{\pm}\bigg[ \exp\bigg( 2p\delta^{\frac{15}{8}} \sum_{x \in \Omega_{\delta}} (\tilde h_{\delta}(x)-\vert \tilde h_{\delta}(x) \vert) \sigma_x \bigg)\bigg] = \EE_{\delta}^{\mp}\bigg[ \exp\bigg( 2p\delta^{\frac{15}{8}} \sum_{x \in \Omega_{\delta}} (\vert \tilde h_{\delta}(x) \vert - \tilde h_{\delta}(x)) \sigma_x \bigg)\bigg]
\end{equation*}
where under $\EE_{\delta}^{\mp}$, the Ising spin configuration $\sigma$ in $\Omega_{\delta}$ has $-1$, respectively $+1$, boundary conditions on $\partial \Omega_{\delta}^{+}$, respectively $\partial \Omega_{\delta}^{-}$. Note that the function $x \mapsto \vert \tilde h_{\delta}(x) \vert - \tilde h_{\delta}(x)$ is non-negative and bounded by $2M$. Therefore, we can proceed as in \eqref{ineq_pm_plus} and \eqref{ineq_log_EE} to obtain that, with $\delta_0, c_1$ and $c_2$ as above, for any $0 < \delta < \delta_0$,
\begin{equation*}
    \EE_{\delta}^{\mp}\bigg[ \exp\bigg( 2p\delta^{\frac{15}{8}} \sum_{x \in \Omega_{\delta}} (\tilde h_{\delta}(x)-\vert \tilde h_{\delta}(x) \vert) \sigma_x \bigg)\bigg] \leq \exp(4pMc_1 + 8p^2M^2c_2).
\end{equation*}
Going back to the inequality \eqref{ineq_Sn_p}, this yields that for any $0 < \delta < \delta_0$,
\begin{equation*}
    \EE_{\delta}^{\pm}[S_{\delta}(\gamma_{\delta})^p] \leq \exp(3pMc_1 + 5p^2M^2c_2).
\end{equation*}
It now remains to obtain a lower bound on $\mathcal{Z}_{\delta}(h)$. For this, we simply observe that Jensen's inequality implies that $\mathcal{Z}_{\delta}(h) \geq \exp(\delta^{\frac{15}{8}}\EE_{\delta}^{\pm}[\sum_x \tilde h_{\delta}(x) \sigma_x])$. Moreover, the results of Section \ref{subsec_controlling_spin_field} show that $\delta^{\frac{15}{8}}\EE_{\delta}^{\pm}[\sum_x \tilde h_{\delta}(x) \sigma_x]$ converges to a finite limit as $\delta \to 0$. Note that this limit may be negative as the model has Dobrushin boundary conditions and $h$ does not have a constant sign. We obtain from this that there exist $\delta_1 > 0$ and $c>0$ such that for any $0 < \delta < \delta_1$, $\mathcal{Z}_{\delta}(h) \geq \exp(-c)$, which concludes the proof of the uniform boundedness of $\EE_{\delta}[F_{\delta}(\gamma_{\delta})^p]$.
Tightness of $(\gamma_{\delta})_{\delta}$ under $(\PP_{\delta,h}^{\pm})_{\delta}$ in the topology \ref{topo_4} now follows from this uniform boundedness and tightness of $(\gamma_{\delta})_{\delta}$ under $(\PP_{\delta}^{\pm})_{\delta}$. Indeed, using Cauchy-Schwarz inequality, for any $K \subset X(\mathbb{C})$, $\PP_{\delta,h}^{\pm}(K)$ can be bounded by a constant, that is uniform in $\delta$, times $\PP_{\delta}^{\pm}(K)^{1/2}$. This concludes the proof of Proposition \ref{prop_tightness}.
\end{proof}

\subsection{Characterization of the limiting law: proof of Propositions \ref{prop_characterization} and \ref{proposition_RN_time}} \label{sec_proof_characterization}

In this section, we prove Proposition \ref{prop_characterization} and Proposition \ref{proposition_RN_time}. We start with Proposition \ref{prop_characterization}.

\begin{proof}[Proof of Proposition \ref{prop_characterization}]
Let $(\delta_n)_n$ be a sequence such that $(\PP_{n,h}^{\pm})_n$ converges weakly as $n \to \infty$ in the topology \ref{topo_4}. Note that such a sequence exists by Proposition \ref{prop_tightness}. By Theorem \ref{theorem_SLE3}, we also know that $(\PP_{n}^{\pm})_n$ converges weakly in the topologies \ref{topo_1}--\ref{topo_4} to $\PP_{\operatorname{SLE}_3}^{(\Omega,a,b)}$. The first step of the proof is to use Skorokhod representation theorem to define a sequence $(\gamma_n)_n$ and a random curve $\gamma$ on the same probability space $(S,\mathcal{F},\PP)$ in such a way that for each $n$, $\gamma_n$ has law $\PP_{n}^{\pm}$, $\gamma$ has law $\PP_{\operatorname{SLE}_3}^{(\Omega,a,b)}$ and $\PP$-almost surely, $\gamma_n \to \gamma$ as $n \to \infty$ in the topologies \ref{topo_1}--\ref{topo_4}.  

Now, the random variables $(F_{n}(\gamma_n))_n$ defined via \eqref{RN_delta} (with $\delta=\delta_n$) can all be defined on $(S, \mathcal{F}, \PP)$. To prove Proposition \ref{prop_characterization}, we are going to show that there exists a subsequence $(\delta_{m})_m \subset (\delta_n)_n$ such that $F_{m}(\gamma_{m})$ converges in $L^1(\PP)$ to the random variable of the right-hand side of \eqref{RN_infty_continuum}. As Proposition \ref{prop_tightness} applied with $p=2$ shows that $(F_{n}(\gamma_n))_n$ is uniformly bounded in $L^2(\PP)$, it actually suffices to show that there exists a subsequence $(\delta_{m})_m \subset (\delta_n)_n$ such that $F_{m}(\gamma_m)$ converges $\PP$-almost surely to the random variable on the right-hand side of \eqref{RN_infty_continuum}. For this, we first observe that by the Markov property of the Ising model and the definition of $\gamma_n$, for any $n \geq 1$, $F_n(\gamma_n)$ can be decomposed as
\begin{align} \label{decomposition_RN_infty}
    F_{n}(\gamma_{n}) = &\frac{1}{\mathcal{Z}_{n}(h)}\exp \big(C_{\sigma}^{-1}h\delta_{n}^{\frac{15}{8}}(\vert V_L(\gamma_{n})\vert-\vert V_R(\gamma_{n}) \vert)\big) \nonumber \\
    &\times \EE_{n,L}^{+}\bigg[ \exp \bigg( C_{\sigma}^{-1}h\delta_{n}^{15/8}\sum_{x \in \Omega_{n,L}}\sigma_x \bigg)\bigg]\EE_{n,R}^{-}\bigg[ \exp \bigg( C_{\sigma}^{-1}h\delta_{n}^{15/8}\sum_{x \in \Omega_{n,R}}\sigma_x \bigg)\bigg]
\end{align}
where $\Omega_{n,L}$, respectively $\Omega_{n,R}$, is the connected component of $\Omega_{n} \setminus V(\gamma_{n})$ on the left, respectively right, of $\gamma_{n}$ and, conditionally on $\gamma_n$, $\EE_{n,L}^{+}$, respectively $\EE_{n,R}^{-}$, denotes the expectation with respect to the critical Ising model with $+$, respectively $-$, boundary conditions in $\Omega_{n,L}$, respectively $\Omega_{n,R}$. Here, $V(\gamma_{n})$ is the set of vertices adjacent to $\gamma_{n}$, where a vertex $x \in \Omega_{n}$ is said to be adjacent to $\gamma_{n}$ if $\gamma_{n}$ traverses an edge of the octagon whose center is $x$. $V_{L}(\gamma_{n})$, respectively $V_{R}(\gamma_{n})$, denotes the set of vertices that are adjacent to $\gamma_{n}$ and on its left, respectively right. For $j \in \{L,R\}$, $\vert V_j(\gamma_n) \vert$ is the number of vertices in $V_j(\gamma_n)$. Note that the $\PP$-almost sure convergence of $\gamma_n$ to $\gamma$ in the topologies \ref{topo_1}--\ref{topo_4} implies that almost surely, $\hat \Omega_{n,L} \to \Omega_{L}$ and $\hat \Omega_{n,R} \to \Omega_{R}$ in the Carath\'eodory topology, where $\Omega_{L}$, respectively $\Omega_{R}$, is the connected component of $\Omega \setminus \gamma$ on the left, respectively right, of $\gamma$.

Expanding the two exponentials in \eqref{decomposition_RN_infty}, the existence of a subsequence along which almost sure convergence of $F_{n}(\gamma_{n})$ holds is a consequence of the following lemmas. Their proofs are given in Section \ref{subsec_invisibility_interface} and Section \ref{subsec_controlling_spin_field}.

\begin{lemma} \label{lemma_thinness_gamma}
   There exists a subsequence $(\delta_m)_m \subset (\delta_{n})_n$ such that, almost surely, as $m \to \infty$, $\delta_m^{\frac{15}{8}}(\vert V_{L}(\gamma_m)\vert-\vert V_{R}(\gamma_m)\vert) \to 0$. 
\end{lemma}

\begin{lemma} \label{lemma_RN_infty}
For $(\delta_m)_m$ as in Lemma \ref{lemma_thinness_gamma}, almost surely, for $\xi \in \{+,-\}$ with $j(\xi)=L$ if $\xi=+$ and $j(\xi)=R$ if $\xi=-1$,
\begin{align} \label{RN_infty_core}
    \lim_{m \to \infty} \sum_{k=0}^{\infty} \frac{C_{\sigma}^{-k}h^k}{k!}\delta_m^{\frac{15k}{8}} \EE_{m,j(\xi)}^{\xi} \big[\big( \sum_{x \in \Omega_{m,j(\xi)}}\sigma_x\big)^k \big]
    = \sum_{k=0}^{\infty} \frac{h^k}{k!} \int_{\Omega_{j(\xi)}^k} f_{j(\xi)}^{(\xi,k)}(x_1, \dots ,x_k) \prod_{j=1}^{k} dx_j < \infty
\end{align}
where the functions $f_{j(\xi)}^{(\xi,k)}:=f_{\Omega_{j(\xi)}}^{(\xi,k)}$ are as in Theorem \ref{theorem_spin_correlations}.
\end{lemma}

Proposition \ref{prop_characterization} directly follows from Lemma \ref{lemma_thinness_gamma} and Lemma \ref{lemma_RN_infty}.
\end{proof}

Let us now turn to the proof of Proposition \ref{proposition_RN_time}; this follows a very similar strategy to the proof of Proposition \ref{prop_characterization}.

\begin{proof}[Proof of Proposition \ref{proposition_RN_time}]
As in the proof of Proposition \ref{prop_characterization} and with $(\delta_n)_n$ as there, we start by coupling $(\gamma_n)_n$ under $(\PP_{n}^{\pm})_n$ and $\gamma$ under $\PP_{\operatorname{SLE}_3}^{(\Omega,a,b)}$ on a common probability space $(S,\mathcal{F},\PP)$. We then parametrize the curves $(\gamma_n)_n$ by the half-plane capacity of their conformal images $(\gamma_n^{\HH})_n$, where for each $n$, $\gamma_n^{\HH}$ is defined as in Section \ref{sec_topology_curves}. For $n \geq 1$, denote by $\mathcal{F}_{n,m}$ the $\sigma$-algebra generated by $\gamma_{n}([0,m])$. Here, for $m \in \mathbb{N}$, $\gamma_{n}([0,m])$ is the curve obtained by following $\gamma_{n}$ for $m$ steps, where a step corresponds to $\gamma_{n}$ traversing edges of the square-octagon lattice to join two edges of $\delta_{n} \mathbb{Z}^2$ that surrounds the same face of $\delta_n \mathbb{Z}^2$. If $m$ is larger than the total number of steps taken by $\gamma_{n}$, we set $\gamma_{n}([0,m]):= \gamma_{n}$. For $n \geq 1$ and $t \geq 0$, we then define the following stopping time for the filtration $(\mathcal{F}_{n,m})_m$:
\begin{equation*}
    t_{n} := \inf\{ m \geq 0: \operatorname{hcap}(\phi_{n}(\gamma_{n}([0,m]))) \geq t \}.
\end{equation*}
Accordingly, we let $\Omega_{n,t}$ be the discrete domain whose vertex set is $V(\Omega_{n}) \setminus V(\gamma_{n};t)$ and whose boundary $\partial \Omega_{n,t}$ is split into two subarcs $\partial \Omega_{n,t}^{+}:= \partial \Omega_{n,t}^{+} \cup V_{L}(\gamma_{n};t)$ and $\partial \Omega_{n,t}^{+}:= \partial \Omega_{n,t}^{+} \cup V_{R}(\gamma_{n};t)$. Here, $V(\gamma_{n};t)$ is the set of vertices adjacent to $\gamma_{n}([0,t_{n}])$. $V_{L}(\gamma_{n};t)$, respectively $V_{R}(\gamma_{n};t)$, denotes the set of vertices that are adjacent to $\gamma_{n}([0,t_{n}])$ and on its left, respectively right.

Since $\PP$-almost surely, $\gamma_{n} \to \gamma$, we have that, $\PP$-almost surely, in the Carath\'eodory topology,
\begin{equation*}
    (\hat \Omega_{n,t}; a_{n,t}, b_{n}) \to (\Omega_t;a_t,b) \quad \text{as } n \to \infty
\end{equation*}
where $a_{n,t}$ denotes the tip of $\gamma_{n}([0,t_{n}])$ and $\Omega_{t}$ is the connected component of $\Omega \setminus \gamma([0,t])$ whose boundary contains both $a_t:= \gamma(t)$ and $b$.

Observe now that, for $t \geq 0$ and $n \geq 1$, the Radon-Nikodym of the law of $\gamma_n$ under $\PP_{n,h}^{\pm}$ with respect to that of $\gamma_n$ under $\PP_{n}^{\pm}$ restricted to $\mathcal{F}_{n,t_n}$ is given by
\begin{equation} \label{RN_hnonzero_time}
    \left. F_{n}(\gamma_{n};t) := \frac{\der \PP_{n,h}^{\pm}}{\der \PPpmzero{n}}(\gamma_{n})\right|_{ \sigma(\gamma_n(s): 0 \leq s \leq t_n)} =  \frac{1}{\mathcal{Z}_{n}(h)}\EE_{n}^{\pm} \bigg[ \exp \bigg(C_{\sigma}^{-1} h\delta_{n}^{\frac{15}{8}}\sum_{x \in \Omega_{n}}\sigma_x \bigg) \vert \gamma_{n}([0,t_{n}]) \bigg].
\end{equation}
The family $(F_n(\gamma_n;t), t \geq 0)_n$ is well-defined on $(S,\mathcal{F},\PP)$. To prove Proposition \ref{proposition_RN_time}, we are going to show that there exists a subsequence $(\delta_m)_m \subset (\delta_n)_n$ such that for each $t \geq 0$, $F_{m}(\gamma_m;t)$ converges in $L^1(\PP)$ to the random variable on the right-hand side of \eqref{RN_time_statement}. As $(\PP_{m,h}^{\pm})_m$ converges weakly to $\PP_{h}^{(\Omega,a,b)}$ by Proposition \ref{prop_tightness} and Proposition \ref{prop_characterization}, this is indeed sufficient to establish Proposition \ref{proposition_RN_time}.

We first observe that by the Markov property of the Ising model and the definition of $\gamma_n$, almost surely, for any $t \geq 0$ and any $n \geq 1$,
\begin{align} \label{decomposition_RN_time}
    F_{n}(\gamma_{n};t) = \frac{1}{\mathcal{Z}_{n}(h)}\exp \big(C_{\sigma}^{-1}h\delta_{n}^{\frac{15}{8}}(\vert V_L(\gamma_{n};t)\vert-\vert V_R(\gamma_{n};t) \vert)\big) \EE_{n,t}^{\pm}\bigg[ \exp \bigg( C_{\sigma}^{-1}h\delta_{n}^{15/8}\sum_{x \in \Omega_{n,t}}\sigma_x \bigg)\bigg]
\end{align}
where conditionally on $\gamma_{n}([0,t_{n}])$, $\EEpmtime{n}{t}[\cdot]$ denotes the expectation with respect to the critical Ising model in $\Omega_{n,t}$ with $+$, respectively $-$, boundary conditions on $\partial \Omega_{n,t}^{+}$, respectively $\partial \Omega_{n,t}^{-}$. Above, for $j \in \{L,R\}$, $\vert V_j(\gamma_{n};t) \vert$ is the number of vertices in $V_j(\gamma_{n};t)$. Proposition \ref{prop_tightness} applied with $p=2$ and combined with Jensen's inequality shows that $(F_n(\gamma_n;t), t \geq 0)_n$ is uniformly bounded in $t$ and $n$ in $L^2(\PP)$. To prove Proposition \ref{proposition_RN_time}, it therefore suffices to show that there exists a subsequence $(\delta_m)_m \subset (\delta_n)_n$ such that for each $t \geq 0$, $F_{m}(\gamma_m;t)$ almost surely converges to the right-hand side of \eqref{RN_time_statement} as $m \to \infty$. Expanding the exponential in \eqref{decomposition_RN_time}, this follows from the two lemmas below, which are analogues of Lemma \ref{lemma_thinness_gamma} and Lemma \ref{lemma_RN_infty} and whose proofs are also postponed to Section \ref{subsec_invisibility_interface} and Section \ref{subsec_controlling_spin_field}.

\begin{lemma} \label{lemma_thinness_gamma_time}
   Let $(\delta_m)_m \subset (\delta_n)_n$ be as in Lemma \ref{lemma_thinness_gamma}. Then, almost surely, for any $t \geq 0$, as $m \to \infty$, $\delta_m^{\frac{15}{8}}(\vert V_{L}(\gamma_m;t)\vert-\vert V_{R}(\gamma_m;t)\vert) \to 0$.
\end{lemma}

\begin{lemma} \label{lemma_RN_time}
    Let $t \in (0,\infty)$ and let $(\delta_m)_m$ be as Lemma \ref{lemma_thinness_gamma_time}. Then, almost surely, 
    \begin{equation} \label{RN_time_core}
        \lim_{m \to \infty} \sum_{k=0}^{\infty}\frac{C_{\sigma}^{-k}h^k}{k!}\delta_m^{\frac{15k}{8}} \EEpmtime{m}{t} \big[\big( \sum_{x \in \Omega_{m,t}}\sigma_x\big)^k \big] = \sum_{k=0}^{\infty} \frac{h^k}{k!}\int_{\Omega_t^k} f_t^{(\pm, k)}(x_1, \dots ,x_k) \prod_{j=1}^{k} dx_j < \infty
    \end{equation}
    where the functions $f_t^{(\pm, k)}:=f_{\Omega_t}^{(\pm,k)}$ are as in Theorem \ref{theorem_spin_correlations}.
\end{lemma}

These two lemmas conclude the proof of Proposition \ref{proposition_RN_time}.
\end{proof}

\subsubsection{Vanishing of the magnetization of the interface} \label{subsec_invisibility_interface}

In this section, we prove Lemma \ref{lemma_thinness_gamma} and Lemma \ref{lemma_thinness_gamma_time}. Note that these lemmas do not follows from the $\PP$-almost sure convergence of $\hat \Omega_{n,t}$ to $\Omega_t$ in the Carath\'eodory topology. Indeed, this topology is not strong enough to guarantee convergence of the Minkowski dimension of $\partial \hat \Omega_{n,t}$ to that of $\partial \Omega_t$. Instead, Lemma \ref{lemma_thinness_gamma} and Lemma \ref{lemma_thinness_gamma_time} will be proved thanks to the following lemma, which provides an almost sure estimate on the size of $\vert V(\gamma_{n}) \vert$ along a subsequence, where $\vert V(\gamma_{n}) \vert$ is the number of vertices adjacent to $\gamma_{n}$.

\begin{lemma}\label{lemma_size_gamma_infty}
Let $\beta > 0$. There exists a subsequence $(\delta_m)_m \subset (\delta_{n})_n$ such that, almost surely, $\delta_m^{\frac{11}{8}+\beta}\vert V(\gamma_m) \vert \to 0$ as $m \to \infty$.
\end{lemma}

Lemma \ref{lemma_thinness_gamma} and Lemma \ref{lemma_thinness_gamma} now straightforwardly follow from Lemma \ref{lemma_size_gamma_infty}. Below, we only give the proof of Lemma \ref{lemma_thinness_gamma_time}, as that of Lemma \ref{lemma_thinness_gamma} is almost identical.

\begin{proof}[Proof of Lemma \ref{lemma_thinness_gamma_time}]
Observe that, almost surely, for any $t \geq 0$ and any $n \geq 1$, $-\vert V(\gamma_n;t)\vert \leq \vert V_L(\gamma_n;t) \vert - \vert V_R(\gamma_n;t) \vert \leq \vert V(\gamma_n;t) \vert$, where $\vert V(\gamma_n;t)\vert$ denotes the number of vertices in $V(\gamma_n;t)$. Moreover, almost surely, for any $n \geq 1$, $t \mapsto \vert V(\gamma_n;t) \vert$ is an increasing function. Therefore, almost surely, for any $t \geq 0$ and any $n \geq 1$, $-\vert V(\gamma_n) \vert \leq \vert V_L(\gamma_n;t) \vert - \vert V_R(\gamma_n;t) \vert \leq \vert V(\gamma_n) \vert$. The statement then follows directly from Lemma \ref{lemma_size_gamma_infty}.
\end{proof}

Let us now turn to the proof of Lemma \ref{lemma_size_gamma_infty}. For $L>0$ and $v \in \mathbb{C}$, we set $\Lambda_{L,\delta}(v) = \Lambda_L(v) \cap \delta \mathbb{Z}^2$, where $\Lambda_L(v)$ is the Euclidean square of side length $2L$ centered at $v$. For $0 < r < R$ and $v \in \Omega_{\delta}$, we say that there is an arm of sign $+$, respectively $-$, in the annulus $\Lambda_{R,\delta}(v) \setminus \Lambda_{r,\delta}(v)$ if there is a nearest neighbor path joining $\partial \Lambda_{r,\delta}(v)$ to $\partial \Lambda_{R,v}$ along which all vertices have spin $+1$, respectively $-1$. For $0 < r < R$ and $v \in \Omega_{\delta}$, we denote by $A_2(v;r,R)$ the event that there exist two arms of opposite signs crossing the annulus $\Lambda_{R,\delta}(v) \setminus \Lambda_{r,\delta}(v)$. We then define
\begin{equation*}
    \alpha_{\delta}(v;r,R) := \PP_{\delta}^{\pm}(A_2(v;r,R)), \quad \alpha_{L,\delta}(v;r,R) := \PP_{\Lambda_{L,\delta}(v)}^{\pm}(A_2(v;r,R))
\end{equation*}
where the boundary conditions on $\partial \Lambda_{L,\delta}(v)$ under $\PP_{\Lambda_{L,\delta}(v)}^{\pm}$ are $-1$ on the bottom side of $\Lambda_{L,\delta}(v)$ and $+1$ elsewhere (this choice is somewhat arbitrary but convenient).

\begin{proof}[Proof of Lemma \ref{lemma_size_gamma_infty}]
The strategy of the proof is to upper bound $\EEpm{n}[\vert V(\gamma_{n}) \vert]$ and show that, once multiplied by $\delta_{n}^{\frac{11}{8}+\beta}$, this expectation decays fast enough to apply Borel-Cantelli lemma to get almost sure convergence along a subsequence $(\delta_{m})_m$. Let $\beta > 0$. Let $\eta >0$ and set $A_{n}(\eta) := \{ x \in \Omega_{n}: d(x,\partial \Omega) < \eta \}$. We have that, for any $n$,
\begin{equation*}
    \EEpm{n}[\vert V(\gamma_{n}) \vert] = \sum_{x \in \Omega_{n}} \PPpmzero{n}(x \in V(\gamma_n)) \leq \sum_{x \in \Omega_{n}} \alpha_{n}(x;\delta_n,d(x,\partial \Omega_{n})).
\end{equation*}
For $x \in \Omega_{n}$, set $L(x) := \frac{1}{2}d(x,\partial \Omega)$ and note that $\Lambda_{L(x),\delta_{n}}(x) \subset \Omega_{n}$. By the Markov property of the Ising model, we have that
\begin{equation*}
    \alpha_{\delta}(x;\delta_n,d(x,\partial \Omega_{n})) \leq \alpha_{n}(x;\delta_n,\frac{1}{4}d(x,\partial \Omega_{n})) = \EE_{n}^{\pm}[\PP_{\Lambda_{L(x),\delta_n}(x)}^{\sigma_{| \partial \Lambda_{L(x)}(x)}}(A_2(x;\delta_n,\frac{1}{2}L(x)))]
\end{equation*}
where $\PP_{\Lambda_{L(x),\delta_n}(x)}^{\sigma_{| \partial \Lambda_{L(x)}(x)}}$ denotes the critical Ising measure in $\Lambda_{L(x),\delta_n}(x)$ with boundary conditions $\sigma_{| \partial \Lambda_{L(x)}(x)}$. Now, \cite[Corrolary~5.2]{armExponentsIsing} guarantees that, when $\beta = \beta_c$, changing the boundary conditions does not change by more than a multiplicative constant the probability of events depending on edges at macroscopic distance from the boundary. This implies that, almost surely,
\begin{equation*}
    \PP_{\Lambda_{L(x)}(x)}^{\sigma_{| \partial \Lambda_{L(x)}(x)}}(A_2(x;\delta_n,\frac{1}{2}L(x))) \leq 2\alpha_{L(x), \delta_n}(x;\delta_n,\frac{1}{2}L(x))) = 2\alpha_{L(x),\delta_n}(0;\delta_n,\frac{1}{2}L(x)))
\end{equation*}
where the last equality uses translation invariance of the critical Ising model. Moreover, the proof of \cite[Theorem~1.2]{armExponentsIsing} shows that, as $\delta_n \to 0$, uniformly in $x \in \Omega \setminus A(\eta)$, 
\begin{equation*}
    \alpha_{L(x),\delta_n}(0;\delta_n,\frac{1}{2}L(x))) \lesssim (\delta_n/L(x))^{7/8-\tilde \beta},
\end{equation*}
for arbitrary $\tilde \beta > 0$, see e.g. \cite[Equation~(5.2)]{armExponentsIsing}. Choose $\tilde \beta = \beta/2$. We thus obtain that there exist $C_{\eta},\delta_0 > 0$ such that for any $0<\delta < \delta_0$,
\begin{equation*}
    \sum_{x \in \Omega_{n} \setminus A_{n}(\eta)} \PP_{n}^{\pm}(x \in V(\gamma_n)) \leq C_{\eta}\eta^{-\frac{7}{8}+\frac{\beta}{2}}\delta_n^{-2-\frac{\beta}{2}+\frac{7}{8}}.
\end{equation*}
To estimate $\EEpm{n}[\vert V(\gamma_{n}) \vert]$, it now remains to deal with the sum over vertices in $A_{n}(\eta)$. The argument below goes along the same lines as the proof of \cite[Theorem~5.5]{GarbanPivotalMeasure}, except that we must make sure that we can get estimates in $\delta_n$ on the probability of a two-arm event that are uniform in space, even when points are close to the boundary of the domain. We divide $A_{n}(\eta)$ into $\eta^{-d_H(\partial \Omega)}$ squares of side length $\eta$, where $d_{H}(\partial \Omega)$ denotes the Hausdorff dimension of $\partial \Omega$. Among these squares, one touches $a$, call it $s_a$, and another touches $b$, call it $s_b$. For $j\geq 0$, we define the set $P_{j}(a) := \{ x \in s_a: 2^{j}\delta_n \leq d(x,a) \leq 2^{j+1}\delta_n \}$ and $P_{j}(b)$ is defined analogously with $a$ replaced by $b$. We then have that, for any $j \geq 0$,
\begin{align*}
    \sum_{x \in P_{j}(a)} \PPpmzero{n}(x \in V(\gamma_n)) \leq \sum_{x \in P_{j}(a)} \alpha_{\Omega_{n}}(x;\delta_n, 2^j\delta_n).    
\end{align*}
By the Markov property of the Ising model and using again \cite[Corollary~5.2]{armExponentsIsing}, we have that there exist $n_0 \geq 1$ and an absolute constant $C>0$ such that, for any $n \geq n_0$ and any $j \geq 0$,
\begin{equation*}
    \alpha_{n}(x;\delta_n, 2^j\delta_n) \leq C\alpha_{L_{j+1},\delta_n}(x;\delta_n,2^j\delta_n),
\end{equation*}
where we have set $L_{j+1} = 2^{j+1}\delta_n$. Rescaling by $1/L_{j+1}$, so that the meshsize becomes $2^{-(j+1)}$, and translating the domain, we then obtain that, for any $j \geq 0$,
\begin{equation*}
    \alpha_{L_{j+1},\delta_n}(x;\delta_n,2^j\delta_n) = \alpha_{1,2^{-(j+1)}}(x;2^{-(j+1)},\frac{1}{2}) = \alpha_{1,2^{-(j+1)}}(0;2^{-(j+1)},\frac{1}{2}).
\end{equation*}
We note that the event $A_2(0;2^{-(j+1)},\frac{1}{2})$ depends only on edges in the square $\Lambda_{\frac{1}{2}, 2^{-(j+1)}}(0)$, which is at macroscopic distance from $\partial \Lambda_{1,2^{-(j+1)}}(0)$. As above, by \cite{armExponentsIsing}, there exists $K>0$ such that for any $j \geq 0$,
\begin{equation*}
    \alpha_{1,2^{-(j+1)}}(0;2^{-(j+1)},\frac{1}{2}) \leq K2^{-(\frac{7}{8}-\frac{\beta}{2})j}.
\end{equation*}
The same reasoning applies for $x \in P_j(b)$. We thus have that there exist $n_0 \geq 1$ and a constant $\tilde K > 0$ such that for any $n \geq n_0$,
\begin{align*}
    \sum_{x \in s_a \cup s_b} \PPpmzero{n}(x \in V(\gamma_n)) &\leq \tilde K\sum_{j=0}^{O(\log_2(\frac{\eta}{\delta_n}))} \big(\sum_{x \in P_j(a)} 2^{-\frac{7j}{8}+\beta j}+ \sum_{x \in P_n(b)} 2^{-\frac{7j}{8} + \beta j} \big) \\
    &\leq \tilde K\sum_{j=0}^{O(\log_2(\frac{\eta}{\delta_n}))} 2^{2j+1} 2^{(\beta-\frac{7}{8})j} \lesssim \eta^{\frac{7}{8}-\beta} \delta_n^{-(2+\beta-\frac{7}{8})}.
\end{align*}
We are now left with controlling the contribution of points in $A_{n}(\eta)$ that are in squares other than $s_a$ and $s_b$. This can be done following the strategy of the proof of \cite[Theorem~5.5]{GarbanPivotalMeasure} and we leave the details to the reader. We note that tighter estimates for the contribution of these points could be obtained by using the half-plane two-arm exponents but it is not really necessary here and would pose additional technical issues. Putting all the above estimates together, we obtain that there exist $n_0 \geq 1$ and $C>0$ such that for any $n \geq n_0$, $\delta_n^{\frac{11}{8}+\beta}\EE_{n}^{\pm}[\vert V(\gamma_{n}) \vert] \leq C\delta_n^{\frac{\beta}{2}}$. As $(\delta_n)_n$ converges to $0$, there exists a subsequence $(\delta_m)_m \subset (\delta_n)_n$ such that $\sum_m \delta_m^{\frac{\beta}{2}} < \infty$. Borel-Cantelli lemma then yields the statement of the lemma.
\end{proof}

\subsubsection{Controlling the exponential of the magnetization field in the slit domain} \label{subsec_controlling_spin_field}

In this section, we prove Lemma \ref{lemma_RN_time} and Lemma \ref{lemma_RN_infty}. Let us start with the proof of Lemma \ref{lemma_RN_time}. Our end goal is to apply the dominated convergence theorem. The proof is therefore split into two parts. First, in Lemma \ref{lemma_uniform_bound_time_Omega}, we establish bounds on the terms $\delta_m^{15k/8}\EE_{m,t}^{\pm}[(\sum_x \sigma_x)^k]$ on the left-hand side of \eqref{RN_time_core} that are uniform in $m$ and tight enough to be summable once multiplied by $h^k/k!$. Then, in Lemma \ref{lemma_moments_time}, we prove that, almost surely, for any $k \geq 0$, the $k$-th term of the sum on the left-hand side of \eqref{RN_time_core} almost surely converges as $m \to \infty$ to the $k$-th term of the sum on the right-hand side of \eqref{RN_time_core}.

\begin{lemma} \label{lemma_uniform_bound_time_Omega}
Let $t \in (0,\infty)$. There almost surely exist $m_0 \geq 1$ and a sequence $(C_{k,t})_k$ such that, almost surely, for any $m \geq m_0$ and any $k \geq 1$,
\begin{equation*}
    \sum_{(x_1,\dots,x_k) \in \Omega_{m,t}} \delta_m^{\frac{15k}{8}}\vert \EEpmtime{m}{t}[\sigma_{x_1} \dots \sigma_{x_k}] \vert \leq C_{k,t}
\end{equation*}
with $\sum_k\frac{h^k}{k!}C_{k,t} < \infty$ almost surely.
\end{lemma}

\begin{lemma} \label{lemma_moments_time}
Let $t \in (0,\infty)$. Almost surely, for any $k \geq 1$,
\begin{equation} \label{eq_kmoments_limit}
    \lim_{m \to \infty} \sum_{(x_1, \dots, x_k) \in \Omega_{m,t}^k}C_{\sigma}^{-k}\delta_m^{\frac{15k}{8}}\EEpmtime{m}{t}[\sigma_{x_1}\dots \sigma_{x_k}] = \int_{\Omega_t^k} f_t^{(\pm, k)}(x_1,\dots, x_k) \prod_{j=1}^{k}dx_j
\end{equation}
where the functions $f_t^{(\pm, k)}:=f_{\Omega_t}^{(\pm, k)}$ are as in Theorem \ref{theorem_spin_correlations}.
\end{lemma}

Before proving Lemma \ref{lemma_uniform_bound_time_Omega} and Lemma \ref{lemma_moments_time}, let us give the proof of Lemma \ref{lemma_RN_time}.

\begin{proof}[Proof of Lemma \ref{lemma_RN_time}]
This follows from the dominated convergence theorem: Lemma \ref{lemma_uniform_bound_time_Omega} gives us tight enough bounds on $\vert \delta_m^{\frac{15k}{8}}\EE_{m,t}^{\pm}[(\sum_x \sigma_x)^k] \vert$ while Lemma \ref{lemma_moments_time} shows that, almost surely, for any $k \geq 1$, $ \delta_m^{\frac{15k}{8}}\EE_{m,t}^{\pm}[(\sum_x \sigma_x)^k]$ converges to the correct quantity.
\end{proof}

We now turn to the proof Lemma \ref{lemma_uniform_bound_time_Omega}. Let us first introduce some notations. Let $\eta > 0$. For $t \geq 0$, $m \geq 1$ and $k \geq 1$, we set
\begin{equation*}
    G_{m,t}^{(k)}(\eta) := \{ (x_1, \dots , x_k) \in \Omega_{m,t}^k: \forall j, d(x_j, \partial \Omega_t) 
    \geq \eta \text{ and } \forall i \neq j, \vert x_i -x_j \vert \geq \eta \}.
\end{equation*}
$G_{m,t}^{(k)}(\eta)$ is the "good" set in which points are well-separated and at macroscopic distance from the boundary. We also set $A_{m,t}(\eta) :=\{ x \in \Omega_{m,t}: d(x,\partial \Omega_{t}) < \eta \}$ and for $k \geq 2$,
\begin{align*}
    B_{m, t}^{(k)}(\eta) := \{(x_1,\dots,x_k) \in (\Omega_{m,t}\setminus A_{m,t}(\eta))^k: \exists i \neq j, \vert x_i - x_j \vert < \eta \}.
\end{align*}

\begin{proof}[Proof of Lemma \ref{lemma_uniform_bound_time_Omega}]
To prove the lemma, we are going to control the contribution of points in $G_{m,t}^{(k)}(\eta)$ and in $\Omega_{m,t}^k \setminus G_{m,t}^{(k)}(\eta)$ separately. Let $t \in (0,\infty)$. Let us begin with the sum over $G_{m,t}^{(k)}(\eta)$ for $k \geq 1$. We want to show that there almost surely exists a sequence $(C_{k,t})_k$ and $m_0 \geq 1$ such that almost surely, for any $m \geq m_0$, any $\eta > 0$ and any $k \geq 1$,
\begin{equation} \label{bound_sum_k_good_set}
    \sum_{(x_1, \dots, x_k) \in G_{m,t}^{(k)}(\eta)} \delta_m^{\frac{15k}{8}} \vert \EEpmtime{m}{t}[\sigma_{x_1}\dots \sigma_{x_k}] \vert \leq C_{k,t}
\end{equation}
with $\sum_{k \geq 0} \frac{h^k}{k!}C_{k,t} < \infty$ almost surely. By Lemma \ref{lemma_EEpm_EEplus} and Lemma \ref{lemma_bound_EEplus}, we have that, almost surely,
\begin{align*}
    \vert \sum_{(x_1, \dots, x_k) \in G_{m,t}^{(k)}(\eta)} \delta_m^{\frac{15k}{8}} \EEpmtime{m}{t}[\sigma_{x_1} \dots \sigma_{x_k}] \vert &\leq \sum_{(x_1, \dots, x_k) \in G_{m,t}^{(k)}(\eta)} \delta_m^{\frac{15k}{8}} \EE^{+}_{m,t}[\sigma_{x_1} \dots \sigma_{x_k}] \\
    &\leq \sum_{(x_1, \dots, x_k) \in G_{m,t}^{(k)}(\eta)} C^k\delta_m^{2k} \prod_{j=1}^{k}\bigg(\frac{1}{2}L(x_j) \wedge d_t(x_j)\bigg)^{-\frac{1}{8}} \\
    &\leq \sum_{(x_1, \dots, x_k) \in \Omega_{m,t}^k \setminus  D_{m,t}^{(k)}} C^k\delta_m^{2k} \prod_{j=1}^k \bigg(\frac{1}{2}L(x_j) \wedge d_t(x_j)\bigg)^{-\frac{1}{8}} =: I_{t}^{(k)}(m)
\end{align*}
where we have set $D_{m,t}^{(k)}:=\{ (x_1, \dots, x_k) \in \Omega_{m,t}^k: \exists i \neq j, x_i =x_j\}$ and, for $x \in \Omega_{m,t}$, $d_t(x):= \operatorname{dist}(x, \partial \Omega_{m,t})$. The last inequality follows from the non-negativity of the summand. Set $d_{\Omega} = 2\diam(\Omega)$. Since almost surely, for any $m$, $t \mapsto \diam(\Omega_{m,t})$ is a decreasing function, we can assume that almost surely, for any $m$ and any $t \geq 0$, $\diam(\Omega_{m,t}) \leq d_{\Omega}$. We now rescale $\Omega_{m,t}$ by $d_{\Omega}^{-1}$. Denote by $\tilde \Omega_{m,t}$ the graph thus obtained and observe that almost surely, for $m \geq 1$, for any $t \geq 0$ and for any $x \in \tilde \Omega_{m,t}$, $d_t(x) \leq 1$. We then obtain that, almost surely, for any $m \geq 1$, any $k \geq 1$ and any $t \geq 0$,
\begin{align} \label{bound_Ik_m_t}
    I_t^{(k)}(m)
    \leq C^k\sum_{(x_1, \dots, x_k) \in \tilde \Omega_{m,t}^k \setminus  D_{m,t}^{(k)}} \delta_m^{2k} d_{\Omega}^{-{\frac{k}{8}}} C_{k}(x_1, \dots, x_k;\frac{1}{8}) \prod_{j=1}^{k}d_t(x_j)^{-\frac{1}{8}} C_{k}(x_1, \dots, x_k;\frac{1}{8})
\end{align}
where the last inequality uses that almost surely, for any $m \geq 1$ and any $t \geq 0$, $d_t(x) \leq 1$ for $x \in \tilde \Omega_{m,t}$ and where we have set, for $\alpha > 0$,
\begin{equation*}
    C_{k}(x_1, \dots, x_k;\alpha) := \prod_{j=1}^{k} (\frac{1}{2}L(x_j)\big)^{-\alpha}.
\end{equation*}
Fix $\beta \in (0, \frac{1}{100})$. To estimate the right-hand side of \eqref{bound_Ik_m_t}, we use H\"older inequality with $\alpha, \tilde \alpha > 1$ such that $\frac{1}{\alpha}+\frac{1}{\tilde \alpha}=1$, $\frac{\tilde \alpha}{8}<\frac{7}{8}-\beta$ and $\frac{\alpha}{8}<1$. This choice of $\tilde \alpha$ will guarantee that Lemma \ref{lemma_size_gamma_infty} can be used to estimate the contributions of points near the boundary. We get that, almost surely,
\begin{align} \label{ineq_I_delta_Holder}
    I_t^{(k)}(m)
    &\leq C^k\delta_m^{2k}d_{\Omega}^{-\frac{k}{8}}\bigg( \sum_{(x_1, \dots, x_k) \in \tilde \Omega_{m,t}^k} \prod_{j=1}^{k}d(x_j)^{-\frac{\tilde \alpha}{8}} \bigg)^{1/\tilde \alpha}\bigg(  \sum_{(x_1, \dots, x_k) \in \tilde \Omega_{m,t}^k \setminus  D_{m,t}^{(k)}} C_k(x_1,\dots,x_k;\frac{\alpha}{8}) \bigg)^{1/\alpha}
\end{align}
Set $J_{t}(m,\tilde \alpha) := \sum_{x \in \tilde \Omega_{m,t}} d(x)^{-\tilde \alpha/8}$. We claim the following.

\begin{claim} \label{claim_integral_CR}
Let $\tilde \alpha >1$ be such that $\frac{\tilde \alpha}{8}<\frac{7}{8}-\beta$ and let $t \in (0,\infty)$. Then, there almost surely exists $m_0 \geq 1$ such that almost surely, for any $m \geq m_0$, $J_{t}(m,\tilde \alpha) \leq C_{t,\tilde \alpha} \delta_{m}^{-2}$, where $C_{t,\tilde \alpha}$ is a random constant depending on $\tilde \alpha$, $t$ and $\Omega$.
\end{claim}

\begin{claim} \label{claim_integral_Ck_alpha}
Let $\alpha > 1$ be such that $\alpha/8 <1$. There exists $c>0$ such that almost surely, for any $m \geq 1$, any $k \geq 2$ and any $t \geq 0$,
\begin{equation*}
    \sum_{(x_1,\dots,x_k) \in \tilde \Omega_{m,t}^{k} \setminus  D_{m,t}^{(k)}} C_{k}(x_1,\dots,x_k;\frac{\alpha}{8}) \leq  c^k\delta_m^{-2k}k^{\frac{k\alpha}{16}}.
\end{equation*}
\end{claim}

These two claims together with the inequality \eqref{ineq_I_delta_Holder} imply that there almost surely exists $m_0 \geq 1$ such that for any $m \geq m_0$ and any $k \geq 1$, 
\begin{equation*}
    I_{t}^{(k)}(m) \leq d_{\Omega}^{-\frac{k}{8}}\delta_m^{2k}(C_{t, \tilde \alpha}\delta_m^{-2k})^{\frac{1}{\tilde \alpha}}c^{\frac{k}{\alpha}}\delta_m^{-\frac{2k}{\alpha}}k^{\frac{k}{16}},
\end{equation*}
which yields the inequality \eqref{bound_sum_k_good_set}.

To prove Lemma \ref{lemma_uniform_bound_time_Omega}, it now remains to control the contribution of points in $\Omega_{m,t}^k \setminus G_{m,t}^{(k)}(\eta)$. We are going to show that there almost surely exists $m_0 \geq 1$ such that for any $m \geq m_0$, any $k \geq 1$ and any $\eta > 0$,
\begin{equation} \label{ineq_contribution_eta}
    M_t^{(k)}(m;\eta) := \sum_{(x_1,\dots,x_k)\in \Omega_{m,t}^k\setminus G_{m,t}^{(k)}(\eta)} \delta_m^{\frac{15k}{8}} \vert \EEpmtime{m}{t}[\sigma_{x_1}\dots \sigma_{x_k}] \vert \leq C_{1,t}^k\eta^{\frac{7}{8}-\beta}k^{\frac{k}{16}} + C_{2,t}^k\eta k^{\frac{k}{16}}
\end{equation}
where $C_{1,t}, C_{2,t} > 0$ are random constants depending on $t$ and $\Omega$. Observe in particular that the terms on the above right-hand side are summable in $k$, once multiplied by $\frac{h^k}{k!}$.
We have that, almost surely, for any $k \geq 1$,
\begin{align} \label{decomposition_contribution_eta}
    M_{t}^{(k)}(m;\eta) &= \sum_{(x_1,\dots,x_k)\in B_{m,t}^{(k)}(\eta)} \delta_m^{\frac{15k}{8}} \vert \EEpmtime{m}{t}[\sigma_{x_1}\dots \sigma_{x_k}] \vert \nonumber \\
    &+ \sum_{r=1}^{k} \sum_{W_r} \sum_{(x_1, \dots, x_k) \in \phi_{W_r}(A_{m,t}(\eta)^r \times (\Omega_{m,t}\setminus A_{m,t}(\eta))^{k-r})}   \delta_m^{\frac{15k}{8}} \vert \EEpmtime{m}{t}[\sigma_{x_1}\dots \sigma_{x_k}] \vert.
\end{align}
Above, $\sum_{W_r}$ ranges over the different possible choices of $r$ coordinates in $A_{m,t}(\eta)$ and $k-r$ ones in $\Omega_{m,t} \setminus A_{m,t}(\eta)$. Given such a choice $W_r$, $\phi_{W_r}(A_{m,t}(\eta)^r \times (\Omega_{m,t}\setminus A_{m,t}(\eta))^{k-r})$ is the permutation of $A_{m,t}(\eta)^r \times (\Omega_{m,t}\setminus A_{m,t}(\eta))^{k-r}$ corresponding to this choice. However, it is easy to see that the sum over points in $\phi_{W_r}(A_{m,t}(\eta)^r \times (\Omega_{m,t}\setminus A_{m,t}(\eta))^{k-r})$ does not depend on this choice of permutations. To obtain the inequality \eqref{ineq_contribution_eta}, we thus see that we must upper bound the two sums appearing in the decomposition \eqref{decomposition_contribution_eta} of $M_{t}^{(k)}(m;\eta)$. These bounds are the content of the following two claims.

\begin{claim} \label{claim_bound_B_eta_k}
There almost surely exists $m_0$ such that, almost surely, for any $m \geq m_0$, any $k \geq 2$ and any $\eta > 0$,
\begin{equation*}
    \sum_{(x_1,\dots,x_k)\in B_{m,t}^{(k)}(\eta)} \delta_m^{\frac{15k}{8}}\vert \EEpmtime{m}{t}[\sigma_{x_1}\dots \sigma_{x_k}] \vert \leq C_1^k\eta k^{\frac{k}{16}}
\end{equation*}
where $C_1=C_1(t,\Omega)$ is a random constant depending on $t$ and $\Omega$.
\end{claim}

\begin{claim} \label{claim_bound_boundary_eta}
There almost surely exists $m_0$ such that, almost surely, for any $m\geq m_0$, any $k\geq r \geq 1$ and any $\eta >0$,
\begin{equation*}
    \sum_{(x_1,\dots,x_k)\in A_{m,t}(\eta)^r\times (\Omega_{m,t}\setminus A_{m,t}(\eta))^{k-r}} \delta_m^{\frac{15k}{8}}\vert \EEpmtime{m}{t}[\sigma_{x_1}\dots \sigma_{x_k}]\vert \leq C_{2}^k \eta^{\frac{r}{\rho}(\frac{7}{8}-\beta)}k^{\frac{k}{16}}
\end{equation*}
where $C_2=C_2(t,\Omega)$ is a random constant depending on $t$ and $\Omega$.
\end{claim}

Claim \ref{claim_bound_B_eta_k} and Claim \ref{claim_bound_boundary_eta} yield that there almost surely exists $m_0 \geq 1$ such that almost surely, for any $m \geq m_0$, any $k \geq 1$ and any $\eta>0$,
\begin{align*}
    M_t^{(k)}(m;\eta)
    \leq C^k\eta k^{\frac{k}{16}} + \sum_{r=1}^{k} \vert W_r \vert C^k \eta^{(\frac{7}{8}-\beta)}k^{\frac{k}{16}} \leq C_{1}^k\eta k^{\frac{k}{16}}+ 2^kC_{2}^k \eta^{\frac{7}{8}-\beta}k^{\frac{k}{16}},
\end{align*}
which shows the inequality \eqref{ineq_contribution_eta} and concludes the proof of the lemma,
\end{proof}

Let us now prove Claim \ref{claim_integral_CR}, Claim \ref{claim_integral_Ck_alpha}, Claim \ref{claim_bound_B_eta_k} and Claim \ref{claim_bound_boundary_eta}.

\begin{proof}[Proof of Claim \ref{claim_integral_CR}]
Let $\eta > 0$. Let us first estimate the contribution of points in $\tilde \Omega_{m,t} \setminus \tilde A_{m,t}(\eta)$ to $J_{t}(m,\tilde \alpha)$, where $\tilde A_{m,t}(\eta) := \{x \in \tilde \Omega_{m,t}: \operatorname{dist}(x,\partial \Omega_{m,t}) < d_{\Omega}\eta\}$. We have that, almost surely, for any $m \geq 1$,
\begin{equation*}
    \sum_{x \in \tilde \Omega_{m,t} \setminus \tilde A_{m,t}(\eta)} d_t(x)^{-\frac{\tilde \alpha}{8}} \lesssim \eta^{-\frac{\tilde \alpha}{8}}\vol(\tilde \Omega) \delta_m^{-2}.
\end{equation*}
We now want to estimate the contribution of points in $\tilde A_{m,t}(\eta)$ to $J_t(m,\tilde \alpha)$. Let us write $\tilde A_{m,t}(\eta) = \tilde A_{m,t}^{(1)}(\eta) \sqcup \tilde A_{m,t}^{(2)}(\eta)$ where we have set
\begin{align*}
    &\tilde A_{m,t}^{(1)}(\eta) =\{ x \in \tilde A_{m,t}(\eta): \operatorname{dist}(x,\partial \Omega) < \operatorname{dist}(x,\gamma_m([0,t_m])) \} \\
    &\tilde A_{m,t}^{(2)}(\eta) = \{ x \in \tilde A_{m,t}(\eta): \operatorname{dist}(x,\gamma_m([0,t_m])) \leq \operatorname{dist}(x,\partial \Omega) \}.
\end{align*}
Observe that almost surely, for any $m \geq 1$ and any $t \geq 0$, $\tilde A_{m,t}^{(1)}(\eta) \subset \cup_{k=0}^{K(m)} D_k(m)$ where for $k \geq 0$, $D_{k}(m):= \{x \in \tilde \Omega_{m}: 2^{k}\delta_m \leq d_{\Omega}\operatorname{dist}(x,\partial \Omega_{m}) \leq 2^{k+1}\delta_m \}$ with $K(m) = \log_2(\eta/\delta_m)+1$. Therefore, using that almost surely, for any $t \geq 0$ and for any $x \in \tilde A_{m,t}^{(1)}(\eta)$, $d_t(x)= \operatorname{dist}(x,\partial \Omega_{m})$, the contribution of points in $\tilde A_{m,t}^{(1)}(\eta)$ is almost surely bounded by
\begin{equation} \label{contribution_points_near_boundary_domain}
    \sum_{x \in \tilde A_{m,t}^{(1)}(\eta)} \operatorname{dist}(x,\partial \Omega_{m})^{-\tilde \alpha/8} \leq \sum_{k=0}^{O(\log_{2}(\eta/\delta_m))} 2^{-\frac{\tilde \alpha}{8}k}\delta_m^{-\frac{\tilde \alpha}{8}} \vert D_{k}(m) \vert \leq \sum_{k=0}^{O(\log_{2}(\eta/\delta_m))} 2^{\frac{\tilde \alpha}{8}k}\delta_m^{-\frac{\tilde \alpha}{8}} 2^{k}\delta_m^{-1} \lesssim \delta_m^{-2}\eta^{1-\frac{\tilde \alpha}{8}}.
\end{equation}
Let us now estimate the contributions of points in $\tilde A_{m,t}^{(2)}(\eta)$. Observe that almost surely, for any $x \in \tilde A_{m,t}^{(2)}(\eta)$ and any $t \geq 0$, $d_t(x) = \operatorname{dist}(x,\gamma_m([0,t_m]))$. Moreover, we have that, almost surely, $\tilde A_{m,t}^{(2)}(\eta) \subset \cup_{k=0}^{K(m)}C_{k}(\gamma_m;\eta,t)$ where $C_k(\gamma_m; \eta,t):= \{ x \in \tilde \Omega_{m,t}: 2^{k}\delta_m \leq d_{\Omega}\operatorname{dist}(x,\gamma_m([0,t_m])) \leq 2^{k+1}\delta_m\}$. Therefore, we have that, almost surely, for any $m \geq 1$ and any $t \geq 0$,
\begin{equation*}
    \sum_{x \in \tilde A_{m,t}^{(2)}(\eta)} d_t(x)^{-\frac{\tilde \alpha}{8}} \leq \sum_{k=0}^{O(\log_{2}(\eta/\delta_m))} 2^{-k\frac{\tilde \alpha}{8}}\delta_m^{-\frac{\tilde \alpha}{8}} \vert C_k(\gamma_m; \eta, t) \vert \leq \sum_{k=0}^{O(\log_{2}(\eta/\delta_m))} 2^{-k\frac{\tilde \alpha}{8}}\delta_m^{-\frac{\tilde \alpha}{8}} \vert C_k(\gamma_m; \eta, \infty) \vert.
\end{equation*}
By Lemma \ref{lemma_thinness_gamma}, almost surely $\delta_m^{\frac{11}{8}+\beta}\vert V(\gamma_m) \vert \to 0$ as $m \to \infty$. This implies that there almost surely exists $m_0 \geq 1$ such that for any $\eta >0$, the set $C_k(\gamma_m, \eta, \infty)$ can be covered by $C_{\infty}\eta^{-\frac{11}{8}-\beta}$ balls of radius $\eta$, where $C_{\infty}$ is a random constant independent of $t$. Each of these balls contains a number of vertices of order $\eta^{2}\delta_m^{-2}$. This yields that almost surely, for any $m \geq m_0$ and any $\eta > 0$,
\begin{equation} \label{contribution_points_near_curve}
    \sum_{x \in \tilde A_{m,t}^{(2)}(\eta)} d_t(x)^{-\frac{\tilde \alpha}{8}} \leq C_{\infty}\delta_m^{-2}\sum_{k=0}^{O(\log_2(\eta/\delta_m))} 2^{\frac{\tilde \alpha(k+1)}{8}} \eta^{-\frac{\tilde \alpha}{8}} \eta^{\frac{7}{8}-\beta} 2^{k(\frac{11}{8}+\beta)}2^{-2k} \leq C_{\infty}\delta_m^{-2}\eta^{\frac{7}{8}-\beta-\frac{\tilde \alpha}{8}} \sum_{k=0}^{\infty}2^{-(\frac{7}{8}-\beta-\frac{\tilde \alpha}{8})}.
\end{equation}
The sum on the above right-hand side is finite since $\frac{\tilde \alpha}{8} < \frac{7}{8}-\beta$ by assumption. Putting everything together, we get that there almost surely exists $m_0 \geq 1$ such that for any $m \geq m_0$, any $t \geq 0$ and any $\eta >0$,
\begin{equation*}
    J_{t}(m,\tilde \alpha) \lesssim \delta_m^{-2}\eta^{\frac{7}{8}-\beta-\frac{\tilde \alpha}{8}} + \delta_m^{-2}\eta^{1-\frac{\tilde \alpha}{8}} + \eta^{-\frac{\tilde \alpha}{8}}\delta_m^{-2}
\end{equation*}
which yields the claim.
\end{proof}

\begin{proof}[Proof of Claim \ref{claim_integral_Ck_alpha}]
Since $C_k$ is a nonnegative function and $\Omega$ is a bounded domain, there exists $\tilde d_{\Omega} > 1$ such that, almost surely, for any $m \geq 1$, any $t \geq 0$ and any $k \geq 0$,
\begin{equation*}
    \sum_{(x_1, \dots, x_k) \in \tilde \Omega_{m,t} \setminus D_{m,t}^{(k)}} C_{k}(x_1,\dots,x_k;\frac{\alpha}{8}) \leq \sum_{(x_1,\dots,x_k)\in B_{m}(0,\tilde d_{\Omega})^k \setminus D_m^{(k)}(\tilde d_{\Omega})} C_{k}(x_1,\dots,x_k;\frac{\alpha}{8})
\end{equation*}
where $B_{m}(0,\tilde d_{\Omega})$ denotes the $\delta_n$-approximation of the Euclidean ball $B(0,\tilde d_{\Omega})$ and $D_m^{(k)}(\tilde d_{\Omega}):= \{ (x_1, \dots,x_k) \in B_m(0,\tilde d_{\Omega})^k: \exists i \neq j, x_i=x_j \}$. By \cite[Proposition~3.10]{MourratIsing} and \cite[Lemma~3.10]{Junnila}, as $\frac{\alpha}{8}<1$, we obtain that there exists a constant $c>0$ such that, almost surely, for any $m \geq 1$ and any $k \geq 2$,
\begin{equation*}
    \sum_{(x_1,\dots,x_k)\in B_{m}(0,\tilde d_{\Omega})^k \setminus D^{(k)}(\tilde d_{\Omega})} C_{k}(x_1,\dots,x_k;\frac{\alpha}{8}) \leq c^k\delta_{m}^{-2k}k^{\frac{k\alpha}{16}},
\end{equation*}
which yields the claim. We note that the estimate of \cite[Lemma~3.10]{Junnila} is obtained in the continuum but can also be derived in the discrete by using the same type of arguments as in the proof of \cite[Proposition~3.10]{MourratIsing}.
\end{proof}

\begin{proof}[Proof of Claim \ref{claim_bound_B_eta_k}]
By Lemma \ref{lemma_EEpm_EEplus}, we have that, almost surely,
\begin{equation*}
    F_t^{(k)}(m;\eta):=\delta_m^{\frac{15k}{8}}\sum_{(x_1,\dots,x_k)\in B_{m,t}^{(k)}(\eta)} \vert \EEpmtime{m}{t}[\sigma_{x_1}\dots \sigma_{x_k}] \vert \leq \delta_m^{\frac{15k}{8}}\sum_{(x_1,\dots,x_k)\in B_{m,t}^{(k)}(\eta)} \vert \EE^{+}_{m,t}[\sigma_{x_1}\dots \sigma_{x_k}] \vert
\end{equation*}
Using the same arguments as in the second part of the proof of \cite[Proposition~3.10]{MourratIsing} together with \cite[Lemma~3.10]{Junnila} to estimate the $k$-dependent constant, we get that there almost surely exists $m_0 \geq 1$ such that for any $m \geq m_0$, any $k \geq 2$ and any $\eta > 0$,
\begin{equation*}
    \sum_{(x_1,\dots,x_k)\in B_{m,t}^{(k)}(\eta)} \vert \EE^{+}_{m,t}[\sigma_{x_1}\dots \sigma_{x_k}] \vert \leq C_{1}^k\eta\delta_{m}^{-\frac{15k}{8}}k^{\frac{k}{16}}
\end{equation*}
where $C_1 = C_{1}(t,\Omega)$ is a random constant depending on $t$ and $\Omega$. This concludes the proof of the claim. Note that the $\eta$-dependence of the above right-hand side could be made to depend on $k$ but these tighter estimates are not necessary here.
\end{proof}

\begin{proof}[Proof of Claim \ref{claim_bound_boundary_eta}]
Using the Edwards-Sokal coupling of Proposition \ref{prop_ES_coupling_no_field} and with the same notations as in the proof of Lemma \ref{lemma_EEpm_EEplus}, we have that $\EE_{m,t}^{+}[\sigma_{x_1}\dots\sigma_{x_k}] = \PP_{m,t}^{\operatorname{FK}, \operatorname{wired}}(A_{x_1,\dots,x_k})$. By the Markov property of the FK-Ising model, we then have that, almost surely, for any $m \geq 1$, any $r \geq 1$, any $k \geq r$ and any $(x_1,\dots, x_k) \in A_{m,t}(\eta)^r \times (\Omega_{m,t} \setminus A_{m,t}(\eta))^{k-r}$,
\begin{equation*}
    \PP_{m,t}^{\operatorname{FK},\operatorname{wired}}(A_{x_1,\dots,x_k}) \lesssim \PP_{m,t}^{\operatorname{FK},\operatorname{wired}}(A_{x_{r+1},\dots, x_k}) \prod_{j=1}^{r}\PP_{\Lambda_{d_t(x_j),m}(x_j)}^{\operatorname{FK},\operatorname{wired}}(x_j \longleftrightarrow \partial \Lambda_{d_t(x_j),m}(x_j))
\end{equation*}
where for $x \in A_{m,t}(\eta)$, $\Lambda_{d_t(x_j),m}(x_j) \subset \delta_{m}\mathbb{Z}^2$ is the square centered at $x_j$ with side-length $\frac{1}{2}d_t(x_j)$ and $\PP_{\Lambda_{d_t(x_j),m}(x_j)}^{\operatorname{FK},\operatorname{wired}}$ is the FK-Ising measure in $\Lambda_{d_t(x_j),m}(x_j)$ with wired boundary conditions. Here,we recall that $d_t(x_j) = \operatorname{dist}(x,\partial \Omega_{m,t})$. Denoting by $S_{m,t}^{(r,k)}$ the sum in the statement of Claim \ref{claim_bound_boundary_eta}, this implies that
\begin{equation*}
    S_{m,t}^{(r,k)} \lesssim \delta_m^{\frac{15k}{8}}\bigg(\sum_{x \in A_{m,t}(\eta)} \PP_{\Lambda_{d_t(x)}(x)}^{\operatorname{FK},\operatorname{wired}}(x_j \longleftrightarrow \partial \Lambda_{d_t(x)}(x))\bigg)^r \sum_{(x_{r+1},\dots, x_{k-r}) \in (\Omega_{m,t} \setminus A_{m,t}(\eta))^{k-r}} \PP_{m,t}^{\operatorname{FK},\operatorname{wired}}(A_{x_{r+1},\dots, x_k}).
\end{equation*}
By \cite[Lemma~5.4]{RSWbounds}, there exists $C>0$ such that, almost surely, for any $m \geq 1$, any $t \geq 0$ and any $x \in A_{m,t}(\eta)$, 
\begin{equation*}
    \PP_{\Lambda_{d_t(x),m}(x)}^{\operatorname{FK},\operatorname{wired}}(x_j \longleftrightarrow \partial \Lambda_{d_t(x),m}(x)) \leq C\delta_m^{1/8}d_{t}(x_j)^{-1/8}.
\end{equation*}
The sum $\sum_{A_{m,t}(\eta)}d_{t}(x_j)^{-1/8}$ can be estimated as in we did in the proof of Claim \ref{claim_integral_CR}, see the discussion leading to the inequalities \eqref{contribution_points_near_boundary_domain} and \eqref{contribution_points_near_curve}. This yields that there almost surely exits $m_0$ such that, almost surely, for any $m \geq m_0$, any $k\geq r \geq 1$ and any $\eta > 0$,
\begin{equation*}
     S_{m,t}^{(r,k)} \lesssim \delta_m^{\frac{15(k-r)}{8}}\eta^{r(\frac{7}{8}-\beta)} \sum_{(x_{1},\dots, x_{k-r}) \in (\Omega_{m,t} \setminus A_{m,t}(\eta))^{k-r}}\PP_{m,t}^{\operatorname{FK},\operatorname{wired}}(A_{x_{1},\dots, x_{k-r}}).
\end{equation*}
The sum on the above right-hand side can now be upper bounded using Claim \ref{claim_integral_Ck_alpha} and Claim \ref{claim_bound_B_eta_k}, which concludes the proof.
\end{proof}

We now turn to the proof of Lemma \ref{lemma_moments_time}. We keep the same definitions for $G_{m,t}^{(k)}(\eta)$, $A_{m,t}(\eta)$ and $B_{m,t}^{(k)}(\eta)$.

\begin{proof}[Proof of Lemma \ref{lemma_moments_time}]
We have that, almost surely, for any $m \geq 1$, any $t \geq 0$ and any $k \geq 1$,
\begin{align*}
    \sum_{(x_1, \dots, x_k) \in \Omega_{m,t}^k} C_{\sigma}^{-k}\delta_m^{\frac{15k}{8}} \EEpmtime{m}{t}[\sigma_{x_1} \dots \sigma_{x_k}] = &\sum_{(x_1, \dots, x_k) \in G_{m,t}^{(k)}(\eta)} C_{\sigma}^{-k}\delta_m^{\frac{15k}{8}} \EEpmtime{m}{t}[\sigma_{x_1} \dots \sigma_{x_k}] \\
    &+ \sum_{(x_1, \dots, x_k) \in \Omega_{m,t}^k \setminus G_{m,t}^{(k)}(\eta)}  C_{\sigma}^{-k}\delta_m^{\frac{15k}{8}} \EEpmtime{m}{t}[\sigma_{x_1} \dots \sigma_{x_k}].
\end{align*}
Since $(\hat \Omega_{m,t};a_{m,t},b_m)$ almost surely converges to $(\Omega;a,b)$ in the Carath\'eodory topology, we obtain from Theorem \ref{theorem_spin_correlations} that, almost surely, for any $k \geq 1$,
\begin{equation} \label{lim_moments_good_set}
    \lim_{m \to \infty} \sum_{(x_1, \dots, x_k) \in G_{m,t}^{(k)}(\eta)} C_{\sigma}^{-k}\delta_m^{\frac{15k}{8}} \EEpmtime{m}{t}[\sigma_{x_1} \dots \sigma_{x_k}] = \int_{ G_{t}^{(k)}(\eta)} f_t^{(k)}(x_1, \dots, x_k) \prod dx_j
\end{equation}
where we have set $G_{t}^{(k)}(\eta):=\{ (x_1, \dots , x_k) \in \Omega_{t}^k: \forall j, d(x_j, \partial \Omega_t) \geq \eta \text{ and } \forall i \neq j, \vert x_i -x_j \vert \geq \eta \}$ and $f_t^{(k)} := f_{\Omega_{t}}^{(\pm, k)}$. Observe now that the inequality \eqref{ineq_contribution_eta} shows that the contribution of points in $\Omega_{m,t}^k \setminus G_{m,t}^{(k)}(\eta)$ is uniformly bounded in $m$ by a quantity that converges to $0$ as $\eta \to 0$. Together with \eqref{lim_moments_good_set}, this yields that, almost surely, for any $k \geq 1$,
\begin{equation*}
    \lim_{m \to \infty} \sum_{(x_1,\dots, x_k) \in \Omega_{m,t}^k} C_{\sigma}^{-k}\delta_m^{\frac{15k}{8}}\EEpmtime{m}{t}[\sigma_{x_1}\dots \sigma_{x_k}] = \lim_{\eta \to 0} \int_{G_t^{(k)}(\eta)} f_t^{(k)}(x_1,\dots, x_k) \prod_{j=1}^{k}dx_j.
\end{equation*}
Lemma \ref{lemma_moments_time} then follows from the claim below, which shows that in the continuum, the contribution of points in $\Omega_t^k \setminus G_{t}^{(k)}(\eta)$ to the integral on the right-hand side of \eqref{eq_kmoments_limit} vanishes as $\eta \to 0$.

\begin{claim} \label{claim_no_contribution_eta}
Almost surely, for any $k \geq 1$ and any $t \in (0,\infty)$, $\lim_{\eta \to 0} \int_{G_t^{(k)}(\eta)} f_t^{(k)}(x_1,\dots, x_k) \prod_{j=1}^{k}dx_j = \int_{\Omega_t^k} f_t^{(k)}(x_1,\dots,x_k)\prod dx_j < \infty$.
\end{claim}
\end{proof}

We now prove Claim \ref{claim_no_contribution_eta}.

\begin{proof}[Proof of Claim \ref{claim_no_contribution_eta}]
Let $k \geq 1$. Observe that almost surely, as $\eta \to 0$, $\mathbb{I}_{G_{t}^{(k)}(\eta)}f_t^{(k)}(x_1, \dots, x_k) \to f_t^{(k)}(x_1, \dots, x_k)$ pointwise. Therefore, to show the lemma, it is enough to show that $\mathbb{I}_{G_{t}^{(k)}(\eta)}f_t^{(k)}(x_1, \dots, x_k)$ is almost surely uniformly bounded in $L^2(\Omega_t^k)$. Going back to the discrete level and using Lemma \ref{lemma_EEpm_EEplus} and Theorem \ref{theorem_spin_correlations}, we have that, almost surely,
\begin{align*}
    \int_{ G_{t}^{(k)}(\eta)} \vert f_t^{(k)}(x_1, \dots, x_k) \vert^2 \prod dx_j &= \lim_{n \to \infty} \sum_{(x_1, \dots, x_k) \in G_{n,t}^{(k)}(\eta)} \delta_n^{\frac{15k}{8}}C_{\sigma}^{-2k} \vert \EEpmtime{n}{t}[\sigma_{x_1} \dots \sigma_{x_k}] \vert^2 \\
    &\leq \lim_{n \to \infty} \sum_{(x_1, \dots, x_k) \in G_{n,t}^{(k)}(\eta)} \delta_n^{\frac{15k}{8}} C_{\sigma}^{-2k}\vert \EE^{+}_{n,t}[\sigma_{x_1} \dots \sigma_{x_k}] \vert^2 \\
    &= \int_{ G_{t}^{(k)}(\eta)} \vert f_t^{(+,k)}(x_1, \dots, x_k) \vert^2 \prod dx_j \\
    &\leq \int_{ \Omega_t^k} \vert f_t^{(+,k)}(x_1, \dots, x_k) \vert^2 \prod dx_j
\end{align*}
where $f_{t}^{(+,k)}:=f_{\Omega_t}^{(+,k)}$ is as in Theorem \ref{theorem_spin_correlations}. Almost sure finiteness of the last integral follows from Lemma \ref{lemma_fineteness_moments_plusBC} in Appendix \ref{sec_appendix} together with the fact since $\gamma$ has law $\PP_{\operatorname{SLE}_3}^{(\Omega,a,b)}$, $\partial \Omega_{t}$ almost surely has Hausdorff dimension $11/8$.
\end{proof}

Let us now turn to the proof of Lemma \ref{lemma_RN_infty}. The proof follows the same strategy as that of the proof of Lemma \ref{lemma_RN_time}. Namely, we first show that each term in the series on the left-hand side of \eqref{RN_infty_core} almost surely converges and then prove estimates on these terms that allow us to use the dominated convergence theorem. These two statements are the content of the two lemmas below. We omit their proof as it is almost identical to that of Lemma \ref{lemma_moments_time} and Lemma \ref{lemma_uniform_bound_time_Omega}.

\begin{lemma} \label{lemma_moments_infty}
Let $\xi \in \{-,+\}$. Almost surely, for any $k \geq 1$,
\begin{equation*}
    \lim_{m \to \infty} \sum_{(x_1, \dots, x_k) \in \Omega_{m, j(\xi)}^k}C_{\sigma}^{-k}\delta_m^{\frac{15k}{8}}\EE_{m,j(\xi)}^{\xi}[\sigma_{x_1}\dots \sigma_{x_k}] = \int_{\Omega_{j(\xi)}^k} f_{j(\xi)}^{(\xi,k)}(x_1,\dots, x_k) \prod_{j=1}^{k}dx_j
\end{equation*}
where $j(\xi)$, $\Omega_{j(\xi)}$ and the functions $f_{j(\xi)}^{(\xi,k)}$ are as in Lemma \ref{lemma_RN_infty}.
\end{lemma}

\begin{lemma} \label{lemma_uniform_bound_infty_Omega}
Let $\xi \in \{-,+\}$. There almost surely exists $m_0$ such that almost surely, for any $m \geq m_0$ and any $k \geq 1$,
\begin{equation*}
    \sum_{(x_1,\dots,x_k) \in \Omega_{m,j(\xi)}} \delta_m^{\frac{15k}{8}}\EE_{m,j(\xi)}^{\xi}[\sigma_{x_1} \dots \sigma_{x_k}] \leq C_{k,\infty}^{(\xi)}
\end{equation*}
with $\sum_k\frac{h^k}{k!}C_{k,\infty}^{(\xi)} < \infty$ almost surely.
\end{lemma}

\begin{proof}[Proof of Lemma \ref{lemma_RN_infty}]
Lemma \ref{lemma_RN_infty} directly follows from Lemma \ref{lemma_moments_infty}, Lemma \ref{lemma_uniform_bound_infty_Omega} and dominated convergence.   
\end{proof}

\subsection{Conformal covariance of the limiting interface}

In this section, we show that the law of $\PP_{h}^{(\Omega,a,b)}$ is conformally covariant.

\begin{proof}[Proof of Proposition \ref{prop_conformal_covariance}]
This follows from the expression \eqref{RN_time_theorem_intro} for the Radon-Nikodym derivative of $\PP_{h}^{(\Omega,a,b)}$ with respect to $\PP_{\operatorname{SLE}_3}^{(\Omega,a,b)}$. Let $\psi: \Omega \to \tilde \Omega$ be a conformal map. $\PP_{\operatorname{SLE}_3}^{(\Omega,a,b)}$ is conformally invariant while by Theorem \ref{theorem_spin_correlations}, a change of variable yields that, almost surely, for any $k \geq 1$,
\begin{equation*}
    \int_{\Omega_t^k} [\prod_{j=1}^{k}h(z_j)] f_t^{(\pm, k)}(z_1,\dots,z_k)\prod_{j=1}^{k}dz_j = \int_{\tilde \Omega_t^k}  [\prod_{j=1}^{k}\tilde h(z_j)] \tilde f_t^{(\pm, k)}(z_1,\dots,z_k) \prod_{j=1}^{k}dz_j
\end{equation*}
where we have set $\tilde \Omega_t = \psi(\Omega_t)$, $\tilde f_{t}^{(\pm,k)}:=f_{\tilde \Omega_t}^{(\pm,k)}$ and for $z \in \tilde \Omega$, $\tilde h(z) = \vert (\psi^{-1})'(z) \vert^{\frac{15}{8}}h(\psi^{-1}(z))$. Similarly, $\mathcal{Z}_{h}(\Omega)=\mathcal{Z}_{\tilde h}(\tilde \Omega)$, which concludes the proof.
\end{proof}

\section{Scaling limit of the interface away from near-criticality}

\subsection{The case of a small magnetic perturbation: proof of Proposition \ref{prop_small_H}} \label{sec_proof_small_H}

In this section, we prove Proposition \ref{prop_small_H}. We give the proof in the case $H_x=h\delta^{\frac{15}{8}}g_{1}(\delta)$ with $h \in \mathbb{R}$ but the extension to the case of space-varying external fields that are bounded is straightforward. We keep the same notations as in Section \ref{sec_proof_near_critical} for $\PP_{\delta}^{\pm}$, $\PP_{\delta,h}^{\pm}$, $\PP_{n}^{\pm}$ and $\PP_{n,h}^{\pm}$.

\begin{proof}[Proof of Proposition \ref{prop_small_H}]
The proof of Proposition \ref{prop_small_H} consists of two steps: first showing tightness of $(\gamma_{\delta})_{\delta}$ under $(\PP_{\delta,h}^{\pm})_{\delta}$ in the topology \ref{topo_4} and then characterizing the law of any subsequential limit of $(\gamma_{\delta})_{\delta}$ under $(\PP_{\delta,h}^{\pm})_{\delta}$. For both of these steps, we rely on the observation that for any $\delta > 0$, the law of $\gamma_{\delta}$ under $\PP_{\delta,h}^{\pm}$ is absolutely continuous with respect to that of $\gamma_{\delta}$ under $\PP_{\delta}^{\pm}$ with corresponding Radon-Nikodym derivative given by
\begin{equation} \label{RN_delta_small_h}
    F_{\delta}(\gamma_{\delta}):= \frac{\der \PP_{\delta,h}^{\pm}}{\der \PP_{\delta}^{\pm}}(\gamma_{\delta}) = \frac{1}{\mathcal{Z}_{\delta}} \EE_{\delta}^{\pm}\bigg[ \exp\bigg(h\delta^{\frac{15}{8}}g_1(\delta) \sum_{x \in \Omega_{\delta}} \sigma_x \bigg) \vert \gamma_{\delta}\bigg].
\end{equation}
With this observation, tightness of $(\gamma_{\delta})_{\delta}$ under $(\PP_{\delta,h}^{\pm})_{\delta}$ in the topology \ref{topo_4} follows from the same arguments as those used in the proof of Proposition \ref{prop_tightness}. Indeed, using the same inequalities as in this proof but with $\delta^{15/8}$ replaced by $\delta^{15/8}g_{1}(\delta)$, it is not hard to see that for any $p \geq 1$, $\sup_{\delta >0}\EE_{\delta}[F_{\delta}(\gamma_{\delta})^p]<\infty$. We leave the details to the reader.

To characterize the law of any subsequential limits of $(\gamma_{\delta})_{\delta}$ under $(\PP_{\delta,h}^{\pm})_{\delta}$, we proceed as in the proof of Proposition \ref{prop_characterization}. Let $(\delta_n)_n$ be a sequence such that $\PP_{n,h}^{\pm}$ converges weakly in the topology \ref{topo_4}. We use Theorem \ref{theorem_SLE3} and Skorokhod representation theorem to define on a common probability space $(S,\mathcal{F},\PP)$ a sequence $(\gamma_n)_n$ and a random curve $\gamma$ such that for any $n \geq 1$, $\gamma_n$ has law $\PP_{n}^{\pm}$, $\gamma$ has law $\PP_{\operatorname{SLE}_3}^{(\Omega,a,b)}$ and $\PP$-almost surely, $\gamma_n \to \gamma$ in the topologies \ref{topo_1}--\ref{topo_4}. The random variables $(F_{n}(\gamma_n))_n$ given by \eqref{RN_delta_small_h} (with $\delta=\delta_n$) can then all be defined on $(S, \mathcal{F}, \PP)$. To characterize the limiting law of $\gamma_n$ under $\PP_{n,h}^{\pm}$, we are going to show that $F_{n}(\gamma_n)$ converges in $L^1(\PP)$ to 1. This will imply that the Radon-Nikodym derivative of the law of any subsequential limit of $(\gamma_{\delta})_{\delta}$ under $(\PP_{\delta,h}^{\pm})_{\delta}$ with respect to $\PP_{\operatorname{SLE}_3}^{(\Omega,a,b)}$ is equal to $1$, thus proving Proposition \ref{prop_small_H}.

To establish the above mentioned convergence, we are going to show that
\begin{equation} \label{lim_RN_equal1}
    \lim_{n \to \infty} \EE_{n}^{\pm}[(F_{n}(\gamma_{n})-1)^2] = 0.
\end{equation}
Developing the square and using that $\EE_{n}^{\pm}[F_{n}(\gamma_{n})] = 1$, to show \eqref{lim_RN_equal1}, we see that it suffices to prove that $\EE_{n}^{\pm}[F_{n}(\gamma_{n})^2]$ converges to $1$. Let us set $S_{n}(\gamma_{n}) := \EE_{n}^{\pm}[\exp(h\delta_n^{15/8}g_{1}(\delta_n)\sum_x \sigma_x) \vert \gamma_{n}]$. We are first going to show that $\EE_{n}^{\pm}[S_{n}(\gamma_{n})^2]$ converges to $1$ as $n \to \infty$ by exhibiting lower and upper bounds on $\EE_{n}^{\pm}[S_{n}(\gamma_{n})^2]$ that both converge to $1$ as $n \to \infty$. For the upper bound, we can proceed as in the proof of Proposition \ref{prop_tightness} (with $p=2$) to get that there exists $n_0 \geq 1$ and $c_1,c_2 > 0$ such that for any $n \geq n_0$,
\begin{equation} \label{upper_bound_Sinfty}
    \EE_{n}^{\pm}[S_{n}(\gamma_{n})^2] \leq \exp(6Mc_1g_{1}(\delta_n) + 20M^2c_2g_1(\delta_n)^2)
\end{equation}
where $M:= \max_{x \in \Omega} \vert h(x) \vert$. Observe that the right-hand side of \eqref{upper_bound_Sinfty} converges to $1$ as $n \to \infty$. On the other hand, by Jensen's inequality, $\EE_{n}^{\pm}[S_{n}(\gamma_{n})^2] \geq \mathcal{Z}_{n}(h)^2$. The same reasoning as in the proof of Proposition \ref{prop_tightness} now shows that there exist $n_1 \geq 1$ and $c > 0$ such that for any $n \geq n_1$, $\mathcal{Z}_{n}(h) \geq \exp(-cg_{1}(\delta_n))$. Together with the inequality \eqref{upper_bound_Sinfty}, this implies that $\EE_{n}^{\pm}[S_{n}(\gamma_{n})^2]$ converges to 1 as $n \to \infty$.

To prove \eqref{lim_RN_equal1}, it remains to show $\mathcal{Z}_{n}(h)$ converges to $1$ as $n \to \infty$. For this, we can once again apply the same reasoning as in the proof of Proposition \ref{prop_tightness} (with $p=1$) to get that for any $n \geq n_0$, $\mathcal{Z}_{n}(h)\leq \exp(4Mc_1g_{1}(\delta_n)+8M^2c_2g_1(\delta_n)^2)$. As we have already shown that for any $n \geq n_1$, $\mathcal{Z}_{n}(h) \geq \exp(-cg_1(\delta_n))$, this yields that $\mathcal{Z}_{n}(h) \to 1$ as $n \to \infty$. As explained above, this also completes the proof of Proposition \ref{prop_small_H}. 
\end{proof}

\subsection{The case of a large magnetic perturbation: proof of Proposition \ref{prop_large_H}} \label{sec_proof_large_H}

\subsubsection{The Edwards-Sokal coupling with non-zero magnetic field} \label{subsec_ES_coupling}

For $\delta > 0$, let $(\Omega_{\delta};a_{\delta},b_{\delta})_{\delta}$ be as in Section \ref{sec_assumptions_domain}. The Edwards-Sokal cooupling of Proposition \ref{prop_ES_coupling_no_field} can be extended to the case when the Ising model is perturbed by an external magnetic field, at least when the external field is non-negative and when the boundary conditions are chosen to be $+$ or free. This coupling will be instrumental for the proof of Proposition \ref{prop_large_H}, so let us describe it.

To construct the Edwards-Sokal coupling with an external field, one direct way to proceed is to consider the extended graph obtained from $(V(\Omega_{\delta}), E(\Omega_{\delta}))$ by adding a ghost vertex $g$ to the graph and drawing an edge $(vg)$ between each vertex $v \in V(\Omega_{\delta})$ and $g$. We denote the set of such edges $E^*(\Omega_{\delta})$ and call a function $\omega: E(\Omega_{\delta}) \cup E^*(\Omega_{\delta}) \to \{0,1\}^{\vert  E(\Omega_{\delta}) \vert + \vert  E^*(\Omega_{\delta}) \vert}$ an edge configuration. We say that the edge $e$ is open (in $\omega$) if $\omega(e)=1$. Otherwise, $e$ is said to be closed. We identify $\omega$ with the subgraph whose vertex set is $V(\Omega_{\delta}) \cup g$ and whose edge set is $\{e \in E(\Omega_{\delta}) \cup E^*(\Omega_{\delta}): \omega(e)=1\}$. A cluster of $\omega$ is a connected component of this subgraph. 

With these notations, the FK-Ising percolation measure with parameter $p$ and external field $H$ on $(V(\Omega_{\delta})\cup g, E(\Omega_{\delta})\cup E^*(\Omega_{\delta}))$ is defined as follows: for $\xi \in \{\operatorname{free}, \operatorname{wired}\}$ and $\omega$ an edge configuration,
\begin{equation*}
    \PP_{\delta,p,H}^{\operatorname{FK}, \xi}(\omega) = \frac{1}{\mathcal{Z}_{p,H}}2^{k(\omega^{\xi})}\prod_{e \in E(\Omega_{\delta})} p^{\omega(e)}(1-p)^{1-\omega(e)} \prod_{e^* \in E^*(\Omega_{\delta})} (1-e^{-2H})^{\omega(e^*)}(e^{-2H})^{1-\omega(e^*)}
\end{equation*}
where $\omega^{\xi}$ is the configuration obtained by identifying wired vertices and $k(\omega^{\xi})$ is the number of clusters of $\omega^{\xi}$ not connected to $g$. $\mathcal{Z}_{p,H}$ is a normalizing constant chosen such that $\PP_{\delta,p,H}^{\operatorname{FK}, \xi}$ is a probability measure. Note that when $H \equiv 0$, $\PP_{\delta,p,0}^{\operatorname{FK}, \xi}$ can be seen as a measure on edge configurations in $(V(\Omega_{\delta}), E(\Omega_{\delta}))$ and it is equivalent to the measure $\PP_{\delta,p}^{\operatorname{FK}, \xi}$ introduced in Section \ref{sec_ES_coupling_no_external_field}.

The Edwards-Sokal coupling with external field $H$ goes as follows \cite{EScoupling_magnetic, CamiaFKIsing, Camia_exponential_decay}. One first samples an edge configuration $\omega$ on $(V(\Omega_{\delta})\cup g, E(\Omega_{\delta})\cup E^*(\Omega_{\delta}))$ according to $\PP_{\delta,p,H}^{\operatorname{FK}, \xi}$ with $p=1-e^{-2\beta}$, where $\beta$ is the inverse temperature of the Ising model. Then, all clusters of $\omega$ connected to $g$, and also the cluster of $\omega$ connected to $\partial \Omega_{\delta}$ if $\xi = \operatorname{wired}$, are assigned the spin $+1$. The spins of the remaining clusters, i.e. the clusters not connected to $g$, are chosen by tossing independent fair coins, one for each cluster. This gives rise to a spin configuration in $\Omega_{\delta}$, where each vertex has the spin of its cluster. This spin configuration has the same distribution as the Ising model in $\Omega_{\delta}$ with inverse temperature $\beta$, external field $H$ and boundary conditions $\xi$.

From now on, we will set $\beta = \beta_{c}$ and for $H \geq 0$ and $\xi \in \{\operatorname{free}, \operatorname{wired} \}$, we denote by $\PP_{\delta,H}^{\operatorname{FK}, \xi}$ the corresponding FK-Ising percolation measure with external field $H$. For the proof of Proposition \ref{prop_large_H}, it will be useful to have a version of the Edwards-Sokal coupling with external field in which one first samples an edge configuration $\omega$ on $(V(\Omega_{\delta}), E(\Omega_{\delta}))$ and then, conditionally on $\omega$, samples the set of clusters that are connected to $g$. This procedure does not sample an edge configuration on $(V(\Omega_{\delta})\cup g, E(\Omega_{\delta})\cup E^*(\Omega_{\delta}))$. However, given the edge configuration $\omega$ on $(V(\Omega_{\delta}), E(\Omega_{\delta}))$ and the set of clusters of $\omega$ connected to $g$, one can still sample the sign of the clusters in such a way that the resulting spin configuration in $\Omega_{\delta}$ has the same distribution as a critical Ising spin configuration in $\Omega_{\delta}$ with external field $H$ and boundary conditions $\xi$.

We summarize this alternative procedure in the following proposition, which is \cite[Proposition~1]{Camia_exponential_decay} and \cite[Lemma~4]{Camia_exponential_decay} and appears in an implicit form in \cite{EScoupling_magnetic}. We state these results for free boundary conditions as these are the boundary conditions that we will use in the proof of Proposition \ref{prop_large_H}. Below, for $\omega: E(\Omega_{\delta}) \to \{0,1\}^{\vert E(\Omega_{\delta}) \vert}$, we denote by $\mathcal{C}_{\delta}(\omega)$ the set of clusters of $\omega$ and for $\mathcal{C} \in \mathcal{C}_{\delta}(\omega)$, $\vert \mathcal{C} \vert$ is the number of vertices in $\mathcal{C}$. The event that a cluster $\mathcal{C} \in \mathcal{C}_{\delta}(\omega)$ is connected to the ghost vertex $g$ is denoted by $\{\mathcal{C} \longleftrightarrow g \}$.

\begin{proposition}\label{prop_ES_magnetic}
Let $\hat \PP^{\operatorname{free}}_{\delta, H}$ denote the law of the Edwards-Sokal coupling with external field $H \geq 0$. For $\omega: E(\Omega_{\delta}) \to \{0,1\}^{\vert E(\Omega_{\delta}) \vert}$ and $\mathcal{C} \in \mathcal{C}_{\delta}(\omega)$, we let $\sigma(\mathcal{C})$ denote the spin of $\mathcal{C}$ assigned by the coupling. We then have that
\begin{equation*}
    \PP_{\delta, H}^{\operatorname{FK}, \operatorname{free}}(\mathcal{C} \longleftrightarrow g \vert \omega) = \tanh(H \vert \mathcal{C} \vert), \quad \hat \PP_{\delta,H}^{\operatorname{free}}(\sigma(\mathcal{C})=+1 \vert \omega)= \frac{1}{2}+ \frac{1}{2}\tanh(H \vert \mathcal{C} \vert).
\end{equation*}
Moreover, conditionally on $\omega$, the events $(\{\mathcal{C} \longleftrightarrow g\})_{\mathcal{C} \in \mathcal{C}_{\delta}(\omega)}$ are mutually independent. Similarly, conditionally on $\omega$, the events $(\{\sigma(\mathcal{C})=+1\})_{\mathcal{C} \in \mathcal{C}_{\delta}(\omega)}$ are mutually independent. The Radon-Nikodym derivative of $\mathbb{Q}_{\delta,H}^{\operatorname{free}}$, the marginal law of $\PP_{\delta,H}^{\operatorname{FK}, \operatorname{free}}$ on edge configurations $\omega: E(\Omega_{\delta}) \to \{0,1\}^{\vert E(\Omega_{\delta}) \vert}$, with respect to $\PP_{\delta}^{\operatorname{FK}, \operatorname{free}}$ is given by
\begin{equation*}
    \frac{\der \mathbb{Q}_{\delta,H}^{\operatorname{free}}}{\der \PP_{\delta}^{\operatorname{FK}, \operatorname{free}}}(\omega) = \frac{1}{\mathcal{Z}_{\delta,H}} \prod_{\mathcal{C} \in \mathcal{C}_{\delta}(\omega)} \cosh(H \vert \mathcal{C} \vert)
\end{equation*}
where $\mathcal{Z}_{\delta,H}$ is a normalization constant.
\end{proposition}

Let us also record two stochastic domination results for the measures $\PP_{\delta,H}^{\operatorname{FK}, \operatorname{free}}$, $\QQ_{\delta,H}^{\operatorname{free}}$ and $\PP_{\delta}^{\operatorname{FK}, \operatorname{free}}$. They are established in \cite[Lemma~2]{Camia_exponential_decay} and \cite[Theorem~6]{EScoupling_magnetic}.

\begin{lemma} \label{lemma_stochastic_domination}
For any $H \geq 0$, $\PP_{\delta,H}^{\operatorname{FK}, \operatorname{free}}$ stochastically dominates $\PP_{\delta}^{\operatorname{FK}, \operatorname{free}}$. Similarly, for any $H \geq 0$, $\QQ_{\delta,H}^{\operatorname{free}}$ stochastically dominates $\PP_{\delta}^{\operatorname{FK}, \operatorname{free}}$.
\end{lemma}

\subsubsection{Proof of Proposition \ref{prop_large_H}}

In this section, we prove Proposition \ref{prop_large_H}. We give the proof in the case $h>0$ but it could easily be adapted to handle the case of a non-constant perturbation $h:x\mapsto h(x)$, as long as $h$ is lower bounded by a positive constant.

\begin{proof}[Proof of Proposition \ref{prop_large_H}]
Let $\eta > 0$. For $\delta > 0$, we set $\Omega_{\delta}^{-}(\eta):=\{ x \in \Omega_{\delta}: d(x,\partial \Omega_{\delta}^{-}) \leq \eta \}$. We are going to show that
\begin{equation} \label{lim_proba_intersection}
    \lim_{\delta \to 0} \PP_{\delta,h}^{\pm}[\gamma_{\delta} \cap (\Omega_{\delta} \setminus \Omega_{\delta}^{-}(\eta)) \neq \emptyset] = 0.
\end{equation}
The first step is to turn the Dobrushin boundary conditions on $\partial \Omega_{\delta}$ into $+1$ boundary conditions, that are easier to handle. For this, let us start by briefly recalling the results of \cite[Section~5.1]{armExponentsIsing}. A spin configuration in $\Omega_{\delta}$ gives rise to an edge configuration in $\Omega_{\delta}$: an edge is open if its two endvertices have the same spin, closed otherwise. Let $E$ be an event measurable with respect to the edge configuration in $\Omega_{\delta} \setminus \Omega_{\delta}^{-}(\eta/2)$. As $E$ is measurable with respect to edges at macroscopic distance from $\partial \Omega_{\delta}^{-}$, one can expect that the boundary conditions on $\partial \Omega_{\delta}^{-}$ do not really influence the probability of $E$: this phenomena is called spatial mixing. More precisely, \cite[Section~5.1]{armExponentsIsing} shows the following. Define $a(\eta/2)$, respectively $b(\eta/2)$, as the intersection point of $\partial B(a,\eta/2)$, respectively $\partial B(b,\eta/2)$, with $\partial \Omega^{+}$. If this intersection point is not unique, then choose the first one encountered when tracing $\partial \Omega^{+}$ starting from $a$, respectively $b$. For $\delta > 0$, let $a_{\delta}(\eta/2)$, respectively $b_{\delta}(\eta/2)$, be the discrete approximation of $a(\eta/2)$, respectively $b(\eta/2)$. Denote by $\{ (b_{\delta}b_{\delta}(\eta/2)) \overset{+}{\longleftrightarrow} (a_{\delta}(\eta/2)a_{\delta}) \operatorname{in} \Omega_{\delta}^{-}(\eta/2) \}$ the event that the topological rectangle $\Omega_{\delta}^{-}(\eta/2)$ is crossed by a path of $+1$ spins staying in $\Omega_{\delta}^{-}(\eta/2)$ and going from the boundary arc $(b_{\delta}b_{\delta}(\eta/2))$ to the boundary arc $(a_{\delta}(\eta/2)a_{\delta})$. Here, the boundary arcs are oriented counter-clockwise. The results of \cite[Section~5.1]{armExponentsIsing} show that if there exists $c>0$ and $\delta_0 >0$ such that for any $\delta>0$,
\begin{equation} \label{lower_bound_crossing}
    \PP_{\delta,h}^{\pm}[(b_{\delta}b_{\delta}(\eta/2)) \overset{+}{\longleftrightarrow} (a_{\delta}(\eta/2)a_{\delta}) \operatorname{in} \Omega_{\delta}^{-}(\eta/2) ] \geq c,
\end{equation}
then there exists $C>0$ such that for any $0<\delta < \delta_0$, $\PP_{\delta,h}^{\pm}[E]\leq C\PP_{\delta,h}^{+}[E]$.

Below, we will apply this result with the event $E=\mathcal{C}^{\pm}(\eta/2,\eta)$, where $\mathcal{C}^{\pm}(\eta/2,\eta)$ is defined as follows. Define $a(\eta)$, $b(\eta)$, $a_{\delta}(\eta)$ and $b_{\delta}(\eta)$ as $a(\eta/2)$, $b(\eta/2)$, $a_{\delta}(\eta/2)$ and $b_{\delta}(\eta/2)$ but with $\eta/2$ replaced by $\eta$. Consider the topological rectangle $(R_{\delta}; a_{\delta}(\eta/2), b_{\delta}(\eta/2), b_{\delta}(\eta),b_{\delta}(\eta))$, where the four marked points are in counterclockwise order and $R_{\delta}:= \Omega_{\delta}^{-}(\eta) \setminus \Omega_{\delta}^{-}(\eta/2)$. With these notations, the event $\mathcal{C}^{\pm}(\eta/2,\eta)$ is the event that there exists a path $\tilde \gamma$ on the dual of $\Omega_{\delta}$ going from the boundary arc $(a_{\delta}(\eta/2)b_{\delta}(\eta/2))$ to the boundary arc $(a_{\delta}(\eta)b_{\delta}(\eta))$ such the following is fulfilled: if $x,y \in \Omega_{\delta} \cup \partial \Omega_{\delta}$ are adjacent to $\tilde \gamma$, then $\sigma_x = \sigma_y$ if $x$ and $y$ are on the same side of $\tilde \gamma$, and $\sigma_x = -\sigma_y$ if $x$ and $y$ are on different sides of $\tilde \gamma$. Here, we say that a vertex is adjacent to $\tilde \gamma$ if it is the endpoint of a primal edge crossed by $\tilde \gamma$. Moreover, as before, the boundary arcs $(a_{\delta}(\eta/2)b_{\delta}(\eta/2))$ and $(a_{\delta}(\eta)b_{\delta}(\eta))$ are oriented counter-clockwise. Observe now that the event $\mathcal{C}^{\pm}(\eta/2,\eta)$ is measurable with respect to the edge configuration in $\Omega_{\delta} \setminus \Omega_{\delta}^{-}(\eta/2)$ and that
\begin{equation} \label{ineq_intersection_path_pm}
    \PP_{\delta,h}^{\pm}[\gamma_{\delta} \cap (\Omega_{\delta} \setminus \Omega_{\delta}^{-}(\eta)) \neq \emptyset] \leq \PP_{\delta,h}^{\pm}[\mathcal{C}^{\pm}(\eta/2,\eta)].
\end{equation}
As explained above, to change the boundary conditions and show that there exist $C>0$ and $\delta_0 > 0$ such that for any $0 < \delta < \delta_0$,
\begin{equation} \label{ineq_bdry_cond_path_pm}
    \PP_{\delta,h}^{\pm}[\mathcal{C}^{\pm}(\eta/2,\eta)] \leq C\PP_{\delta,h}^{+}[\mathcal{C}^{\pm}(\eta/2,\eta)],
\end{equation}
we must show that the inequality \eqref{lower_bound_crossing} is satisfied. Here, we recall that under $\PP_{\delta,h}^{+}$, the boundary conditions of the Ising model in $\Omega_{\delta}$ are $+1$. Using the Markov property of the Ising model and the fact that $\{ (b_{\delta}b_{\delta}(\eta/2)) \overset{+}{\longleftrightarrow} (a_{\delta}(\eta/2)a_{\delta}) \operatorname{in} \Omega_{\delta}^{-}(\eta/2)\}$ is an increasing event, we have that
\begin{align*}
    \PP_{\delta,h}^{\pm}\big[(b_{\delta}b_{\delta}(\eta/2)) \overset{+}{\longleftrightarrow} (a_{\delta}(\eta/2)a_{\delta}) \operatorname{in} \Omega_{\delta}^{-}(\eta/2) \big] &\geq \PP_{\Omega_{\delta}^{-}(\eta/2),h}^{\pm}\big[ (b_{\delta}b_{\delta}(\eta/2)) \overset{+}{\longleftrightarrow} (a_{\delta}(\eta/2)a_{\delta})\big] \\
    &\geq \PP_{\Omega_{\delta}^{-}(\eta/2)}^{\pm}\big[ (b_{\delta}b_{\delta}(\eta/2)) \overset{+}{\longleftrightarrow} (a_{\delta}(\eta/2)a_{\delta}) \big]
\end{align*}
where $\PP_{\Omega_{\delta}^{-}(\eta/2)}^{\pm}$ is the probability measure on Ising spin configurations in $\Omega_{\delta}^{-}(\eta/2)$ with $-1$ boundary conditions on $\partial \Omega_{\delta}^{-}$ and the boundary arc $(b_{\delta}(\eta/2)a_{\delta}(\eta/2))$ and $+1$ boundary conditions on the boundary arcs $(a_{\delta}(\eta/2)a_{\delta})$ and $(b_{\delta}b_{\delta}(\eta/2))$. By \cite[Corollary~1.7]{crossing-rectangles}, there exist $c=c(\eta)>0$ and $\delta_0 > 0$ such that for any $0 < \delta < \delta_0$,
\begin{equation*}
    \PP_{\Omega_{\delta}(\eta/2)}^{\pm}\big[(b_{\delta}b_{\delta}(\eta/2)) \overset{+}{\longleftrightarrow} (a_{\delta}(\eta/2)a_{\delta})\big] \geq c,
\end{equation*}
which yields \eqref{ineq_bdry_cond_path_pm}.

With the inequality \eqref{ineq_bdry_cond_path_pm} at hands, the next step of the proof is to replace $\mathcal{C}^{\pm}(\eta/2,\eta)$ by an event whose probability is easier to estimate. To this end, define $\{ (a_{\delta}(\eta)a_{\delta}(\eta/2)) \overset{+}{\longleftrightarrow} (b_{\delta}(\eta/2)b_{\delta}(\eta)) \operatorname{in} R_{\delta}\}$ as the event that there exits a crossing of $R_{\delta}$ by a path of $+1$ spins going from the boundary arc $(a_{\delta}(\eta)a_{\delta}(\eta/2))$ to the boundary arc $(b_{\delta}(\eta/2)b_{\delta}(\eta))$ and staying in $R_{\delta}$. Observe that if $\{ (a_{\delta}(\eta)a_{\delta}(\eta/2)) \overset{+}{\longleftrightarrow} (b_{\delta}(\eta/2)b_{\delta}(\eta)) \operatorname{in} R_{\delta}\}$ happens, then $\mathcal{C}^{\pm}(\eta/2,\eta)$ cannot happen. Therefore,
\begin{equation} \label{ineq_path_pm_long_plus_path}
    \PP_{\delta,h}^{+}[\mathcal{C}^{\pm}(\eta/2,\eta)] \leq 1 - \PP_{\delta,h}^{+}\big[ (a_{\delta}(\eta)a_{\delta}(\eta/2)) \overset{+}{\longleftrightarrow} (b_{\delta}(\eta/2)b_{\delta}(\eta)) \operatorname{in} R_{\delta} \big].
\end{equation}
We are now going to show that the $\liminf$ as $\delta \to 0$ of the probability on the right-hand side above is equal to 1. For this, we will use the Edwards-Sokal coupling for the Ising model with an external field, see Section \ref{subsec_ES_coupling}. For $K \in \mathbb{N}$, we let $\mathcal{C}_{\delta}(K)$ be the event that there exists a sequence $\mathcal{C}_1, \dots , \mathcal{C}_k$ of open clusters in $\Omega_{\delta}$ such that $k \leq K$, $\vert \mathcal{C}_i \vert \geq \delta^{-15/8}/K$ for each $i$ and one can find a path $\tilde \gamma$ staying in $R_{\delta}$ that joins $(a_{\delta}(\eta)a_{\delta}(\eta/2))$ to $(b_{\delta}(\eta/2)b_{\delta}(\eta))$ such that the edges traversed by $\tilde \gamma$ are either open edges of some $\mathcal{C}_{i}$ or closed edges of the form $(xy)$ with $x \in \mathcal{C}_i$, $y \in \mathcal{C}_j$ and $i \neq j$. Here, we stress that connections are consider in $(V(\Omega_{\delta}) \cup \partial \Omega_{\delta}, E(\Omega_{\delta}))$. We further define the event $\mathcal{C}_{\delta}(g;K)$ as $\mathcal{C}_{\delta}(K)$ but we additionally require that all clusters $(\mathcal{C}_i)_i$ visited by the path $\tilde \gamma$ are connected to the ghost vertex $g$.

In the Edwards-Sokal coupling with an external field, on the event $\mathcal{C}_{\delta}(g;K)$, the clusters visited by $\tilde \gamma$ all have spin $+1$. Thus, we have that $\mathcal{C}_{\delta}(g;K) \subset \{ (a_{\delta}(\eta)a_{\delta}(\eta/2)) \overset{+}{\longleftrightarrow} (b_{\delta}(\eta/2)b_{\delta}(\eta)) \operatorname{in} R_{\delta} \}$. We note that on the event $\mathcal{C}_{\delta}(g;K)$, no restriction is imposed on the edges of $\tilde \gamma$ joining the boundary arcs $(a_{\delta}(\eta)a_{\delta}(\eta/2))$ and $(b_{\delta}(\eta/2)b_{\delta}(\eta))$ to some vertex in some $\mathcal{C}_i$: these edges may be open or closed. This is because under $\PP_{\delta,h}^{+}$, the Ising model has $+1$ boundary conditions and therefore, the spins on $(a_{\delta}(\eta)a_{\delta}(\eta/2))$ and  $(b_{\delta}(\eta/2)b_{\delta}(\eta))$ are $+1$.

Since $\mathcal{C}_{\delta}(g;K)$ is an increasing event, by Proposition \ref{prop_ES_magnetic}, we obtain that
\begin{align} \label{ineq_C_delta_ghost}
    \PP_{\delta,h}^{+}\big[ (a_{\delta}(\eta)a_{\delta}(\eta/2)) \overset{+}{\longleftrightarrow} (b_{\delta}(\eta/2)b_{\delta}(\eta)) \operatorname{in} R_{\delta} \big] &\geq \hat \PP_{\delta,H}^{\operatorname{wired}}[\mathcal{C}_{\delta}(g;K)] \nonumber \\
    &\geq \hat \PP_{\delta,H}^{\operatorname{free}}[\mathcal{C}_{\delta}(g;K)] = \hat \PP_{\delta,H}^{\operatorname{free}}[\mathcal{C}_{\delta}(g;K) \vert \mathcal{C}_{\delta}(K)] \hat \PP_{\delta, H}^{\operatorname{free}}[\mathcal{C}_{\delta}(K)].
\end{align}
Using that $\tanh(x) \leq 1$ and $x \mapsto \tanh(x)$ is increasing, Proposition \ref{prop_ES_magnetic} yields that
\begin{equation} \label{ineq_cond_C_delta_ghost}
    \hat \PP_{\delta,H}^{\operatorname{free}}[] \mathcal{C}_{\delta}(g;K) \vert \mathcal{C}_{\delta}(K)] = \hat \EE_{\delta,H}^{\operatorname{free}}\big[ \prod_{\mathcal{C}_i \cap \tilde \gamma \neq \emptyset} \tanh(H(\delta)\vert \, \mathcal{C}_i \vert) \vert \mathcal{C}_{\delta}(K) \big] \geq \bigg[\tanh\bigg( H(\delta) \frac{\delta^{-\frac{15}{8}}}{K}\bigg)\bigg]^{K} = \bigg[\tanh \bigg( \frac{g_{2}(\delta)}{K}\bigg)\bigg]^K.
\end{equation}
On the other hand, by stochastic domination, see Lemma \ref{lemma_stochastic_domination}, $\PP_{\delta,H}^{\operatorname{free}}[\mathcal{C}_{\delta}(K)] \geq \PP_{\delta}^{\operatorname{FK}, \operatorname{free}}[\mathcal{C}_{\delta}(K)]$. Since $\mathcal{C}_{\delta}(K)$ is an increasing event, we then arrive at
\begin{equation} \label{ineq_C_delta_K}
    \PP_{\delta,H}^{\operatorname{free}}[\mathcal{C}_{\delta}(K)] \geq \PP_{R_{\delta}}^{\operatorname{FK}, \operatorname{free}}[\mathcal{C}_{\delta}(K)]
\end{equation}
where $\PP_{R_{\delta}}^{\operatorname{FK}, \operatorname{free}}$ is the probability measure on FK-Ising configurations in $R_{\delta}$ with free boundary conditions. We now claim the following.

\begin{claim} \label{claim_proba_path_clusters}
    Let $\eps > 0$. There exists $K=K(\eps)<\infty$ such that $\liminf_{\delta \to 0} \PP_{R_{\delta}}^{\operatorname{FK}, \operatorname{free}}[\mathcal{C}_{\delta}(K)] > 1-\eps$.
\end{claim}

We postpone the proof of Claim \ref{claim_proba_path_clusters} to the end and first explain how to conclude the proof of the proposition. Let $\eps > 0$. Claim \ref{claim_proba_path_clusters} together with the inequalities \eqref{ineq_C_delta_ghost}, \eqref{ineq_cond_C_delta_ghost} and \eqref{ineq_C_delta_K} yield that there exists $K > 0$ such that
\begin{equation*}
    \liminf_{\delta \to 0} \PP_{\delta,h}^{+}\big[ (a_{\delta}(\eta)a_{\delta}(\eta/2)) \overset{+}{\longleftrightarrow} (b_{\delta}(\eta/2)b_{\delta}(\eta)) \operatorname{in} R_{\delta} \big] > (1-\eps) \liminf_{\delta \to 0} \bigg[\tanh \bigg( \frac{g_{2}(\delta)}{K}\bigg)\bigg]^K.
\end{equation*}
Since $g_{2}(\delta) \to \infty$ as $\delta \to 0$, the $\liminf$ on the right-hand side above is equal to $1$. By choosing $\eps >0$ small enough, this implies that for any $\rho >0$, there exists $0<K<\infty$ such that $\liminf_{\delta \to 0}  \PP_{\delta,h}^{+}\big[ (a_{\delta}(\eta)a_{\delta}(\eta/2)) \overset{+}{\longleftrightarrow} (b_{\delta}(\eta/2)b_{\delta}(\eta)) \operatorname{in} R_{\delta} \big] > 1- \rho$. Going back to \eqref{ineq_bdry_cond_path_pm} and \eqref{ineq_path_pm_long_plus_path}, this shows that for any $\rho > 0$, 
\begin{equation*}
    \limsup_{\delta \to 0} \PP_{\delta,h}^{\pm}[\mathcal{C}^{\pm}(\eta/2,\eta)] < C\rho,
\end{equation*}
which, by \eqref{ineq_intersection_path_pm}, establishes \eqref{lim_proba_intersection} and concludes the proof.
\end{proof}

Let us now turn to the proof of Claim \ref{claim_proba_path_clusters}.

\begin{proof}[Proof of Claim \ref{claim_proba_path_clusters}]
The proof is similar to the proofs of \cite[Lemma~5]{Camia_exponential_decay} and \cite[Lemma~7]{Camia_exponential_decay}. Indeed, under $\PP_{R_{\delta}}^{\operatorname{FK},\operatorname{free}}$, the event $\mathcal{C}_{\delta}(K)$ is the same as the event $E^a(K,N)$ with $a=\delta$ and $N=1$ of \cite[Lemma~7]{Camia_exponential_decay}. This is because under $\PP_{R_{\delta}}^{\operatorname{FK},\operatorname{free}}$, on the event $\mathcal{C}_{\delta}(g;K)$, the path $\tilde \gamma$ automatically stays in $R_{\delta}$. The proofs of \cite[Lemma~5]{Camia_exponential_decay} and \cite[Lemma~7]{Camia_exponential_decay} are given when the domain is a rectangle of a fixed aspect ratio. However, one can observe that it is easy to extend the arguments to the case of an arbitrary bounded and simply connected domain $D$ and its discrete approximation $D_{\delta}$, as long as one considers boundary arcs of $D$ that have positive harmonic measure seen from a point in the interior of the domain. The only restriction that must be imposed is that $\partial D$ has Minkowski dimension $1$: if $\partial D$ is too rough, one may not be able to rule out that an arm event of type $010$ takes place near $\partial D_{\delta}$ in the limit $\delta \to 0$. This point is crucial to deduce results for $\PP_{R_{\delta}}^{\operatorname{FK}, \operatorname{free}}(\mathcal{C}_{\delta}(K))$ from results obtained for CLE$_{16/3}$ and CME$_{16/3}$ in the continuum. We leave the details to the reader.
\end{proof}

\subsection{On the $h \to \infty$ limit of $\PP_{h}^{(\Omega,a,b)}$: proof of Proposition \ref{prop_h_limit}} \label{sec_h_limit}

\begin{proof}[Proof of Proposition \ref{prop_h_limit}]
For $h \in \mathbb{R}$ and $\eta > 0$, set $p(h,\eta):= \PP_{h}^{(\Omega,a,b)}[\gamma \cap (\Omega \setminus \Omega^{-}(\eta)) \neq \emptyset]$. To prove the proposition, we argue by contradiction. Assume that there exists $\eta > 0$ such that $p(h,\eta)$ does not converge to $0$ as $h \to \infty$. Then, there exist $\eps > 0$ and a sequence $(h_n)_n$ such that $\lim_{n \to \infty} h_n = \infty$ and for any $n \geq 1$, $p(h_n,\eta) \geq \eps$. Moreover, with $H(n,\delta)=C_{\sigma}^{-1}h_n\delta^{\frac{15}{8}}$, Theorem \ref{theorem_magnetic_SLE3}, when considering convergence in the topology \ref{topo_4}, implies that for any $n \geq 1$,
\begin{equation} \label{probab_greater_eps}
    \lim_{\delta \to 0} \PP_{\delta,H(n,\delta)}^{\pm}[\gamma_{\delta} \cap (\Omega_{\delta} \setminus \Omega_{\delta}^{-}(\eta)) \neq \emptyset] = p(h_n,\eta) \geq \eps
\end{equation}
where we have set $\Omega_{\delta}^{-}(\eta):= \{ x \in \Omega_{\delta}: \operatorname{dist}(x,\partial \Omega_{\delta}) \leq \eta \}$. Now, let $(\delta_n)_n$ be a sequence converging to $0$ such that for any $n \geq 1$,
\begin{equation} \label{proba_n_geq_eps}
    \PP_{\delta_n, H(n,\delta_n)}^{\pm}[\gamma_{\delta_n} \cap (\Omega_{\delta_n} \setminus \Omega_{\delta_n}^{-}(\eta)) \neq \emptyset] \geq \eps.
\end{equation}
Note that such a sequence $(\delta_n)_n$ exists by \eqref{probab_greater_eps}. As $h_n$ tends to $\infty$, we can write $h_n=g(\delta_n)$ with $g(\delta_n) \to \infty$ as $n \to \infty$. Therefore, by Proposition \ref{prop_large_H}, we obtain that
\begin{equation*}
    \lim_{n \to \infty} \PP_{n,g(\delta_n)}^{\pm}[\gamma_{\delta_n} \cap (\Omega_{\delta_n} \setminus \Omega_{\delta_n}^{-}(\eta)) \neq \emptyset] = 0.
\end{equation*}
This contradicts \eqref{proba_n_geq_eps} and completes the proof of Proposition \ref{prop_h_limit}.
\end{proof}

\section{Appendix} \label{sec_appendix}

\begin{lemma} \label{lemma_fineteness_moments_plusBC}
Let $D \subset \mathbb{C}$ be a bounded, open and simply connected domain such that $\partial D$ has Hausdorff dimension strictly smaller than $7/4$. Then, for any $k \geq 1$,
\begin{equation*}
    \int_{D^k} \vert f_{D}^{(+,k)}(x_1,\dots,x_k) \vert^2 \prod_{j=1}^{k}dx_j < \infty
\end{equation*}
where $f_{D}^{(+,k)}$ is as in Theorem \ref{theorem_spin_correlations}.
\end{lemma}

\begin{proof}
To show this lemma, we use the explicit expression of $f_{D}^{(+,k)}$: by \cite[Theorem~7.1]{spinCorrelations}, we have that, for any $k \geq 1$ and any $x_1, \dots, x_k \in D^k$,
\begin{equation*}
    f_{D}^{(+,k)}(x_1, \dots, x_k) = \prod_{j=1}^{k}\CR(x_j,\partial D)^{-\frac{1}{8}} \bigg( 2^{-\frac{k}{2}} \sum_{\mu \in \{-1,1\}^k} \prod_{1 \leq j < m \leq k}\exp\big(\frac{\mu_i\mu_m}{2}G_{D}(x_j,x_m)\big)\bigg)^{\frac{1}{2}}
\end{equation*}
where $G_D$ denotes the Green function in $D$ with Dirichlet boundary conditions. Here, $G_D$ is normalized so that as $y \to x$, $G_{D}(x,y) \sim -\log(\vert x - y \vert)$. Let $\alpha, \tilde \alpha > 1$ be such that $\frac{1}{\alpha}+\frac{1}{\tilde \alpha}=1$ and $\frac{\tilde \alpha}{4} < 2- \dim_H(\partial D)$ and $\alpha < 2$, where $\dim_{H}(\partial D)$ denotes the Hausdorff dimension of $\partial D$. Applying H\"older inequality with the functions
\begin{equation*}
    g_1(x_1,\dots,x_k) = \prod_{j=1}^k \CR(x_j,\partial D)^{-\frac{1}{4}}, \quad g_{2}(x_1,\dots,x_k) = \sum_{\mu \in \{-1,1\}^k} \prod_{1 \leq j < m \leq k}\exp\big(\frac{\mu_i\mu_m}{2}G_{D}(x_j,x_m)\big),
\end{equation*}
we obtain that for any $k \geq 1$,
\begin{equation*}
    \int_{D^k} \vert f_{\Omega}^{(+,k)}(x_1,\dots,x_k) \vert^2 \prod_{j=1}^{k} dx_j \leq \bigg(\int_{D} \CR(x,\partial D)^{-\frac{\tilde \alpha}{4}} dx \bigg)^{\frac{k}{\tilde \alpha}}\bigg(\int_{D^k} g_{2}(x_1,\dots,x_k)^{\alpha} \prod_{j=1}^{k}dx_{j} \bigg)^{\frac{1}{\alpha}}.
\end{equation*}
By our assumption on $d_{H}(\partial D)$ and our choice of $\tilde \alpha$, it is easy to see that the first integral on the above right-hand side is indeed finite. On the other hand, since $G_D$ is a non-negative function, we have that
\begin{equation*}
    \int_{D^k} g_{2}(x_1,\dots,x_k)^{\alpha} \prod_{j=1}^{k}dx_j \leq 2^k \int_{D^k} \prod_{1 \leq j < m \leq k}\exp\big(\frac{\alpha}{4}G_{D}(x_j,x_m)\big) \prod_{j=1}^{k}dx_{j}.
\end{equation*}
Finiteness of the above integral then follows by combining \cite[Proposition~3.9]{Junnila} and \cite[Lemma~3.10]{Junnila}.
\end{proof}

\bibliography{interface-magnetic-Ising-arxiv}

\end{document}